\documentclass[a4paper,12pt,reqno]{amsart}
\usepackage{graphics}\usepackage{graphicx}
\usepackage{pstricks}
\usepackage[colorlinks=true]{hyperref}
\usepackage{MnSymbol}
\usepackage[a4paper,margin=1.9cm,top=2.5cm,bottom=2.5cm,centering,vcentering]{geometry}

\allowdisplaybreaks[1]

\newtheorem{theorem}{Theorem}
\newtheorem{corollary}[theorem]{Corollary}
\newtheorem{proposition}[theorem]{Proposition}

\newcommand{\ASM}{\mathrm{ASM}}
\newcommand{\DASASM}{\mathrm{DASASM}}
\newcommand{\SSYT}{\mathrm{SSYT}}

\definecolor{brwn}{rgb}{0.59,0.29,0.0}

\renewcommand{\ss}{\scriptstyle}
\newcommand{\W}[5]{W\left(\raisebox{-3mm}{\psset{unit=4.6mm}\pspicture(-0.4,-0.4)(1.5,1.5)
\rput(0.5,0.5){\psline[linewidth=0.5pt](0,-0.5)(0,0.5)\psline[linewidth=0.5pt](-0.5,0)(0.5,0)}
\rput(0.5,0.5){$\ss\bullet$}\rput[r](-0.1,0.5){$\ss#1$}\rput[t](0.5,-0.1){$\ss#2$}\rput[l](1.1,0.5){$\ss#3$}\rput[b](0.5,1.1){$\ss#4$}
\endpspicture},#5\right)}
\newcommand{\WL}[3]{W\bigl(\raisebox{-1.1mm}{\psset{unit=4.6mm}\pspicture(-0.25,-0.2)(1,0.9)
\psline[linewidth=0.5pt](0,0.5)(0,0)(0.5,0)\rput(0,0){$\ss\bullet$}\rput[b](0,0.6){$\ss#1$}\rput[l](0.6,0){$\ss#2$}
\endpspicture},#3\bigr)}
\newcommand{\WR}[3]{W\bigl(\raisebox{-1.1mm}{\psset{unit=4.6mm}\pspicture(-0.4,-0.2)(0.8,0.9)
\psline[linewidth=0.5pt](0,0)(0.5,0)(0.5,0.5)\rput(0.5,0){$\ss\bullet$}\rput[r](-0.1,0){$\ss#1$}\rput[b](0.5,0.6){$\ss#2$}
\endpspicture},#3\bigr)}
\newcommand{\Wi}{\raisebox{-2mm}{\psset{unit=5.8mm}\pspicture(0.1,0)(1,1)
\rput(0.5,0.5){\psline[linewidth=0.5pt](0,-0.5)(0,0.5)\psline[linewidth=0.5pt](-0.5,0)(0.5,0)}
\rput(0.5,0.5){$\ss\bullet$}
\psdots[dotstyle=triangle*,dotscale=0.9](0.5,0.8)\psdots[dotstyle=triangle*,dotscale=0.9](0.5,0.2)
\psdots[dotstyle=triangle*,dotscale=0.9,dotangle=-90](0.8,0.5)\psdots[dotstyle=triangle*,dotscale=0.9,dotangle=-90](0.2,0.5)
\endpspicture}}
\newcommand{\Wii}{\raisebox{-2mm}{\psset{unit=5.8mm}\pspicture(0.1,0)(1,1)
\rput(0.5,0.5){\psline[linewidth=0.5pt](0,-0.5)(0,0.5)\psline[linewidth=0.5pt](-0.5,0)(0.5,0)}
\rput(0.5,0.5){$\ss\bullet$}
\psdots[dotstyle=triangle*,dotscale=0.9,dotangle=180](0.5,0.8)\psdots[dotstyle=triangle*,dotscale=0.9,dotangle=180](0.5,0.2)
\psdots[dotstyle=triangle*,dotscale=0.9,dotangle=90](0.8,0.5)\psdots[dotstyle=triangle*,dotscale=0.9,dotangle=90](0.2,0.5)
\endpspicture}}
\newcommand{\Wiii}{\raisebox{-2mm}{\psset{unit=5.8mm}\pspicture(0.1,0)(1,1)
\rput(0.5,0.5){\psline[linewidth=0.5pt](0,-0.5)(0,0.5)\psline[linewidth=0.5pt](-0.5,0)(0.5,0)}
\rput(0.5,0.5){$\ss\bullet$}
\psdots[dotstyle=triangle*,dotscale=0.9](0.5,0.8)\psdots[dotstyle=triangle*,dotscale=0.9](0.5,0.2)
\psdots[dotstyle=triangle*,dotscale=0.9,dotangle=90](0.8,0.5)\psdots[dotstyle=triangle*,dotscale=0.9,dotangle=90](0.2,0.5)
\endpspicture}}
\newcommand{\Wiv}{\raisebox{-2mm}{\psset{unit=5.8mm}\pspicture(0.1,0)(1,1)
\rput(0.5,0.5){\psline[linewidth=0.5pt](0,-0.5)(0,0.5)\psline[linewidth=0.5pt](-0.5,0)(0.5,0)}
\rput(0.5,0.5){$\ss\bullet$}
\psdots[dotstyle=triangle*,dotscale=0.9,dotangle=180](0.5,0.8)\psdots[dotstyle=triangle*,dotscale=0.9,dotangle=180](0.5,0.2)
\psdots[dotstyle=triangle*,dotscale=0.9,dotangle=-90](0.8,0.5)\psdots[dotstyle=triangle*,dotscale=0.9,dotangle=-90](0.2,0.5)
\endpspicture}}
\newcommand{\Wv}{\raisebox{-2mm}{\psset{unit=5.8mm}\pspicture(0.1,0)(1,1)
\rput(0.5,0.5){\psline[linewidth=0.5pt](0,-0.5)(0,0.5)\psline[linewidth=0.5pt](-0.5,0)(0.5,0)}
\rput(0.5,0.5){$\ss\bullet$}
\psdots[dotstyle=triangle*,dotscale=0.9](0.5,0.8)\psdots[dotstyle=triangle*,dotscale=0.9,dotangle=180](0.5,0.2)
\psdots[dotstyle=triangle*,dotscale=0.9,dotangle=-90](0.2,0.5)\psdots[dotstyle=triangle*,dotscale=0.9,dotangle=90](0.8,0.5)
\endpspicture}}
\newcommand{\Wvi}{\raisebox{-2mm}{\psset{unit=5.8mm}\pspicture(0.1,0)(1,1)
\rput(0.5,0.5){\psline[linewidth=0.5pt](0,-0.5)(0,0.5)\psline[linewidth=0.5pt](-0.5,0)(0.5,0)}
\rput(0.5,0.5){$\ss\bullet$}
\psdots[dotstyle=triangle*,dotscale=0.9,dotangle=180](0.5,0.8)\psdots[dotstyle=triangle*,dotscale=0.9](0.5,0.2)
\psdots[dotstyle=triangle*,dotscale=0.9,dotangle=90](0.2,0.5)\psdots[dotstyle=triangle*,dotscale=0.9,dotangle=-90](0.8,0.5)
\endpspicture}}
\newcommand{\WLi}{\raisebox{-0.5mm}{\psset{unit=5.8mm}\pspicture(0.4,0.4)(1.02,1.1)
\rput(0.5,0.5){\psline[linewidth=0.5pt](0,0)(0,0.5)\psline[linewidth=0.5pt](0,0)(0.5,0)}
\rput(0.5,0.5){$\ss\bullet$}\psdots[dotstyle=triangle*,dotscale=0.9](0.5,0.8)
\psdots[dotstyle=triangle*,dotscale=0.9,dotangle=-90](0.8,0.5)
\endpspicture}}
\newcommand{\WLii}{\raisebox{-0.5mm}{\psset{unit=5.8mm}\pspicture(0.4,0.4)(1.02,1.1)
\rput(0.5,0.5){\psline[linewidth=0.5pt](0,0)(0,0.5)\psline[linewidth=0.5pt](0,0)(0.5,0)}
\rput(0.5,0.5){$\ss\bullet$}\psdots[dotstyle=triangle*,dotscale=0.9,dotangle=180](0.5,0.8)
\psdots[dotstyle=triangle*,dotscale=0.9,dotangle=90](0.8,0.5)
\endpspicture}}
\newcommand{\WLiii}{\raisebox{-0.5mm}{\psset{unit=5.8mm}\pspicture(0.4,0.4)(1.02,1.1)
\rput(0.5,0.5){\psline[linewidth=0.5pt](0.5,0)(0,0)(0,0.5)}
\rput(0.5,0.5){$\ss\bullet$}\psdots[dotstyle=triangle*,dotscale=0.9](0.5,0.8)
\psdots[dotstyle=triangle*,dotscale=0.9,dotangle=90](0.8,0.5)
\endpspicture}}
\newcommand{\WLiv}{\raisebox{-0.5mm}{\psset{unit=5.8mm}\pspicture(0.4,0.4)(1.02,1.1)
\rput(0.5,0.5){\psline[linewidth=0.5pt](0,0)(0,0.5)\psline[linewidth=0.5pt](0,0)(0.5,0)}
\rput(0.5,0.5){$\ss\bullet$}\psdots[dotstyle=triangle*,dotscale=0.9,dotangle=180](0.5,0.8)
\psdots[dotstyle=triangle*,dotscale=0.9,dotangle=-90](0.8,0.5)
\endpspicture}}
\newcommand{\WRi}{\raisebox{-0.5mm}{\psset{unit=5.8mm}\pspicture(0,0.4)(0.65,1.1)
\rput(0.5,0.5){\psline[linewidth=0.5pt](-0.5,0)(0,0)(0,0.5)}
\rput(0.5,0.5){$\ss\bullet$}
\psdots[dotstyle=triangle*,dotscale=0.9](0.5,0.8)
\psdots[dotstyle=triangle*,dotscale=0.9,dotangle=90](0.2,0.5)
\endpspicture}}
\newcommand{\WRii}{\raisebox{-0.5mm}{\psset{unit=5.8mm}\pspicture(0,0.4)(0.65,1.1)
\rput(0.5,0.5){\psline[linewidth=0.5pt](-0.5,0)(0,0)(0,0.5)}
\rput(0.5,0.5){$\ss\bullet$}
\psdots[dotstyle=triangle*,dotscale=0.9,dotangle=180](0.5,0.8)
\psdots[dotstyle=triangle*,dotscale=0.9,dotangle=-90](0.2,0.5)
\endpspicture}}
\newcommand{\WRiii}{\raisebox{-0.5mm}{\psset{unit=5.8mm}\pspicture(0,0.4)(0.65,1.1)
\rput(0.5,0.5){\psline[linewidth=0.5pt](-0.5,0)(0,0)(0,0.5)}
\rput(0.5,0.5){$\ss\bullet$}
\psdots[dotstyle=triangle*,dotscale=0.9](0.5,0.8)
\psdots[dotstyle=triangle*,dotscale=0.9,dotangle=-90](0.2,0.5)
\endpspicture}}
\newcommand{\WRiv}{\raisebox{-0.5mm}{\psset{unit=5.8mm}\pspicture(0,0.4)(0.65,1.1)
\rput(0.5,0.5){\psline[linewidth=0.5pt](-0.5,0)(0,0)(0,0.5)}
\rput(0.5,0.5){$\ss\bullet$}
\psdots[dotstyle=triangle*,dotscale=0.9,dotangle=180](0.5,0.8)
\psdots[dotstyle=triangle*,dotscale=0.9,dotangle=90](0.2,0.5)
\endpspicture}}

\newcommand{\Vi}{\raisebox{-3.2mm}{\pspicture(-0.1,-0.1)(1.1,1.1)
\rput(0.5,0.5){\psline[linewidth=0.5pt](0,-0.5)(0,0.5)\psline[linewidth=0.5pt](-0.5,0)(0.5,0)}
\rput(0.5,0.5){$\ss\bullet$}
\psdots[dotstyle=triangle*,dotscale=1](0.5,0.8)\psdots[dotstyle=triangle*,dotscale=1](0.5,0.2)
\psdots[dotstyle=triangle*,dotscale=1,dotangle=-90](0.8,0.5)\psdots[dotstyle=triangle*,dotscale=1,dotangle=-90](0.2,0.5)
\endpspicture}}
\newcommand{\Vii}{\raisebox{-3.2mm}{\pspicture(-0.1,-0.1)(1.1,1.1)
\rput(0.5,0.5){\psline[linewidth=0.5pt](0,-0.5)(0,0.5)\psline[linewidth=0.5pt](-0.5,0)(0.5,0)}
\rput(0.5,0.5){$\ss\bullet$}
\psdots[dotstyle=triangle*,dotscale=1,dotangle=180](0.5,0.8)\psdots[dotstyle=triangle*,dotscale=1,dotangle=180](0.5,0.2)
\psdots[dotstyle=triangle*,dotscale=1,dotangle=90](0.8,0.5)\psdots[dotstyle=triangle*,dotscale=1,dotangle=90](0.2,0.5)
\endpspicture}}
\newcommand{\Viii}{\raisebox{-3.2mm}{\pspicture(-0.1,-0.1)(1.1,1.1)
\rput(0.5,0.5){\psline[linewidth=0.5pt](0,-0.5)(0,0.5)\psline[linewidth=0.5pt](-0.5,0)(0.5,0)}
\rput(0.5,0.5){$\ss\bullet$}
\psdots[dotstyle=triangle*,dotscale=1](0.5,0.8)\psdots[dotstyle=triangle*,dotscale=1](0.5,0.2)
\psdots[dotstyle=triangle*,dotscale=1,dotangle=90](0.8,0.5)\psdots[dotstyle=triangle*,dotscale=1,dotangle=90](0.2,0.5)
\endpspicture}}
\newcommand{\Viv}{\raisebox{-3.2mm}{\pspicture(-0.1,-0.1)(1.1,1.1)
\rput(0.5,0.5){\psline[linewidth=0.5pt](0,-0.5)(0,0.5)\psline[linewidth=0.5pt](-0.5,0)(0.5,0)}
\rput(0.5,0.5){$\ss\bullet$}
\psdots[dotstyle=triangle*,dotscale=1,dotangle=180](0.5,0.8)\psdots[dotstyle=triangle*,dotscale=1,dotangle=180](0.5,0.2)
\psdots[dotstyle=triangle*,dotscale=1,dotangle=-90](0.8,0.5)\psdots[dotstyle=triangle*,dotscale=1,dotangle=-90](0.2,0.5)
\endpspicture}}
\newcommand{\Vv}{\raisebox{-3.2mm}{\pspicture(-0.1,-0.1)(1.1,1.1)
\rput(0.5,0.5){\psline[linewidth=0.5pt](0,-0.5)(0,0.5)\psline[linewidth=0.5pt](-0.5,0)(0.5,0)}
\rput(0.5,0.5){$\ss\bullet$}
\psdots[dotstyle=triangle*,dotscale=1](0.5,0.8)\psdots[dotstyle=triangle*,dotscale=1,dotangle=180](0.5,0.2)
\psdots[dotstyle=triangle*,dotscale=1,dotangle=-90](0.2,0.5)\psdots[dotstyle=triangle*,dotscale=1,dotangle=90](0.8,0.5)
\endpspicture}}
\newcommand{\Vvi}{\raisebox{-3.2mm}{\pspicture(-0.1,-0.1)(1.1,1.1)
\rput(0.5,0.5){\psline[linewidth=0.5pt](0,-0.5)(0,0.5)\psline[linewidth=0.5pt](-0.5,0)(0.5,0)}
\rput(0.5,0.5){$\ss\bullet$}
\psdots[dotstyle=triangle*,dotscale=1,dotangle=180](0.5,0.8)\psdots[dotstyle=triangle*,dotscale=1](0.5,0.2)
\psdots[dotstyle=triangle*,dotscale=1,dotangle=90](0.2,0.5)\psdots[dotstyle=triangle*,dotscale=1,dotangle=-90](0.8,0.5)
\endpspicture}}
\newcommand{\Li}{\raisebox{0.3mm}{\pspicture(0.4,0.4)(1.1,1.1)
\rput(0.5,0.5){\psline[linewidth=0.5pt](0,0)(0,0.5)\psline[linewidth=0.5pt](0,0)(0.5,0)}
\rput(0.5,0.5){$\ss\bullet$}\psdots[dotstyle=triangle*,dotscale=1](0.5,0.8)
\psdots[dotstyle=triangle*,dotscale=1,dotangle=-90](0.8,0.5)
\endpspicture}}
\newcommand{\Lii}{\raisebox{0.3mm}{\pspicture(0.4,0.4)(1.1,1.1)
\rput(0.5,0.5){\psline[linewidth=0.5pt](0,0)(0,0.5)\psline[linewidth=0.5pt](0,0)(0.5,0)}
\rput(0.5,0.5){$\ss\bullet$}\psdots[dotstyle=triangle*,dotscale=1,dotangle=180](0.5,0.8)
\psdots[dotstyle=triangle*,dotscale=1,dotangle=90](0.8,0.5)
\endpspicture}}
\newcommand{\Liii}{\raisebox{0.3mm}{\pspicture(0.4,0.4)(1.1,1.1)
\rput(0.5,0.5){\psline[linewidth=0.5pt](0.5,0)(0,0)(0,0.5)}
\rput(0.5,0.5){$\ss\bullet$}\psdots[dotstyle=triangle*,dotscale=1](0.5,0.8)
\psdots[dotstyle=triangle*,dotscale=1,dotangle=90](0.8,0.5)
\endpspicture}}
\newcommand{\Liv}{\raisebox{0.3mm}{\pspicture(0.4,0.4)(1.1,1.1)
\rput(0.5,0.5){\psline[linewidth=0.5pt](0,0)(0,0.5)\psline[linewidth=0.5pt](0,0)(0.5,0)}
\rput(0.5,0.5){$\ss\bullet$}\psdots[dotstyle=triangle*,dotscale=1,dotangle=180](0.5,0.8)
\psdots[dotstyle=triangle*,dotscale=1,dotangle=-90](0.8,0.5)
\endpspicture}}
\newcommand{\Ri}{\raisebox{0.3mm}{\pspicture(-0.1,0.4)(0.6,1.1)
\rput(0.5,0.5){\psline[linewidth=0.5pt](-0.5,0)(0,0)(0,0.5)}
\rput(0.5,0.5){$\ss\bullet$}
\psdots[dotstyle=triangle*,dotscale=1](0.5,0.8)
\psdots[dotstyle=triangle*,dotscale=1,dotangle=90](0.2,0.5)
\endpspicture}}
\newcommand{\Rii}{\raisebox{0.3mm}{\pspicture(-0.1,0.4)(0.6,1.1)
\rput(0.5,0.5){\psline[linewidth=0.5pt](-0.5,0)(0,0)(0,0.5)}
\rput(0.5,0.5){$\ss\bullet$}
\psdots[dotstyle=triangle*,dotscale=1,dotangle=180](0.5,0.8)
\psdots[dotstyle=triangle*,dotscale=1,dotangle=-90](0.2,0.5)
\endpspicture}}
\newcommand{\Riii}{\raisebox{0.3mm}{\pspicture(-0.1,0.4)(0.6,1.1)
\rput(0.5,0.5){\psline[linewidth=0.5pt](-0.5,0)(0,0)(0,0.5)}
\rput(0.5,0.5){$\ss\bullet$}
\psdots[dotstyle=triangle*,dotscale=1](0.5,0.8)
\psdots[dotstyle=triangle*,dotscale=1,dotangle=-90](0.2,0.5)
\endpspicture}}
\newcommand{\Riv}{\raisebox{0.3mm}{\pspicture(-0.1,0.4)(0.6,1.1)
\rput(0.5,0.5){\psline[linewidth=0.5pt](-0.5,0)(0,0)(0,0.5)}
\rput(0.5,0.5){$\ss\bullet$}
\psdots[dotstyle=triangle*,dotscale=1,dotangle=180](0.5,0.8)
\psdots[dotstyle=triangle*,dotscale=1,dotangle=90](0.2,0.5)
\endpspicture}}
\newcommand{\T}{\raisebox{-0.5mm}{\pspicture(0.4,-0.1)(0.6,0.6)
\rput(0.5,0.5){\psline[linewidth=0.5pt](0,0)(0,-0.5)}
\rput(0.5,0.5){$\ss\bullet$}
\psdots[dotstyle=triangle*,dotscale=1](0.5,0.25)
\endpspicture}}
\newcommand{\Bi}{\raisebox{0.3mm}{\pspicture(0.4,0.4)(0.6,1.1)
\rput(0.5,0.5){\psline[linewidth=0.5pt](0,0)(0,0.5)}
\rput(0.5,0.5){$\ss\bullet$}
\psdots[dotstyle=triangle*,dotscale=1](0.5,0.8)
\endpspicture}}
\newcommand{\Bii}{\raisebox{0.3mm}{\pspicture(0.4,0.4)(0.6,1.1)
\rput(0.5,0.5){\psline[linewidth=0.5pt](0,0)(0,0.5)}
\rput(0.5,0.5){$\ss\bullet$}
\psdots[dotstyle=triangle*,dotscale=1,dotangle=180](0.5,0.8)
\endpspicture}}

\renewcommand{\u}{\bar{u}}
\newcommand{\q}{\bar{q}}
\newcommand{\ui}{{\color{blue}u_1}}
\newcommand{\uii}{{\color{green}u_2}}
\newcommand{\uiii}{{\color{red}u_3}}
\newcommand{\uiv}{{\color{brwn}u_4}}
\newcommand{\ubi}{\bar{\color{blue}u}{\color{blue}\mbox{}_1}}
\newcommand{\ubii}{\bar{\color{green}u}{\color{green}\mbox{}_2}}

\author[R.~E.~Behrend]{Roger E.~Behrend}
\address{R.~E.~Behrend, School of Mathematics, Cardiff University, Cardiff, CF24 4AG, UK}
\email{behrendr@cardiff.ac.uk}
\author[I.~Fischer]{Ilse Fischer}
\address{I.~Fischer, Fakult\"{a}t f\"{u}r Mathematik, Universit\"{a}t Wien, Oskar-Morgenstern-Platz 1, 1090 Wien, Austria}
\email{ilse.fischer@univie.ac.at}
\author[M.~Konvalinka]{Matja\v{z} Konvalinka}
\address{M.~Konvalinka, Fakulteta za matematiko in fiziko, Univerza v Ljubljani in In\v{s}titut za matematiko, 
fiziko in mehaniko, Jadranska 19, Ljubljana, Slovenia}
\email{matjaz.konvalinka@fmf.uni-lj.si}
\title[DASASMs of odd order]{Diagonally and antidiagonally symmetric\\
alternating sign matrices of odd order}
\begin{document}
\begin{abstract}
We study the enumeration of diagonally and antidiagonally symmetric alternating sign matrices (DASASMs) of fixed odd order
by introducing a case of the six-vertex model whose configurations are in bijection with such matrices.
The model involves a grid graph on a triangle, with
bulk and boundary weights which satisfy the Yang--Baxter and reflection equations.
We obtain a general expression for the partition function of this model as a sum of two determinantal terms,
and show that at a certain point each of these terms reduces to a Schur function.
We are then able
to prove a conjecture of Robbins from the mid 1980's that the total number of $(2n+1)\times(2n+1)$ DASASMs is
$\prod_{i=0}^n\frac{(3i)!}{(n+i)!}$, and a conjecture of Stroganov from 2008 that the
ratio between the numbers of $(2n+1)\times(2n+1)$ DASASMs with central entry~$-1$ and~$1$ is $n/(n+1)$.
Among the several product formulae for the enumeration of symmetric alternating sign
matrices which were conjectured in the 1980's, that for odd-order DASASMs is the last to
have been proved.
\end{abstract}
\maketitle

\section{Preliminaries}
\subsection{Introduction}\label{intro}
An alternating sign matrix (ASM) is a square matrix in which each entry is~$0$, $1$ or~$-1$, and along each row and column
the nonzero entries alternate in sign and have a sum of~$1$.  These matrices were introduced by Mills, Robbins and Rumsey in the early 1980s,
accompanied by various conjectures concerning their enumeration~\cite[Conjs.~1 \&~2]{MilRobRum82},~\cite[Conjs.~1--7]{MilRobRum83}.
Shortly after this, as discussed by Robbins~\cite[p.~18]{Rob91},~\cite[p.~2]{Rob00},
Richard Stanley made the important suggestion of systematically studying the enumeration
of ASMs invariant under the action of subgroups of the symmetry group of a square.
This suggestion led to numerous conjectures for the straight and weighted enumeration of such symmetric ASMs, with these conjectures being
summarized by Robbins in a preprint written in the mid 1980s, and placed on the arXiv in~2000~\cite{Rob00}.
Much of the content of this preprint also appeared in review papers in 1986 and~1991 by Stanley~\cite{Sta86b} and Robbins~\cite{Rob91},
and in 1999 in a book by Bressoud~\cite[pp.~201--202]{Bre99}.
In the preprint, simple product formulae (or, more specifically, recursion relations which lead to such formulae)
were conjectured for the straight enumeration of several symmetry classes of ASMs,
and it was suggested that no such product formulae exist for the other nonempty classes.  
All except one of these conjectured product formulae had been proved by 2006. (See Section~\ref{symmclass} for further details.)
The single remaining case was that the number of $(2n+1)\times(2n+1)$ diagonally and antidiagonally symmetric ASMs (DASASMs) is
$\prod_{i=0}^n\frac{(3i)!}{(n+i)!}$, and a primary purpose of this paper is to provide the first proof of this formula.
In so doing, a further open conjecture of Stroganov~\cite[Conj.~2]{Str08},
that the ratio between the numbers of $(2n+1)\times(2n+1)$ DASASMs with central entry~$-1$ and~$1$ is $n/(n+1)$, will also be proved.

These results will be proved using a method involving the statistical mechanical six-vertex model.
The structure of the proofs can be summarized as follows.
First, a new case of the six-vertex model is introduced in Sections~\ref{sixvertexmodel}--\ref{weightspartfunct}.
The configurations of this case of the model are
in bijection with DASASMs of fixed odd order, and consist of certain orientations of the edges of a grid graph on a
triangle,~\eqref{Gn}.  The associated partition function~\eqref{Z} consists of a sum, over all such configurations, of products of certain
parameterized bulk and boundary weights, as given in Table~\ref{weights}, with these weights satisfying
the Yang--Baxter and reflection equations,~\eqref{YBE} and~\eqref{RE}.
Straight sums over all configurations, or over all configurations which correspond to DASASMs
with a fixed central entry, are obtained for certain specializations,~\eqref{ZASMenum},~\eqref{ZASMenumPM} and~\eqref{ZPM}, of 
the parameters in the partition function.
By identifying particular properties which uniquely characterize the partition function (including symmetry
with respect to certain parameters, and reduction to a lower order partition function at certain values of some of the parameters),
it is shown in Section~\ref{DASASMFulldetThPr} that the partition function can be expressed as a sum of two determinantal terms,
as given in Theorem~\ref{DASASMFulldetTh}.
It is also shown, using a general determinantal identity~\eqref{OkadaId}, that at a certain point, each
of these terms reduces, up to simple factors, to a Schur function,
as given in Theorem~\ref{ZDASASMFullSchurTh}.
Finally, the main results for the enumeration of odd-order DASASMs, Corollaries~\ref{numDASASMth} and~\ref{StroganovTh},
are obtained by appropriately specializing the variables in the Schur functions, and using
standard results for these functions and for numbers of semistandard Young tableaux.

The proofs given in this paper of enumeration results for odd-order DASASMs share several features with known proofs of enumeration formulae
for other symmetry classes of ASMs, such as those of 
Kuperberg~\cite{Kup96,Kup02}, Okada~\cite{Oka06}, and Razumov and Stroganov~\cite{RazStr06b,RazStr06a}.
However, the proofs of this paper also contain various new and distinguishing characteristics,
which will be outlined in Section~\ref{symmclassproof}.

\subsection{Symmetry classes of ASMs}\label{symmclass}
Several enumerative aspects of standard symmetry classes, and some
related classes, of ASMs will now be discussed in more detail.  

The symmetry group of a square is the dihedral group $D_4=
\{\mathcal{I},\mathcal{V},\mathcal{H},\mathcal{D},\mathcal{A},\mathcal{R}_{\pi/2},\mathcal{R}_{\pi},\mathcal{R}_{-\pi/2}\}$,
where $\mathcal{I}$ is the identity, $\mathcal{V}$, $\mathcal{H}$, $\mathcal{D}$ and $\mathcal{A}$ are
reflections in vertical, horizontal, diagonal and antidiagonal axes, respectively, and $\mathcal{R}_\theta$
is counterclockwise rotation by~$\theta$.  The group has a natural action on the set $\ASM(n)$ of $n\times n$ ASMs, in which
$(\mathcal{I}A)_{ij}=A_{ij}$, $(\mathcal{V}A)_{ij}=A_{i,n+1-j}$, $(\mathcal{H}\,A)_{ij}=A_{n+1-i,j}$, $(\mathcal{D}A)_{ij}=A_{ji}$,
$(\mathcal{A}\,A)_{ij}=A_{n+1-j,n+1-i}$, $(\mathcal{R}_{\pi/2}\,A)_{ij}=A_{j,n+1-i}$, $(\mathcal{R}_{\pi}\,A)_{ij}=A_{n+1-i,n+1-j}$
and $(\mathcal{R}_{-\pi/2}\,A)_{ij}=A_{n+1-j,i}$, for each $A\in\ASM(n)$.
The group has ten subgroups: $\{\mathcal{I}\}$, $\{\mathcal{I},\mathcal{V}\}\approx\{\mathcal{I},\mathcal{H}\}$,
$\{\mathcal{I},\mathcal{V},\mathcal{H},\mathcal{R}_{\pi}\}$, $\{\mathcal{I},\mathcal{D}\}\approx\{\mathcal{I},\mathcal{A}\}$,
$\{\mathcal{I},\mathcal{D},\mathcal{A},\mathcal{R}_{\pi}\}$, $\{\mathcal{I},\mathcal{R}_{\pi}\}$,
$\{\mathcal{I},\mathcal{R}_{\pi/2},\mathcal{R}_{\pi},\mathcal{R}_{-\pi/2}\}$
and~$D_4$, where $\approx$ denotes conjugacy.  In studying the enumeration of symmetric ASMs, the primary task is to obtain formulae
for the cardinalities of each set $\ASM(n,H)$ of $n\times n$ ASMs invariant under the action of subgroup $H$.  Since this cardinality
is the same for conjugate subgroups,
there are eight inequivalent classes.  The standard choices and names for these classes, together with information
about empty subclasses, conjectures and proofs of straight enumeration formulae, and numerical data are as follows.
\begin{list}{$\bullet$}{\setlength{\topsep}{0.8mm}\setlength{\labelwidth}{2mm}\setlength{\leftmargin}{6mm}}
\item $\ASM(n)=\ASM(n,\{\mathcal{I}\})$. Unrestricted ASMs. The formula $|\ASM(n)|=\prod_{i=0}^{n-1}\frac{(3i+1)!}{(n+i)!}$
was conjectured by Mills, Robbins and Rumsey~\cite[Conj.~1]{MilRobRum82}, and
first proved by Zeilberger~\cite[p.~5]{Zei96a}, with further proofs, involving different methods,
subsequently being obtained by Kuperberg~\cite{Kup96} and Fischer~\cite{Fis07}.
\item $\ASM(n,\{\mathcal{I},\mathcal{V}\})$. Vertically symmetric ASMs (VSASMs).
For~$n$ even, the set is empty. For~$n$ odd, a formula was conjectured by Robbins~\cite[Sec.~4.2]{Rob00}, and proved by Kuperberg~\cite[Thm.~2]{Kup02}.
\item $\ASM(n,\{\mathcal{I},\mathcal{V},\mathcal{H},\mathcal{R}_{\pi}\})$. Vertically and horizontally symmetric ASMs (VHSASMs).
For~$n$ even, the set is empty. For~$n$ odd, formulae were conjectured by Mills~\cite[Sec.~4.2]{Rob00}, and proved by Okada~\cite[Thm.~1.2 (A5) \&~(A6)]{Oka06}.
\item $\ASM(n,\{\mathcal{I},\mathcal{R}_{\pi}\})$. Half-turn symmetric ASMs (HTSASMs).
Formulae were conjectured by Mills, Robbins and Rumsey~\cite[p.~285]{MilRobRum86}, and
proved for~$n$ even by Kuperberg~\cite[Thm.~2]{Kup02}, and~$n$ odd by Razumov and Stroganov~\cite[p.~1197]{RazStr06a}.
\item $\ASM(n,\{\mathcal{I},\mathcal{R}_{\pi/2},\mathcal{R}_{\pi},\mathcal{R}_{-\pi/2}\})$. Quarter-turn symmetric ASMs (QTSASMs).
For $n\equiv2\bmod4$, the set is empty.  For $n\not\equiv2\bmod{4}$, formulae were conjectured by Robbins~\cite[Sec.~4.2]{Rob00},
and proved for $n\equiv0\bmod{4}$ by Kuperberg~\cite[Thm.~2]{Kup02}, and~$n$ odd by Razumov and Stroganov~\cite[p.~1649]{RazStr06b}.
\item $\ASM(n,\{\mathcal{I},\mathcal{D}\})$. Diagonally symmetric ASMs (DSASMs).
No formula is currently known or conjectured. Data for $n\le20$ is given by Bousquet-M{\'e}lou and Habsieger~\cite[Tab.~1]{BouHab93}.
\item $\ASM(n,\{\mathcal{I},\mathcal{D},\mathcal{A},\mathcal{R}_{\pi}\})$. Diagonally and antidiagonally symmetric ASMs (DASASMs).
For~$n$ even, no formula is currently known or conjectured. Data for $n\le24$
is given by Bousquet-M{\'e}lou and Habsieger~\cite[Tab.~1]{BouHab93}.
For~$n$ odd, a formula was conjectured by Robbins~\cite[Sec.~4.2]{Rob00}, and is proved in this paper.
\item $\ASM(n,D_4)$. Totally symmetric ASMs (TSASMs). For~$n$ even, the set is empty. For~$n$ odd,
no formula is currently known or conjectured. Data for $n\le27$ is given by Bousquet-M{\'e}lou and Habsieger~\cite[Tab.~1]{BouHab93}.
\end{list}

Alternative approaches to certain parts of some of the proofs cited in this list are also known.
For example, such alternatives have been obtained for
unrestricted ASMs by Colomo and Pronko~\cite[Sec.~5.3]{ColPro05a},~\cite[Sec.~4.2]{ColPro06}, Okada~\cite[Thm.~2.4(1)]{Oka06},
Razumov and Stroganov~\cite[Sec.~2]{RazStr04b},~\cite[Sec.~2]{RazStr09}, and Stroganov~\cite[Sec.~4]{Str06};
VSASMs by Okada~\cite[Thm.~2.4(3)]{Oka06}, and Razumov and Stroganov~\cite[Sec.~3]{RazStr04b};
even-order HTSASMs by Okada~\cite[Thm.~2.4(2)]{Oka06}, Razumov and Stroganov~\cite[Eq.~(31)]{RazStr06a}, and Stroganov~\cite[Eq.~(11)]{Str04};
and QTSASMs of order $0\bmod{4}$ by Okada~\cite[Thm~2.5(1)]{Oka06}.

In addition to the eight standard symmetry classes of ASMs, various closely related classes have also been studied.  A few examples are as follows.
\begin{list}{$\bullet$}{\setlength{\topsep}{0.8mm}\setlength{\labelwidth}{2mm}\setlength{\leftmargin}{6mm}}
\item Quasi quarter-turn symmetric ASMs (qQTSASMs). These are $(4n+2)\times(4n+2)$ ASMs $A$ for which the four central entries
$A_{2n+1,2n+1}$, $A_{2n+1,2n+2}$, $A_{2n+2,2n+1}$ and $A_{2n+2,2n+2}$ are
either $1$, $0$, $0$ and $1$ respectively, or 0, $-1$, $-1$ and 0 respectively,
while the remaining entries satisfy invariance under quarter-turn rotation, i.e., $A_{ij}=A_{j,4n+3-i}$
for all other $i,j$. They were introduced, and a product formula for their enumeration
was obtained, by Aval and Duchon~\cite{AvaDuc09,AvaDuc10}.
\item Off-diagonally symmetric ASMs (OSASMs). These are $2n\times2n$ DSASMs in which each entry on the diagonal is~$0$.
They were introduced, and a product formula for their straight enumeration (which is identical to that for $(2n+1)\times(2n+1)$ VSASMs) was obtained,
by Kuperberg~\cite[Thm.~5]{Kup02}.
\item Off-diagonally and off-antidiagonally symmetric ASMs (OOSASMs). These are $4n\times4n$ DASASMs
in which each entry on the diagonal and antidiagonal is~$0$.  They were introduced, and certain results were obtained,
by Kuperberg~\cite{Kup02}, although no simple formula for their straight enumeration is currently known or conjectured.
\end{list}

In addition to results for the straight enumeration of all elements of standard or related symmetry classes of ASMs, 
various other enumeration results and conjectures are known for certain classes, some examples being as follows.
\begin{list}{$\bullet$}{\setlength{\topsep}{0.8mm}\setlength{\labelwidth}{2mm}\setlength{\leftmargin}{6mm}}
\item Certain results and conjectures are known for the refined or
weighted enumeration of classes of ASMs with respect to statistics from among the following: the positions of $1$'s in or near the outer rows
or columns of an ASM; the number of~$-1$'s in an ASM or part of an ASM; and the number of so-called inversions in an ASM.
For the case of unrestricted ASMs see, for example, Ayyer and Romik~\cite{AyyRom13}, Behrend~\cite{Beh13}, and references therein.
For cases involving certain other classes of ASMs see, for example, Cantini~\cite{Can13}, de Gier, Pyatov and Zinn-Justin~\cite{DegPyaZin09}, 
Fischer and Riegler~\cite{FisRie15}, Hagendorf and Morin-Duchesne~\cite{HagMor16},
Kuperberg~\cite{Kup02}, Okada~\cite{Oka06},
Razumov and Stroganov~\cite{RazStr04b,RazStr06a}, Robbins~\cite{Rob00}, and Stroganov~\cite{Str04}.
\item In addition to the formula~\eqref{StroganovEq} proved in this paper for
the ratio between the numbers of odd-order DASASMs with central entry~$-1$ and~$1$,
analogous formulae have been proved for HTSASMs by Razumov and Stroganov~\cite[Sec.~5.2]{RazStr06a}
and for qQTSASMs by Aval and Duchon~\cite[Sec.~5]{AvaDuc10}, and have been conjectured for odd-order QTSASMs by Stroganov~\cite[Conjs.~1a \& 1b]{Str08}.
\item Various results and conjectures are known for the refined enumeration of several classes of ASMs
with respect to so-called link patterns of associated fully packed loop configurations.
For further information, see for example Cantini and Sportiello~\cite{CanSpo11,CanSpo14}, and de Gier~\cite{Deg09}.
\item Various connections between the straight or refined enumeration of classes of ASMs
and classes of certain plane partitions have been proved or conjectured.
For further information see, for example, Behrend~\cite[Secs.~3.12, 3.13 \&~3.15]{Beh13} and references therein.
\item Relationships between the partition functions of cases of the six-vertex model associated with classes of ASMs
and certain refined Cauchy and Littlewood identities have been obtained by Betea and Wheeler~\cite{BetWhe16},
Betea, Wheeler and Zinn-Justin~\cite{BetWheZin15}, and Wheeler and Zinn-Justin~\cite{WheZin16}.
\end{list}

\subsection{Proofs of enumeration results for symmetry classes of ASMs}\label{symmclassproof}
Most of the proofs mentioned in Section~\ref{symmclass} use a method which directly involves
the statistical mechanical six-vertex model. (A few exceptions are proofs of Cantini and Sportiello~\cite{CanSpo11,CanSpo14},
Fischer~\cite{Fis07}, and Zeilberger~\cite{Zei96a}.)

Although there are several variations on this method, the general features are often as follows.
First, a case of the six-vertex model on a particular graph with certain boundary conditions is introduced, for which the configurations
are in bijection with the ASMs under consideration. The associated partition function is then a sum, over all such configurations,
of products of parameterized bulk weights, and possibly also boundary weights.  By using certain local relations satisfied by these weights,
such as the Yang--Baxter and reflection equations, properties which uniquely determine the partition function are identified.
These properties are then used to show that the partition function can be expressed in terms of one or more determinants or Pfaffians.
Finally, enumeration formulae are obtained by suitably specializing the parameters in the partition function,
and applying certain results for the transformation or evaluation of determinants or Pfaffians.

This general method was introduced in the proofs of Kuperberg~\cite{Kup96,Kup02}, with certain steps in the unrestricted ASM and VSASM cases
being based on previously-known results.
For example, a bijection between $\ASM(n)$ and configurations of the six-vertex model on an $n\times n$ square grid with
domain-wall boundary conditions had been
observed by
Elkies, Kuperberg, Larsen and Propp~\cite[Sec.~7]{ElkKupLarPro92b}, and (using different terminology) by Robbins and Rumsey~\cite[pp.~179--180]{RobRum86}, 
properties which uniquely determine the partition function for $\ASM(n)$
had been identified by Korepin~\cite{Kor82}, a determinantal expression for the partition function for $\ASM(n)$
had been obtained by Izergin~\cite[Eq.~(5)]{Ize87}, and a determinantal expression closely related to the partition function for
odd-order VSASMs had been obtained by Tsuchiya~\cite{Tsu98}.

Following the pioneering proofs of Kuperberg~\cite{Kup96,Kup02}, a further important development was the observation, made by
Okada~\cite{Oka06}, Razumov and Stroganov~\cite{RazStr04b,RazStr06a} and Stroganov~\cite{Str04,Str06}, that for certain
combinatorially-relevant values of a particular parameter, the partition function can often be written in terms of determinants
which are associated with characters of irreducible representations of classical groups.  Expressing the partition function in this
way, and using known product formulae for the dimensions of such representations, can then lead to
simplifications in the proofs of enumeration results.

The approach used in this paper to prove enumeration formulae for odd-order DASASMs has already been
summarized in Section~\ref{intro}, and largely follows the general method outlined in the current section.
However, some specific comparisons between odd-order DASASMs and other ASM classes which have been studied using
this method, are as follows.
\begin{list}{$\bullet$}{\setlength{\topsep}{0.8mm}\setlength{\labelwidth}{2mm}\setlength{\leftmargin}{6mm}}
\item As will be discussed in Section~\ref{weightspartfunct}, six-vertex model boundary weights depend on four possible local configurations at a degree-2
boundary vertex. Boundary weights were previously used by Kuperberg~\cite[Fig.~15]{Kup02} for classes
including VSASMs, VHSASMs, OSASMs and OOSASMs, and in each of these cases, the boundary weights
were identically zero for two of the four local configurations.  However, the boundary weights
used for odd-order DASASMs differ from those used previously for other ASM classes, and are not identically zero
for any of the four local configurations.  As will be discussed in Section~\ref{LocRel}, the boundary weights used by Kuperberg~\cite[Fig.~15]{Kup02}
and those used in this paper constitute various special cases of the most general boundary weights which satisfy the reflection equation
for the six-vertex model.
\item The natural decomposition of the partition function into a sum of two terms has previously only been observed for one case,
that of odd-order HTSASMs, for which a decomposition in which each term involves a product of two determinants, of matrices whose
sizes differ by~1, was obtained
by Razumov and Stroganov~\cite[Thm.~1]{RazStr06a}.  Odd-order DASASMs now provide a further example in which
the partition function is expressed, in Theorem~\ref{DASASMFulldetTh}, as a sum of two terms,
but in this case each term involves only a single determinant.  A further difference is that the matrices for 
odd-order HTSASMs have a uniform structure in all
rows and columns, whereas the matrices for odd-order DASASMs have a special structure in the last row, and a uniform structure 
elsewhere. For odd-order HTSASMs, the decomposition of the partition function into two terms is closely related to the behavior of
parameters associated with the central row and column of the HTSASMs,
and similarly, for odd-order DASASMs, it is related to the behavior of a parameter
associated with the central column of the DASASMs.  However, when this parameter is set to~1, the odd-order DASASM partition function
reduces to a single determinantal term, as given in Corollary~\ref{DASASMdetCoroll}.  
\item The previously-studied ASM classes can be broadly divided into the following
types: (i) those (including unrestricted ASMs, VSASMs, VHSASMs and HTSASMs) in which one set of parameters is associated with
the horizontal edges of a grid graph, and another set of parameters is associated with the vertical edges;
(ii) those (including QTSASMs, qQTSASMs, OSASMs and OOSASMs) in which a single set of parameters is associated with both horizontal
and vertical edges of a grid graph.  For the classes of type (i), the partition function is naturally expressed in terms of determinants
of matrices whose rows are associated with the horizontal parameters, and columns are associated with the vertical parameters.
For the classes of type (ii), the partition function is naturally expressed in terms of Pfaffians.
For odd-order DASASMs, there is a single set of parameters
associated with both horizontal and vertical edges of a grid graph, and hence this class can be regarded as belonging to type (ii).  
However, in contrast to the
previously-studied classes of this type, the partition function is naturally expressed in terms of determinants.
\end{list}

\subsection{DASASMs}\label{DASASMs}
The set $\ASM(n,\{\mathcal{I},\mathcal{D},\mathcal{A},\mathcal{R}_{\pi}\})$ of all $n\times n$ DASASMs will be denoted
in the rest of this paper as $\DASASM(n)$. Hence,
\begin{equation}\label{DASASM}\DASASM(n)=\{A\in\ASM(n)\mid A_{ij}=A_{ji}=A_{n+1-j,n+1-i},\text{ for all }1\le i,j\le n\}.\end{equation}
Note that each $A\in\DASASM(n)$ also satisfies $A_{ij}=A_{n+1-i,n+1-j}$ for all $1\le i,j\le n$, i.e., $A$ is half-turn symmetric.

For example,
\begin{equation}\label{DASASM3}
\DASASM(3)=\left\{
\begin{pmatrix}1&0&0\\0&1&0\\0&0&1\end{pmatrix},\,
\begin{pmatrix}0&1&0\\1&-1&1\\0&1&0\end{pmatrix},\,
\begin{pmatrix}0&0&1\\0&1&0\\1&0&0\end{pmatrix}\right\},\end{equation}
and an element of $\DASASM(7)$ is
\begin{equation}\label{DASASMexample}
\begin{pmatrix}0&0&1&0&0&0&0\\0&1&-1&0&1&0&0\\1&-1&0&1&-1&1&0\\0&0&1&-1&1&0&0\\0&1&-1&1&0&-1&1\\0&0&1&0&-1&1&0\\0&0&0&0&1&0&0
\end{pmatrix}.\end{equation}

Consider any $A\in\DASASM(2n+1)$. 
The central column (or row) of $A$ is invariant under reversal of the order of its entries, and
its nonzero entries alternate in sign and have a sum of~$1$.
Therefore, the central entry $A_{n+1,n+1}$ is nonzero (since either $A_{n+1,n+1}=1$ and
all other entries of the central column are~0, or else the two nearest nonzero entries to $A_{n+1,n+1}$
in the central column are identical), and it is~$1$ or~$-1$
according to whether the number of nonzero entries among the first $n$ entries of the central column
is even or odd, respectively (since the first nonzero entry in the central column is a~$1$).

These observations can be summarized as
\begin{equation}\label{N0}A_{n+1,n+1}=(-1)^{n+N(A)},\quad\text{for each }A\in\DASASM(2n+1),\end{equation}
where $N(A)$ is the number of~$0$'s among the first $n$ entries of the central column of $A$.

The sets of all $(2n+1)\times(2n+1)$ DASASMs with fixed central entry~$1$ and $-1$ will be denoted
as $\DASASM_+(2n+1)$ and $\DASASM_-(2n+1)$, respectively, i.e.,
\begin{equation}\DASASM_\pm(2n+1)=\{A\in\DASASM(2n+1)\mid A_{n+1,n+1}=\pm1\}.\end{equation}

The rest of this paper will be primarily focused on obtaining results which lead to formulae (given in
Corollaries~\ref{numDASASMth} and~\ref{StroganovTh}) for $|\DASASM(2n+1)|$ and
$|\DASASM_\pm(2n+1)|$.  For reference, these cardinalities for $n=0,\ldots,7$ are given in Table~\ref{card}.

\begin{table}[h]\centering
$\begin{array}{|c@{\;\;\;}|@{\;\;\;}c@{\;\;\;\;}c@{\;\;\;\;}c@{\;\;\;\;}c@{\;\;\;\;}c@{\;\;\;\;}c@{\;\;\;\;}c@{\;\;\;\;}c@{\;\;\;}|}\hline\rule{0ex}{4.7mm}
n&0&1&2&3&4&5&6&7\\[1mm]\hline
\rule{0ex}{5.9mm}|\DASASM(2n+1)|&1&3&15&126&1782&42471&1706562&115640460\\[1.8mm]
|\DASASM_+(2n+1)|&1&2&9&72&990&23166&918918&61674912\\[1.8mm]
|\DASASM_-(2n+1)|&0&1&6&54&792&19305&787644&53965548\\[1.8mm]\hline
\end{array}$\\[2.4mm]\caption{Numbers of odd-order DASASMs.}\label{card}
\end{table}

\subsection{Odd DASASM triangles}\label{OddDASASMtriangles}
Let an odd DASASM triangle $A$ of order~$n$ be a triangular array
\begin{equation}\begin{array}{ccc@{}c@{\;\:}c@{\,}c@{}ccc}
A_{11}&A_{12}&A_{13}&\;\:\ldots&A_{1,n+1}&\ldots&A_{1,2n-1}&A_{1,2n}&A_{1,2n+1}\\
&A_{22}&A_{23}&\;\:\ldots&A_{2,n+1}&\ldots&A_{2,2n-1}&A_{2,2n}\\
&&\!\!\ddots&&\vdots&&\!\!\!\!\udots\\[-0.8mm]
&&&A_{nn}&A_{n,n+1}&A_{n,n+2}\\
&&&&\!\!A_{n+1,n+1},\!\!\end{array}
\end{equation}
such that each entry is~$0$, $1$ or~$-1$ and, for each $i=1,\ldots,n+1$, the nonzero entries along the sequence
\begin{equation}
\label{path}
\begin{matrix}
A_{1i}&&&&A_{1,2n+2-i}\\
A_{2i}&&&&A_{2,2n+2-i}\\
\vdots&&&&\vdots\\
A_{i-1,i}&&&&A_{i-1,2n+2-i}\\
A_{ii}&A_{i,i+1}&\ldots&A_{i,2n+1-i}&A_{i,2n+2-i}
\end{matrix}
\end{equation}
alternate in sign and have a sum of~1, where the sequence is read downward from $A_{1i}$ to $A_{ii}$, then rightward
to $A_{i,2n+2-i}$, and then upward to $A_{1,2n+2-i}$ (and for $i=n+1$, the sequence is taken to be
$A_{1,n+1},\ldots,A_{n,n+1},A_{n+1,n+1},A_{n,n+1},\ldots,A_{1,n+1}$).

It can be seen that there is a bijection from $\DASASM(2n+1)$ to the set of odd DASASM triangles of order $n$,
in which the entries $A_{ij}$ of $A\in\DASASM(2n+1)$ are simply restricted to $i=1,\ldots,n+1$ and
$j=i,\ldots,2n+2-i$.

As examples, the set of odd DASASM triangles of order $1$ is
\begin{equation}\label{DASASMtri3}
\left\{\begin{matrix}1&0&0\\&1\end{matrix}\;,\begin{matrix}&0&1&0\\&&-1\end{matrix}\;,\begin{matrix}&0&0&1\\&&1\end{matrix}\right\}\end{equation}
(where the elements correspond, in order, to the DASASMs in~\eqref{DASASM3}), and the odd DASASM triangle
which corresponds to the DASASM in~\eqref{DASASMexample} is
\begin{equation}\label{DASASMtriangleexample}
\begin{matrix}0&0&1&0&0&0&0\\&1&-1&0&1&0\\&&0&1&-1\\&&&-1
\end{matrix}.\end{equation}

\subsection{Six-vertex model configurations}\label{sixvertexmodel}
Define a grid graph on a triangle as
\psset{unit=6mm}
\begin{equation}\label{Gn}\mathcal{T}_n=\raisebox{-26mm}{\pspicture(-0.2,-1)(12,6.7)
\rput(0.9,6.4){$\scriptscriptstyle(0,1)$}\rput(2.1,6.4){$\scriptscriptstyle(0,2)$}\rput(6,6.4){$\scriptscriptstyle(0,n+1)$}
\rput(9.9,6.4){$\scriptscriptstyle(0,2n)$}\rput(11.5,6.4){$\scriptscriptstyle(0,2n+1)$}
\rput(0.9,4.7){$\scriptscriptstyle(1,1)$}\rput(11.3,4.7){$\scriptscriptstyle(1,2n+1)$}
\rput(1.9,3.7){$\scriptscriptstyle(2,2)$}\rput(10.1,3.7){$\scriptscriptstyle(2,2n)$}
\rput(4.8,0.7){$\scriptscriptstyle(n,n)$}\rput(7.3,0.7){$\scriptscriptstyle(n,n+2)$}\rput(6,-0.3){$\scriptscriptstyle(n+1,n+1)$}
\rput(12,3.5){.}
\multirput(1,6)(1,0){11}{$\ss\bullet$}\multirput(1,5)(1,0){11}{$\ss\bullet$}\multirput(2,4)(1,0){9}{$\ss\bullet$}
\multirput(3,3)(1,0){7}{$\ss\bullet$}\multirput(4,2)(1,0){5}{$\ss\bullet$}\multirput(5,1)(1,0){3}{$\ss\bullet$}
\rput(6,0){$\ss\bullet$}
\psline[linewidth=0.5pt](1,6)(1,5)(11,5)(11,6)\psline[linewidth=0.5pt](2,6)(2,4)(10,4)(10,6)\psline[linewidth=0.5pt](3,6)(3,3)(9,3)(9,6)
\psline[linewidth=0.5pt](4,6)(4,2)(8,2)(8,6)\psline[linewidth=0.5pt](5,6)(5,1)(7,1)(7,6)\psline[linewidth=0.5pt](6,6)(6,0)
\endpspicture}\end{equation}
Note that~$\mathcal{T}_n$ can be regarded as an odd-order analog of the graph introduced by Kuperberg~\cite[Fig.~13]{Kup02}
for OOSASMs.

The vertices of~$\mathcal{T}_n$ consist of
top vertices $(0,j)$, $j=0,\ldots,2n+1$, of degree~1, left boundary vertices $(i,i)$, $i=1,\ldots,n$, of degree~2,
bulk vertices $(i,j)$, $i=1,\ldots,n$, $j=i+1,\ldots,2n+1-i$, of degree~4, right boundary vertices $(i,2n+2-i)$, $i=1,\ldots,n$,
of degree~2, and a bottom vertex $(n+1,n+1)$ of degree~1.  The edges incident with the top vertices will be referred to as top edges.

Now define a configuration of the six-vertex model on~$\mathcal{T}_n$ to be an orientation of the edges of~$\mathcal{T}_n$, such that each
top edge is directed upwards, and among the four edges incident to each
bulk vertex, two are directed into and two are directed out of the vertex, i.e., the so-called six-vertex rule is
satisfied.

\psset{unit=7mm}
For such a configuration $C$,
and a vertex $(i,j)$ of~$\mathcal{T}_n$, define the local configuration $C_{ij}$ at $(i,j)$ to be the orientation of the edges incident to $(i,j)$.
Hence, the possible local configurations are~\T\ at a top vertex,
\Li, \Lii, \Liii\ or \Liv\ at a left boundary vertex,
\Vi, \Vii, \Viii, \Viv, \Vv \ or \Vvi\ at a bulk vertex,
\Ri, \Rii, \Riii\ or \Riv\ at a right boundary vertex, and \Bi\ or \Bii\  at the bottom vertex.

There is a natural bijection from the set of configurations of the six-vertex model on~$\mathcal{T}_n$ to
the set of odd DASASM triangles of order~$n$, in which a configuration $C$ is mapped to an
odd DASASM triangle~$A$ given by
\begin{equation}\label{bij}A_{ij}=
\left\{\rule[-10mm]{0mm}{20mm}\right.\!\!
\begin{array}{@{}r@{\;\;\;\;}l@{}}1,&C_{ij}=\Vv, \Liii, \Riii\text{ or }\Bi,\\
-1,&C_{ij}=\Vvi, \Liv, \Riv\text{ or }\Bii,\\
0,&C_{ij}=\Vi, \Vii, \Viii, \Viv, \Li, \Lii, \Ri\text{ or }\Rii,\end{array}\end{equation}
for $i=1,\ldots,n+1$ and $j=i,\ldots,2n+2-i$.
Note that the (fixed) local configurations at the top vertices are not associated with entries of $A$.
Note also that the cases of~\eqref{bij} can be summarized as follows: $A_{ij}=1$ if $C_{ij}$ is \Vv or a restriction
of that (to the upper and right, upper and left, or upper edges), $A_{ij}=-1$ if $C_{ij}$ is \Vvi or a restriction
of that (again to the upper and right, upper and left, or upper edges), and $A_{ij}=0$ otherwise.

\psset{unit=6mm}
As examples, the set of configurations of the six-vertex model on $\mathcal{T}_1$ is
\begin{equation}\label{SV3}\left\{\raisebox{-6.3mm}{
\pspicture(0,-0.2)(2.7,2.2)\multirput(0,2)(1,0){3}{$\ss\bullet$}\multirput(0,1)(1,0){3}{$\ss\bullet$}\rput(1,0){$\ss\bullet$}
\psline[linewidth=0.5pt](0,2)(0,1)(2,1)(2,2)\psline[linewidth=0.5pt](1,0)(1,2)
\multirput(0,1)(1,0){3}{\psdots[dotstyle=triangle*,dotscale=1.1](0,0.5)}\multirput(0,1)(1,0){2}{\psdots[dotstyle=triangle*,dotscale=1.1,dotangle=90](0.5,0)}
\psdots[dotstyle=triangle*,dotscale=1.1](1,0.5)\rput(2.3,0.8){,}\endpspicture
\pspicture(-0.2,-0.2)(2.7,2.2)\multirput(0,2)(1,0){3}{$\ss\bullet$}\multirput(0,1)(1,0){3}{$\ss\bullet$}\rput(1,0){$\ss\bullet$}
\psline[linewidth=0.5pt](0,2)(0,1)(2,1)(2,2)\psline[linewidth=0.5pt](1,0)(1,2)
\multirput(0,1)(1,0){3}{\psdots[dotstyle=triangle*,dotscale=1.1](0,0.5)}
\psdots[dotstyle=triangle*,dotscale=1.1,dotangle=-90](0.5,1)\psdots[dotstyle=triangle*,dotscale=1.1,dotangle=90](1.5,1)
\psdots[dotstyle=triangle*,dotscale=1.1,dotangle=180](1,0.5)\rput(2.3,0.8){,}\endpspicture
\pspicture(-0.2,-0.2)(2.2,2.2)\multirput(0,2)(1,0){3}{$\ss\bullet$}\multirput(0,1)(1,0){3}{$\ss\bullet$}\rput(1,0){$\ss\bullet$}
\psline[linewidth=0.5pt](0,2)(0,1)(2,1)(2,2)\psline[linewidth=0.5pt](1,0)(1,2)
\multirput(0,1)(1,0){3}{\psdots[dotstyle=triangle*,dotscale=1.1](0,0.5)}\multirput(0,1)(1,0){2}{\psdots[dotstyle=triangle*,dotscale=1.1,dotangle=-90](0.5,0)}
\psdots[dotstyle=triangle*,dotscale=1.1](1,0.5)\endpspicture}\right\}\end{equation}
(where the elements correspond, in order, to the odd DASASM triangles in~\eqref{DASASMtri3}), and
the configuration which corresponds to the odd DASASM triangle in~\eqref{DASASMtriangleexample} is
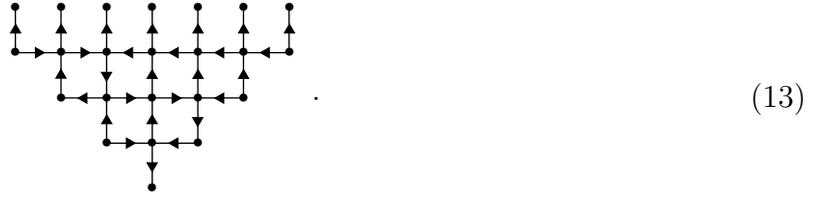
\begin{equation}\label{configex}\raisebox{-12mm}{\pspicture(0.8,-0.2)(7.2,4.2)
\rput(7.6,2){.}
\multirput(1,4)(1,0){7}{$\ss\bullet$}\multirput(1,3)(1,0){7}{$\ss\bullet$}\multirput(2,2)(1,0){5}{$\ss\bullet$}
\multirput(3,1)(1,0){3}{$\ss\bullet$}\rput(4,0){$\ss\bullet$}
\psline[linewidth=0.5pt](1,4)(1,3)(7,3)(7,4)\psline[linewidth=0.5pt](2,4)(2,2)(6,2)(6,4)\psline[linewidth=0.5pt](3,4)(3,1)(5,1)(5,4)
\psline[linewidth=0.5pt](4,4)(4,0)
\multirput(1,3)(1,0){7}{\psdots[dotstyle=triangle*,dotscale=1.1](0,0.5)}
\psdots[dotstyle=triangle*,dotscale=1.1](2,2.5)
\psdots[dotstyle=triangle*,dotscale=1.1,dotangle=180](3,2.5)
\multirput(4,2)(1,0){3}{\psdots[dotstyle=triangle*,dotscale=1.1](0,0.5)}
\multirput(3,1)(1,0){2}{\psdots[dotstyle=triangle*,dotscale=1.1](0,0.5)}
\psdots[dotstyle=triangle*,dotscale=1.1,dotangle=180](5,1.5)
\psdots[dotstyle=triangle*,dotscale=1.1,dotangle=180](4,0.5)
\multirput(1,3)(1,0){2}{\psdots[dotstyle=triangle*,dotscale=1.1,dotangle=-90](0.5,0)}
\multirput(3,3)(1,0){4}{\psdots[dotstyle=triangle*,dotscale=1.1,dotangle=90](0.5,0)}
\multirput(2,2)(3,0){2}{\psdots[dotstyle=triangle*,dotscale=1.1,dotangle=90](0.5,0)}
\multirput(3,2)(1,0){2}{\psdots[dotstyle=triangle*,dotscale=1.1,dotangle=-90](0.5,0)}
\psdots[dotstyle=triangle*,dotscale=1.1,dotangle=-90](3.5,1)
\psdots[dotstyle=triangle*,dotscale=1.1,dotangle=90](4.5,1)
\endpspicture}\end{equation}

\psset{unit=7mm}
The fact that~\eqref{bij} is a well-defined bijection can be verified by considering the 
standard bijection between $\ASM(2n+1)$ and the set of configurations of the six-vertex model on a $(2n+1)\times(2n+1)$ square grid $\mathcal{S}_n$ 
with domain-wall boundary conditions, and then restricting from $\ASM(2n+1)$ to $\DASASM(2n+1)$. 
Some of the details of this bijection and the restriction are as follows.
The grid~$\mathcal{S}_n$, which contains the grid $\mathcal{T}_n$ as a subgraph, consists of bulk vertices $(i,j)$, of degree~4,
together with top vertices $(0,j)$, right vertices $(i,2n+2)$, bottom vertices $(2n+2,j)$ and left vertices $(i,0)$, all of degree~1, for $i,j=1,\ldots,2n+1$,
where $(i,j)$ appears in row $i$ and column $j$.
The configurations are orientations of the edges of $\mathcal{S}_n$, such that 
each edge incident to a top, right, bottom or left vertex is directed upwards, leftwards, downwards or rightwards, respectively,
and the six-vertex rule is satisfied at each bulk vertex. The ASM $A$ which corresponds to a configuration~$C$ is given by
$A_{ij}=1$ or $A_{ij}=-1$ if the local configuration of~$C$ at $(i,j)$ is $\Vv$ or $\Vvi$, respectively, and 
$A_{ij}=0$, otherwise.  If $A$ is a DASASM, then the symmetry conditions for~$A$ imply that~$C$ is uniquely determined by its restriction
to the edges of~$\mathcal{T}_n$.  Note also that in this case, within~$\mathcal{S}_n$,
only $\Vi$, $\Vii$, $\Vv$ and $\Vvi$ can occur at a vertex $(i,i)$ on the diagonal, 
only $\Viii$, $\Viv$, $\Vv$ and $\Vvi$ can occur at a vertex $(i,2n+2-i)$ on the antidiagonal, 
and hence only $\Vv$ or $\Vvi$ can occur at the central vertex $(n+1,n+1)$.

\subsection{Weights and the partition function}\label{weightspartfunct}
Throughout the rest of this paper, the notation
\begin{equation}\label{not}\bar{x}=x^{-1}\qquad\text{and}\qquad\sigma(x)=x-\bar{x}\end{equation}
will be used.

For each possible local configuration $c$ at a bulk or boundary vertex,
and for parameters $q$ and~$u$, assign
a weight $W(c,u)$, as given in Table~\ref{weights}
(where each boundary weight $W(c,u)$ appears in the same row as bulk weights 
whose local configurations restrict to $c$).
\begin{table}[h]\centering
$\begin{array}{|@{\;}l|@{\;}l|@{\;}l|}\hline\rule{0ex}{3.5ex}
\text{Bulk weights}&\text{Left boundary weights}&\text{Right boundary weights}\\[2.2mm]\hline
\hline\rule{0ex}{3.6ex}W(\Wv,u)=W(\Wvi,u)=1&W(\WLiii,u)=W(\WLiv,u)=1&W(\WRiii,u)=W(\WRiv,u)=1\\[2.8mm]
W(\Wi,u)=W(\Wii,u)=\frac{\sigma(q^2u)}{\sigma(q^4)}&W(\WLi,u)=W(\WLii,u)=\frac{\sigma(q\,u)}{\sigma(q)}&\\[3.2mm]
W(\Wiii,u)=W(\Wiv,u)=\frac{\sigma(q^2\u)}{\sigma(q^4)}&&W(\WRi,u)=W(\WRii,u)=\frac{\sigma(q\,\u)}{\sigma(q)}\\[3.4mm]\hline
\end{array}$\\[2.4mm]\caption{Bulk and boundary weights.}\label{weights}
\end{table}

It will be convenient for the dependence on~$q$ not to be indicated explicitly by
the notation $W(c,u)$, or by the notation for further $q$-dependent quantities in this paper.
Note also that the corresponding parameter which is used for such weights in the literature
is often the square of the~$q$ used in this paper.

It can be seen that the weights of Table~\ref{weights} satisfy
\begin{align}\notag W(c,1)\big|_{q=e^{i\pi/6}}&=1,\quad\text{for each local configuration $c$ at a bulk vertex},\\
\label{comb}\text{and}\quad W(c,1)&=1,\quad\text{for each local configuration $c$ at a boundary vertex}.\end{align}

For a configuration $C$ of the six-vertex model on~$\mathcal{T}_n$, and parameters $q$ and $u_1,\ldots,u_{n+1}$,
define the weight of left boundary vertex $(i,i)$ to be $W(C_{ii},u_i)$, the weight
of bulk vertex $(i,j)$ to be $W(C_{ij},u_i\,u_{\min(j,2n+2-j)})$, and the
weight of right boundary vertex $(i,2n+2-i)$ to be $W(C_{i,2n+2-i},u_i)$,
where, as before, $C_{ij}$ is the local configuration of $C$ at $(i,j)$.
Note that~$q$ is an overall constant,
which is the same in all of these weights.

The assignment of $u_1,\ldots,u_{n+1}$ in these weights can be illustrated, for $n=3$, as
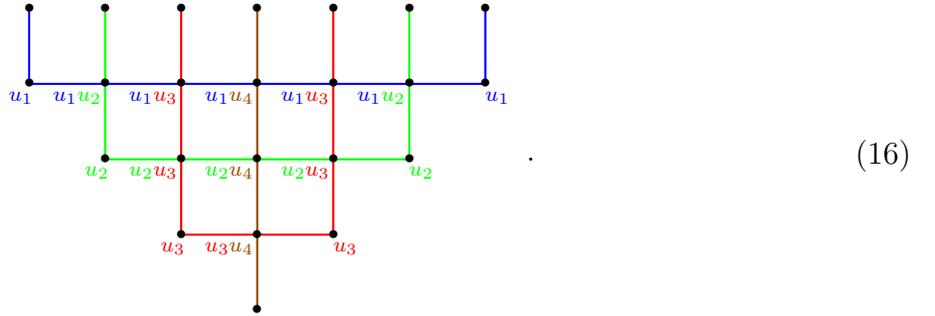
\psset{unit=10mm}
\begin{equation}\label{col}\raisebox{-21.4mm}{\pspicture(0.8,-0.2)(7.2,4.2)
\rput(7.6,2){.}
\psline[linewidth=0.8pt,linecolor=blue](1,4)(1,3)(7,3)(7,4)\psline[linewidth=0.8pt,linecolor=green](2,4)(2,2)(6,2)(6,4)
\psline[linewidth=0.8pt,linecolor=red](3,4)(3,1)(5,1)(5,4)\psline[linewidth=0.8pt,linecolor=brwn](4,4)(4,0)
\multirput(1,4)(1,0){7}{$\ss\bullet$}\multirput(1,3)(1,0){7}{$\ss\bullet$}\multirput(2,2)(1,0){5}{$\ss\bullet$}
\multirput(3,1)(1,0){3}{$\ss\bullet$}\rput(4,0){$\ss\bullet$}
\rput[tr](1.05,2.9){$\ss\ui$}\rput[tr](1.95,2.9){$\ss\ui\uii$}
\rput[tr](2.95,2.9){$\ss\ui\uiii$}\rput[tr](3.95,2.9){$\ss\ui\uiv$}
\rput[tr](4.95,2.9){$\ss\ui\uiii$}\rput[tr](5.95,2.9){$\ss\ui\uii$}\rput[tl](7,2.9){$\ss\ui$}
\rput[tr](2.05,1.9){$\ss\uii$}\rput[tr](2.95,1.9){$\ss\uii\uiii$}\rput[tr](3.95,1.9){$\ss\uii\uiv$}
\rput[tr](4.95,1.9){$\ss\uii\uiii$}\rput[tl](6,1.9){$\ss\uii$}
\rput[tr](3.05,0.9){$\ss\uiii$}\rput[tr](3.95,0.9){$\ss\uiii\uiv$}\rput[tl](5,0.9){$\ss\uiii$}
\endpspicture}\end{equation}
The color coding in~\eqref{col} indicates that for $i=1,\ldots,n+1$, $u_i$
can be naturally associated with the edges in column~$i$, row~$i$ and column $2n+2-i$ of~$\mathcal{T}_n$, such that the parameter for
the weight of a boundary vertex is the single parameter associated with the incident edges,
and the parameter for the weight of a bulk vertex is the product of the two different parameters associated with the incident edges.

Now define the odd-order DASASM partition function $Z(u_1,\ldots,u_{n+1})$ to be the sum of products of bulk and boundary vertex weights, over all
configurations $C$ of the six-vertex model on~$\mathcal{T}_n$, i.e.,
\begin{equation}\label{Z}Z(u_1,\ldots,u_{n+1})=\sum_{C}\,\prod_{i=1}^n W(C_{ii},u_i)\,\Biggl(\prod_{j=i+1}^{2n+1-i}
W(C_{ij},u_i\,u_{\min(j,2n+2-j)})\Biggr)\,W(C_{i,2n+2-i},u_i).\end{equation}
Note that the top vertices and bottom vertex can be regarded as each having weight~1.

For example,
\begin{equation}\label{Z1}Z(u_1,u_2)=\frac{\sigma(q^2\u_1\u_2)\,\sigma(q\u_1)}{\sigma(q^4)\,\sigma(q)}+
\frac{\sigma(qu_1)\,\sigma(q\u_1)}{\sigma(q)^2}+
\frac{\sigma(qu_1)\,\sigma(q^2u_1u_2)}{\sigma(q)\,\sigma(q^4)}\end{equation}
(where the terms are written in an order which corresponds to that used in~\eqref{SV3}),
and the term of~$Z(u_1,\ldots,u_4)$ which corresponds to the configuration in~\eqref{configex} is
\begin{equation}\frac{
\sigma(qu_1)\,\sigma(q^2u_1u_2)\,\sigma(q^2\u_1\u_4)\,\sigma(q^2\u_1\u_3)\,\sigma(q^2\u_1\u_2)\,\sigma(q\u_1)
\,\sigma(q^2u_2u_4)\,\sigma(q\u_2)\,\sigma(qu_3)}{\sigma(q)^4\,\sigma(q^4)^5}.\end{equation}

Similarly, define the odd-order DASASM$_+$ and DASASM$_-$ partition functions
$Z_+(u_1,\ldots,u_{n+1})$ and $Z_-(u_1,\ldots,u_{n+1})$ to be the sum of products of bulk and boundary vertex weights, over all
configurations~$C$ of the six-vertex model on~$\mathcal{T}_n$ in which the edge incident to the bottom vertex is directed 
upwards or downwards, respectively. Hence,
\begin{equation}\label{Zsum}Z(u_1,\ldots,u_{n+1})=Z_+(u_1,\ldots,u_{n+1})+Z_-(u_1,\ldots,u_{n+1}).\end{equation}

It follows from~\eqref{comb}, and the bijections among the set of configurations of the six-vertex model on~$\mathcal{T}_n$,
the set of odd DASASM triangles of order~$n$ and $\DASASM(2n+1)$, that
\begin{equation}\label{ZASMenum}|\DASASM(2n+1)|=Z(\underbrace{1,\ldots,1}_{n+1})\big|_{q=e^{i\pi/6}}\end{equation}
and
\begin{equation}\label{ZASMenumPM}|\DASASM_\pm(2n+1)|=Z_\pm(\underbrace{1,\ldots,1}_{n+1})\big|_{q=e^{i\pi/6}}.\end{equation}

Now consider the replacement of $u_{n+1}$ with $-u_{n+1}$ in the sum of~\eqref{Z}.
Then the bulk weights $W(\Wi,u_iu_{n+1})$, $W(\Wii,u_iu_{n+1})$,
$W(\Wiii,u_iu_{n+1})$ and $W(\Wiv,u_iu_{n+1})$, whose local configurations are associated with the 0's in
the first $n$ rows of the central column of the corresponding DASASMs, change signs under this replacement,
while all other bulk weights and all boundary weights are unchanged.

It follows, using~\eqref{N0}, that
\begin{equation}\label{Zdiff}(-1)^nZ(u_1,\ldots,u_n,-u_{n+1})=
Z_+(u_1,\ldots,u_{n+1})-Z_-(u_1,\ldots,u_{n+1}),\end{equation}
and so, using~\eqref{Zsum}, that
\begin{equation}\label{ZPM}Z_\pm(u_1,\ldots,u_{n+1})=\textstyle\frac{1}{2}\bigl(Z(u_1,\ldots,u_n,u_{n+1})\pm
(-1)^nZ(u_1,\ldots,u_n,-u_{n+1})\bigr).\end{equation}

\subsection{Schur functions and semistandard Young tableaux}\label{SchurfunctionsSSYT}
The notation and results regarding Schur functions and semistandard Young tableaux which will be used in this paper are as follows.
(For further information, see for example Stanley~\cite[Ch.~7]{Sta99}.)

For a partition $\lambda$ of length $\ell(\lambda)\le k$, and variables $x_1,\ldots,x_k$,
let $s_\lambda(x_1,\ldots,x_k)$ be the Schur function (or Schur polynomial) indexed by $\lambda$,
and let $\SSYT_\lambda(k)$ be the set of semistandard Young tableaux of shape $\lambda$ with entries from $\{1,\ldots,k\}$.

A determinantal formula for Schur functions is
\begin{equation}\label{Schurdet}s_\lambda(x_1,\ldots,x_k)=\frac{\displaystyle\det_{1\le i,j\le k}\Bigl(x_i^{\lambda_j+k-j}\Bigr)}
{\prod_{1\le i<j\le k}\bigl(x_i-x_j\bigr)},\end{equation}
a formula for Schur functions involving a sum over semistandard Young tableaux is
\begin{equation}\label{SchurSSYT}s_\lambda(x_1,\ldots,x_k)=\sum_{T\in\SSYT_\lambda(k)}x_1^{\#(1,T)}\ldots x_k^{\#(k,T)},\end{equation}
and a product formula for the number of semistandard Young tableaux is
\begin{equation}\label{numSSYTWeyl}|\SSYT_\lambda(k)|=\frac{\prod_{1\le i<j\le k}(\lambda_i-\lambda_j-i+j)}
{\prod_{i=1}^{k-1}i!},\end{equation}
where $\lambda_i=0$ for $i=\ell(\lambda)+1,\ldots,k$, and $\#(i,T)$ denotes the number of occurrences of $i$ in $T$.

It follows immediately from~\eqref{Schurdet} that $s_\lambda(x_1,\ldots,x_k)$ is symmetric in $x_1,\ldots,x_k$,
and immediately from~\eqref{SchurSSYT} that
\begin{equation}\label{numSSYTSchur}s_\lambda(\underbrace{1,\ldots,1}_k)=|\SSYT_\lambda(k)|.\end{equation}
A further identity is
\begin{equation}\label{Schurder}\textstyle\frac{d}{dx}s_\lambda(\underbrace{1,\ldots,1}_{k-1},x)\big|_{x=1}=
\frac{|\lambda|}{k}\;|\SSYT_\lambda(k)|,\end{equation}
which can be proved by using symmetry and~\eqref{SchurSSYT} to give
\begin{equation*}k\,s_\lambda(\underbrace{1,\ldots,1}_{k-1},x)=s_\lambda(x,\underbrace{1,\ldots,1}_{k-1})+\,\cdots\,+
s_\lambda(\underbrace{1,\ldots,1}_{k-1},x)=
\sum_{T\in\SSYT_\lambda(k)}\bigl(x^{\#(1,T)}+\,\cdots\,+x^{\#(k,T)}\bigr),\end{equation*}
and hence
\begin{equation*}k\,{\textstyle\frac{d}{dx}}s_\lambda(\underbrace{1,\ldots,1}_{k-1},x)\big|_{x=1}
=\sum_{T\in\SSYT_\lambda(k)}\bigl(\#(1,T)+\,\cdots\,+\#(k,T)\bigr)=\sum_{T\in\SSYT_\lambda(k)}|\lambda|=|\SSYT_\lambda(k)|\;|\lambda|.\end{equation*}

\section{Main results}\label{MainResults}
In this section, the main results of the paper are presented, with the cases of 
odd-order DASASMs whose central entry is arbitrary or fixed 
being considered in separate subsections.  The proofs of Theorems~\ref{DASASMFulldetTh} and~\ref{ZDASASMFullSchurTh}
will be given in Section~\ref{Proofs}.  All of the other results are corollaries of these theorems, 
and their proofs from the theorems are given in this section. The notation of~\eqref{not} is used in all of the results.

\subsection{Results for odd-order DASASMs with arbitrary central entry}
The main results of this paper involving odd-order DASASMs whose central entry is arbitrary are as follows.

\begin{theorem}\label{DASASMFulldetTh}
The odd-order DASASM partition function is given by
\begin{multline}\label{DASASMFulldet}Z(u_1,\ldots,u_{n+1})=\\
\frac{\sigma(q^2)^n}{\sigma(q)^{2n}\,\sigma(q^4)^{n^2}}
\prod_{i=1}^n\frac{\sigma(u_i)\sigma(qu_i)\sigma(q\u_i)\sigma(q^2u_iu_{n+1})\sigma(q^2\u_i\u_{n+1})}{\sigma(u_i\u_{n+1})}
\prod_{1\le i<j\le n}\biggl(\frac{\sigma(q^2u_iu_j)\sigma(q^2\u_i\u_j)}{\sigma(u_i\u_j)}\biggr)^2\\
\,\times\left(\det_{1\le i,j\le n+1}\left(\begin{cases}\frac{q^2+\q^2+u_i^2+\u_j^2}{\sigma(q^2u_iu_j)\,\sigma(q^2\u_i\u_j)},&i\le n\\[1.5mm]
\frac{u_{n+1}-1}{u_j^2-1},&i=n+1\end{cases}\right)+
\det_{1\le i,j\le n+1}\left(\begin{cases}\frac{q^2+\q^2+\u_i^2+u_j^2}{\sigma(q^2u_iu_j)\,\sigma(q^2\u_i\u_j)},&i\le n\\[1.5mm]
\frac{\u_{n+1}-1}{\u_j^2-1},&i=n+1\end{cases}\right)\right).\end{multline}
\end{theorem}
This result will be proved in Section~\ref{DASASMFulldetThPr}, and an alternative proof will be sketched in Section~\ref{DASASMFulldetThPrAlt}.

Note that the two determinants on the RHS of~\eqref{DASASMFulldet} are related to each other by replacement of
$u_1,\ldots,u_{n+1}$ by $\u_1,\ldots,\u_{n+1}$, and that the prefactor is unchanged under this transformation.

\begin{corollary}\label{DASASMdetCoroll}
The odd-order DASASM partition function at $u_{n+1}=1$ is given by
\begin{multline}\label{DASASMdet}Z(u_1,\ldots,u_n,1)=\\
\frac{\sigma(q^2)^n}{\sigma(q)^{2n}\,\sigma(q^4)^{n^2}}
\prod_{i=1}^n\sigma(qu_i)\sigma(q\u_i)\sigma(q^2u_i)\sigma(q^2\u_i)
\prod_{1\le i<j\le n}\biggl(\frac{\sigma(q^2u_iu_j)\sigma(q^2\u_i\u_j)}{\sigma(u_i\u_j)}\biggr)^2\\
\times\det_{1\le i,j\le n}\biggl(\frac{q^2+\q^2+u_i^2+\u_j^2}
{\sigma(q^2u_iu_j)\,\sigma(q^2\u_i\u_j)}\biggr).\end{multline}
\end{corollary}
\begin{proof}
Taking $u_{n+1}\rightarrow1$ in~\eqref{DASASMFulldet}, the last row of each matrix
becomes $(0,\ldots,0,\frac{1}{2})$, and the result then follows.
\end{proof}
An alternative proof of Corollary~\ref{DASASMdetCoroll}, 
which does not use Theorem~\ref{DASASMFulldetTh}, will be outlined in Section~\ref{DASASMFulldetThPrAlt}.

Note that, due to the factor $\prod_{i=1}^n\sigma(q^2u_i)\sigma(q^2\u_i)$ on the RHS of~\eqref{DASASMdet}
(which, unlike other parts of the prefactor, does not cancel with terms from the determinant),
$Z(u_1,\ldots,u_n,1)$ is zero at $u_i=q^{\pm2}$ and $u_i=-q^{\pm2}$, for each $i=1,\ldots,n$.
This property will be explained further by the proof of Proposition~\ref{spec4prop}.

\begin{theorem}\label{ZDASASMFullSchurTh}
The odd-order DASASM partition function at $q=e^{i\pi/6}$ is given by
\begin{multline}\label{ZDASASMFullSchur}
Z(u_1,\ldots,u_{n+1})\big|_{q=e^{i\pi/6}}=
\textstyle3^{-n(n-1)/2}\,\Bigl(\frac{u_{n+1}^n}{u_{n+1}+1}\:s_{(n,n-1,n-1,\ldots,2,2,1,1)}(u_1^2,\u_1^2,\ldots,u_n^2,\u_n^2,\u_{n+1}^2)\\
\textstyle+\frac{\u_{n+1}^n}{\u_{n+1}+1}\:s_{(n,n-1,n-1,\ldots,2,2,1,1)}(u_1^2,\u_1^2,\ldots,u_n^2,\u_n^2,u_{n+1}^2)\Bigr).\end{multline}
\end{theorem}
This result will be proved in Section~\ref{ZDASASMFullSchurPr}, using Theorem~\ref{DASASMFulldetTh}.

Note that, by using the standard formula
for taking the reciprocals of all variables in a Schur function (e.g., Stanley~\cite[Ex.~7.41]{Sta99}), the
Schur functions in~\eqref{ZDASASMFullSchur} could be written instead as
$s_{(n,n-1,n-1,\ldots,2,2,1,1)}(u_1^2,\u_1^2,\ldots,u_n^2,\u_n^2,u_{n+1}^{\mp2})=
u_{n+1}^{\mp2n}\:s_{(n,n,\ldots,2,2,1,1)}(u_1^2,\u_1^2,\ldots,u_n^2,\u_n^2,u_{n+1}^{\pm2})$.

\begin{corollary}\label{ZDASASMSchurTh}
The odd-order DASASM partition function at $u_{n+1}=1$ and $q=e^{i\pi/6}$ is given by
\begin{equation}\label{ZDASASMSchur}
Z(u_1,\ldots,u_n,1)\big|_{q=e^{i\pi/6}}=3^{-n(n-1)/2}\,s_{(n,n-1,n-1,\ldots,2,2,1,1)}(u_1^2,\u_1^2,\ldots,u_n^2,\u_n^2,1).\end{equation}
\end{corollary}
\begin{proof}Set $u_{n+1}=1$ in~\eqref{ZDASASMFullSchur}.\end{proof}
A factorization of the Schur function in~\eqref{ZDASASMSchur},
involving odd orthogonal and symplectic characters, will be given in~\eqref{DASASMSchurfact}.  

Note also that the factor $\prod_{i=1}^n\sigma(q^2u_i)\sigma(q^2\u_i)$ on the RHS of~\eqref{DASASMdet} leads,
at $q=e^{i\pi/6}$, to a factor $\prod_{i=1}^n(u_i^2+1+\u_i^2)$ in $s_{(n,n-1,n-1,\ldots,2,2,1,1)}(u_1^2,\u_1^2,\ldots,u_n^2,\u_n^2,1)$ 
on the RHS of~\eqref{ZDASASMSchur}.  This will appear explicitly in the factorization in~\eqref{DASASMSchurfact}.

\begin{corollary}\label{numDASASMth}
The number of $(2n+1)\times(2n+1)$ DASASMs is given by
\begin{align}\notag|\DASASM(2n+1)|&=3^{-n(n-1)/2}\,\bigl|\SSYT_{(n,n-1,n-1,\ldots,2,2,1,1)}(2n+1)\bigr|\\
\label{numDASASM}&=\prod_{i=0}^n\frac{(3i)!}{(n+i)!}.\end{align}
\end{corollary}
\begin{proof}
The first equality follows immediately by setting $u_1=\ldots=u_n=1$ in~\eqref{ZDASASMSchur},
and using~\eqref{ZASMenum} and~\eqref{numSSYTSchur}.  The second equality
(which was previously obtained by Okada~\cite[Conj.~5.1(2)]{Oka06})
then follows by applying~\eqref{numSSYTWeyl}, and simplifying the resulting expression.
\end{proof}
As indicated in Sections~\ref{intro}--\ref{symmclass},
a recursion relation for $|\DASASM(2n+1)|$ which gives the product formula in~\eqref{numDASASM} was conjectured by Robbins~\cite[Sec.~4.2]{Rob00}.

Note also that, due to the comment after Theorem~\ref{ZDASASMFullSchurTh}, the partition $(n,n-1,n-1,\ldots,2,2,1,1)$
in~\eqref{ZDASASMSchur} and~\eqref{numDASASM} could be replaced by $(n,n,\ldots,2,2,1,1)$.

\subsection{Results for odd-order DASASMs with fixed central entry}
The results of the previous section have certain corollaries for odd-order DASASMs whose central entry is fixed, as follows.

\begin{corollary}
The odd-order DASASM$_\pm$ partition functions are given by
\begin{multline}\label{DASASMFulldetP}
Z_+(u_1,\ldots,u_{n+1})=\\
\frac{\sigma(q^2)^n}{\sigma(q)^{2n}\,\sigma(q^4)^{n^2}}
\prod_{i=1}^n\frac{\sigma(u_i)\sigma(qu_i)\sigma(q\u_i)\sigma(q^2u_iu_{n+1})\sigma(q^2\u_i\u_{n+1})}{\sigma(u_i\u_{n+1})}
\prod_{1\le i<j\le n}\biggl(\frac{\sigma(q^2u_iu_j)\sigma(q^2\u_i\u_j)}{\sigma(u_i\u_j)}\biggr)^2\\
\times\left(\det_{1\le i,j\le n+1}\left(\begin{cases}\frac{q^2+\q^2+u_i^2+\u_j^2}{\sigma(q^2u_iu_j)\,\sigma(q^2\u_i\u_j)},&i\le n\\[1.5mm]
\frac{1}{1-u_j^2},&i=n+1\end{cases}\right)+
\det_{1\le i,j\le n+1}\left(\begin{cases}\frac{q^2+\q^2+\u_i^2+u_j^2}{\sigma(q^2u_iu_j)\,\sigma(q^2\u_i\u_j)},&i\le n\\[1.5mm]
\frac{1}{1-\u_j^2},&i=n+1\end{cases}\right)\right)\end{multline}
and
\begin{multline}\label{DASASMFulldetM}
Z_-(u_1,\ldots,u_{n+1})=\\
\frac{\sigma(q^2)^n}{\sigma(q)^{2n}\,\sigma(q^4)^{n^2}}
\prod_{i=1}^n\frac{\sigma(u_i)\sigma(qu_i)\sigma(q\u_i)\sigma(q^2u_iu_{n+1})\sigma(q^2\u_i\u_{n+1})}{\sigma(u_i\u_{n+1})}
\prod_{1\le i<j\le n}\biggl(\frac{\sigma(q^2u_iu_j)\sigma(q^2\u_i\u_j)}{\sigma(u_i\u_j)}\biggr)^2\\
\times\left(\det_{1\le i,j\le n+1}\left(\begin{cases}\frac{q^2+\q^2+u_i^2+\u_j^2}{\sigma(q^2u_iu_j)\,\sigma(q^2\u_i\u_j)},&i\le n\\[1.5mm]
\frac{u_{n+1}}{u_j^2-1},&i=n+1\end{cases}\right)+
\det_{1\le i,j\le n+1}\left(\begin{cases}\frac{q^2+\q^2+\u_i^2+u_j^2}{\sigma(q^2u_iu_j)\,\sigma(q^2\u_i\u_j)},&i\le n\\[1.5mm]
\frac{\u_{n+1}}{\u_j^2-1},&i=n+1\end{cases}\right)\right).
\end{multline}
\end{corollary}
\begin{proof}
Apply \eqref{ZPM} to~\eqref{DASASMFulldet}.
\end{proof}

Note that, in contrast to the determinants in~\eqref{DASASMFulldet}, each
of the determinants in~\eqref{DASASMFulldetP} and~\eqref{DASASMFulldetM}
is singular at $u_{n+1}\rightarrow1$ (due to the form of the bottom right entry of each of the matrices).

\begin{corollary}
The odd-order DASASM$_\pm$ partition functions at $q=e^{i\pi/6}$ are given by
\begin{multline}\label{ZDASASMFullSchurP}
Z_+(u_1,\ldots,u_{n+1})\big|_{q=e^{i\pi/6}}=\textstyle3^{-n(n-1)/2}\,
\Bigl(\frac{u_{n+1}^n}{1-u_{n+1}^2}\:s_{(n,n-1,n-1,\ldots,2,2,1,1)}(u_1^2,\u_1^2,\ldots,u_n^2,\u_n^2,\u_{n+1}^2)\\
\textstyle+\frac{\u_{n+1}^n}{1-\u_{n+1}^2}\:s_{(n,n-1,n-1,\ldots,2,2,1,1)}(u_1^2,\u_1^2,\ldots,u_n^2,\u_n^2,u_{n+1}^2)\Bigr)\end{multline}
and
\begin{multline}\label{ZDASASMFullSchurM}
Z_-(u_1,\ldots,u_{n+1})\big|_{q=e^{i\pi/6}}=\textstyle3^{-n(n-1)/2}\,
\Bigl(\frac{u_{n+1}^{\,n+1}}{u_{n+1}^2-1}\:s_{(n,n-1,n-1,\ldots,2,2,1,1)}(u_1^2,\u_1^2,\ldots,u_n^2,\u_n^2,\u_{n+1}^2)\\
\textstyle+\frac{\u_{n+1}^{n+1}}{\u_{n+1}^2-1}\:s_{(n,n-1,n-1,\ldots,2,2,1,1)}(u_1^2,\u_1^2,\ldots,u_n^2,\u_n^2,u_{n+1}^2)\Bigr).\end{multline}
\end{corollary}
\begin{proof}
Apply~\eqref{ZPM} to~\eqref{ZDASASMFullSchur}.
\end{proof}

\begin{corollary}
The odd-order DASASM$_\pm$ partition functions at $u_{n+1}=1$ and $q=e^{i\pi/6}$ are given by
\begin{multline}\label{ZDASASMSchurP}
Z_+(u_1,\ldots,u_n,1)\big|_{q=e^{i\pi/6}}=3^{-n(n-1)/2}\,
\textstyle\Bigl(2\frac{d}{dx}s_{(n,n-1,n-1,\ldots,2,2,1,1)}(u_1^2,\u_1^2,\ldots,u_n^2,\u_n^2,x)\big|_{x=1}\\
-(n-1)\,s_{(n,n-1,n-1,\ldots,2,2,1,1)}(u_1^2,\u_1^2,\ldots,u_n^2,\u_n^2,1)\Bigr)\end{multline}
and
\begin{multline}\label{ZDASASMSchurM}
Z_-(u_1,\ldots,u_n,1)\big|_{q=e^{i\pi/6}}=3^{-n(n-1)/2}\,
\Bigl(n\,s_{(n,n-1,n-1,\ldots,2,2,1,1)}(u_1^2,\u_1^2,\ldots,u_n^2,\u_n^2,1)\\
\textstyle-2\frac{d}{dx}s_{(n,n-1,n-1,\ldots,2,2,1,1)}(u_1^2,\u_1^2,\ldots,u_n^2,\u_n^2,x)\big|_{x=1}\Bigr).\end{multline}
\end{corollary}
\begin{proof}Take $u_{n+1}\rightarrow1$ in~\eqref{ZDASASMFullSchurP} and~\eqref{ZDASASMFullSchurM},
and apply L'H\^opital's rule.\end{proof}

\begin{corollary}\label{StroganovTh}
The numbers of $(2n+1)\times(2n+1)$ DASASMs with fixed central entry~$1$ and $-1$ are given by
\begin{equation}|\DASASM_+(2n+1)|=\frac{n+1}{2n+1}\prod_{i=0}^n\frac{(3i)!}{(n+i)!}\end{equation}
and
\begin{equation}|\DASASM_-(2n+1)|=\frac{n}{2n+1}\prod_{i=0}^n\frac{(3i)!}{(n+i)!},\end{equation}
and hence satisfy
\begin{equation}\label{StroganovEq}\frac{|\DASASM_-(2n+1)|}{|\DASASM_+(2n+1)|}=\frac{n}{n+1}.\end{equation}
\end{corollary}
\begin{proof}
Set $u_1=\ldots=u_n=1$ in~\eqref{ZDASASMSchurP} and~\eqref{ZDASASMSchurM}, and apply~\eqref{ZASMenumPM},~\eqref{numSSYTSchur},~\eqref{Schurder}
and the second equality of~\eqref{numDASASM}.
\end{proof}
As indicated in Section~\ref{intro},
the relation~\eqref{StroganovEq} was conjectured by Stroganov~\cite[Conj.~2]{Str08}.

\section{Proofs}\label{Proofs}
In this section, full proofs of Theorems~\ref{DASASMFulldetTh} and~\ref{ZDASASMFullSchurTh},
and sketches of alternative proofs of Theorem~\ref{DASASMFulldetTh} and Corollary~\ref{DASASMdetCoroll}, are given.
Preliminary results
are stated or obtained in Sections~\ref{SimpProp}--\ref{Spec},
while the main steps of the proofs appear in Sections~\ref{DASASMFulldetThPr}--\ref{ZDASASMFullSchurPr}.

\subsection{Simple properties of bulk and boundary weights}\label{SimpProp}
The bulk and boundary weights, as given in Table~\ref{weights}, can immediately be seen to satisfy certain simple properties.
Some examples are as follows, where the notation will be explained at the end of the list.
\begin{list}{$\bullet$}{\setlength{\topsep}{0.8mm}\setlength{\labelwidth}{2mm}\setlength{\leftmargin}{6mm}}
\item Invariance under diagonal reflection or arrow reversal,
\begin{gather}\notag\W{a}{b}{c}{d}{u}=\W{d}{c}{b}{a}{u}=\W{b}{a}{d}{c}{u}=\W{\tilde{a}}{\tilde{b}}{\tilde{c}}{\tilde{d}}{u},\\[1.5mm]
\label{Wref}\WL{a}{b}{u}=\WL{b}{a}{u}=\WL{\tilde{a}}{\tilde{b}}{u},\quad \WR{a}{b}{u}=\WR{b}{a}{u}=\WR{\tilde{a}}{\tilde{b}}{u}.\end{gather}
\item Invariance under simultaneous vertical reflection and replacement of $u$ with $\u$,
\begin{equation}\label{Wrep}\W{a}{b}{c}{d}{u}=\W{c}{b}{a}{d}{\u},\qquad\WL{a}{b}{u}=\WR{b}{a}{\u}.\end{equation}
\item Reduction of the bulk weights at $q^{\pm2}$ or boundary weights at $q^{\pm1}$,
\begin{equation}\label{Wred}\W{a}{b}{c}{d}{q^2}=\W{c}{b}{a}{d}{\q^2}=\delta_{a\tilde{d}}\,\delta_{b\tilde{c}},\qquad
\WL{a}{b}{\q}=\WR{b}{a}{q}=\delta_{a\tilde{b}}.\end{equation}
\end{list}

These equations are satisfied for all edge orientations $a$, $b$, $c$ and $d$ such that the six-vertex rule is satisfied
at degree~4 vertices, with an orientation being taken as in or out, with respect to the indicated endpoint of the edge.
Also, $\tilde{a}$ denotes the reversal of edge orientation~$a$, and $\delta$ is the Kronecker delta.

As examples, the $a=b=\,$in and $c=d=\,$out cases of~\eqref{Wrep} and~\eqref{Wred} are
$W(\Wi,u)=W(\Wiii,\u)$, $W(\WLii,u)=W(\WRii,\u)$, $W(\Wi,q^2)=W(\Wiii,\q^2)=1$
and $W(\WLii,\q)=W(\WRii,q)=0$.

\subsection{Local relations for bulk and boundary weights}\label{LocRel}
The bulk and boundary weights, as given in Table~\ref{weights}, also satisfy certain local relations.
The relations relevant to this paper are as follows, 
where the notation will again be explained at the end of the list.
\begin{list}{$\bullet$}{\setlength{\topsep}{0.8mm}\setlength{\labelwidth}{2mm}\setlength{\leftmargin}{4mm}}
\item Vertical and horizontal forms of the Yang--Baxter equation (VYBE and HYBE),
\psset{unit=1.5mm}
\begin{gather}\notag\pspicture(0,-1)(27,21)
\psline[linewidth=0.8pt,linecolor=green](0,6)(20,6)\psline[linewidth=0.8pt,linecolor=blue,linearc=2](5,19)(5,17)(15,11)(15,1)
\psline[linewidth=0.8pt,linecolor=red,linearc=2](15,19)(15,17)(5,11)(5,1)
\multirput(5,6)(10,0){2}{$\ss\bullet$}\rput(10,14){$\ss\bullet$}
\rput(9.8,11.6){$\ss q^2{\color{blue}u}\bar{\color{red}v}$}
\rput[tr](4.5,5.2){$\ss\color{red}v\color{green}w$}\rput[tr](14.5,5.2){$\ss\color{blue}u\color{green}w$}
\rput[b](5,19.5){$\ss\color{blue}a_1$}\rput[b](15,19.5){$\ss\color{red}a_2$}
\rput[t](15,0.5){$\ss\color{blue}b_1$}\rput[t](5,0.5){$\ss\color{red}b_2$}
\rput[r](-0.5,6){$\ss\color{green}a_3$}\rput[l](20.5,6){$\ss\color{green}b_3$}\rput(27,10){$=$}
\endpspicture\pspicture(-7,-1)(20,21)
\psline[linewidth=0.8pt,linecolor=green](0,14)(20,14)\psline[linewidth=0.8pt,linecolor=blue,linearc=2](5,19)(5,9)(15,3)(15,1)
\psline[linewidth=0.8pt,linecolor=red,linearc=2](15,19)(15,9)(5,3)(5,1)
\multirput(5,14)(10,0){2}{$\ss\bullet$}\rput(10,6){$\ss\bullet$}
\rput(9.8,3.6){$\ss q^2{\color{blue}u}\bar{\color{red}v}$}
\rput[tr](4.5,13.2){$\ss\color{blue}u\color{green}w$}\rput[tr](14.5,13.2){$\ss\color{red}v\color{green}w$}
\rput[b](5,19.5){$\ss\color{blue}a_1$}\rput[b](15,19.5){$\ss\color{red}a_2$}
\rput[t](15,0.5){$\ss\color{blue}b_1$}\rput[t](5,0.5){$\ss\color{red}b_2$}
\rput[r](-0.5,14){$\ss\color{green}a_3$}\rput[l](20.5,14){$\ss\color{green}b_3$}\endpspicture\\
\label{YBE}\raisebox{-17mm}{\pspicture(-2,-2)(26,22)\rput[l](-10,10){and}
\psline[linewidth=0.8pt,linecolor=green](14,0)(14,20)\psline[linewidth=0.8pt,linecolor=blue,linearc=2](1,15)(3,15)(9,5)(19,5)
\psline[linewidth=0.8pt,linecolor=red,linearc=2](1,5)(3,5)(9,15)(19,15)
\multirput(14,5)(0,10){2}{$\ss\bullet$}\rput(6,10){$\ss\bullet$}
\rput[r](5,10){$\ss q^2{\color{blue}u}\bar{\color{red}v}$}
\rput[tr](13.5,14.2){$\ss\color{red}v\color{green}w$}\rput[tr](13.5,4.2){$\ss\color{blue}u\color{green}w$}
\rput[r](0.5,15){$\ss\color{blue}a_1$}\rput[r](0.5,5){$\ss\color{red}a_2$}\rput[t](14,-0.5){$\ss\color{green}a_3$}
\rput[l](19.5,5){$\ss\color{blue}b_1$}\rput[l](19.5,15){$\ss\color{red}b_2$}\rput[b](14,20.5){$\ss\color{green}b_3$}\rput(26,10){$=$}\endpspicture
\pspicture(-6,-2)(22,22)
\psline[linewidth=0.8pt,linecolor=green](6,0)(6,20)\psline[linewidth=0.8pt,linecolor=red,linearc=2](1,5)(11,5)(17,15)(19,15)
\psline[linewidth=0.8pt,linecolor=blue,linearc=2](1,15)(11,15)(17,5)(19,5)
\multirput(6,5)(0,10){2}{$\ss\bullet$}\rput(14,10){$\ss\bullet$}
\rput[r](13,10){$\ss q^2{\color{blue}u}\bar{\color{red}v}$}
\rput[tr](5.5,14.2){$\ss\color{blue}u\color{green}w$}\rput[tr](5.5,4.2){$\ss\color{red}v\color{green}w$}
\rput[r](0.5,15){$\ss\color{blue}a_1$}\rput[r](0.5,5){$\ss\color{red}a_2$}\rput[t](6,-0.5){$\ss\color{green}a_3$}
\rput[l](19.5,5){$\ss\color{blue}b_1$}\rput[l](19.5,15){$\ss\color{red}b_2$}\rput[b](6,20.5){$\ss\color{green}b_3$}\rput(23,8.9){.}\endpspicture}\end{gather}
\item Left and right forms of the reflection (or boundary Yang--Baxter) equation (LRE and RRE),
\psset{unit=1.45mm}
\begin{align}\notag\raisebox{-16.5mm}{\pspicture(-4,-2)(19,25)\psline[linewidth=0.8pt,linecolor=blue,linearc=2](0,23)(0,21)(10,15)(10,0)
\psline[linewidth=0.8pt,linecolor=blue](10,0)(15,0)
\psline[linewidth=0.8pt,linecolor=red,linearc=2](10,23)(10,21)(0,15)(0,10)\psline[linewidth=0.8pt,linecolor=red](0,10)(15,10)
\multirput(10,0)(-10,10){2}{$\ss\bullet$}\rput(10,10){$\ss\bullet$}\rput(5,18){$\ss\bullet$}
\rput(4.9,15.6){$\ss q^2{\color{blue}u}\bar{\color{red}v}$}\rput[tr](9.5,9.2){$\ss\color{blue}u\color{red}v$}
\rput[tr](-0.3,9.7){$\ss\color{red}v$}\rput[tr](9.7,-0.3){$\ss\color{blue}u$}
\rput[b](0,23.5){$\ss\color{blue}a_1$}\rput[b](10,23.5){$\ss\color{red}a_2$}
\rput[l](15.5,0){$\ss\color{blue}b_1$}\rput[l](15.5,10){$\ss\color{red}b_2$}\endpspicture}
&=\raisebox{-16.5mm}{\pspicture(-4,-2)(27,25)\psline[linewidth=0.8pt,linecolor=blue](0,15)(0,10)
\psline[linewidth=0.8pt,linecolor=blue,linearc=2](0,10)(15,10)(21,0)(23,0)
\psline[linewidth=0.8pt,linecolor=red](10,15)(10,0)\psline[linewidth=0.8pt,linecolor=red,linearc=2](10,0)(15,0)(21,10)(23,10)
\multirput(0,10)(10,-10){2}{$\ss\bullet$}\rput(10,10){$\ss\bullet$}\rput(18,5){$\ss\bullet$}
\rput(14.7,5){$\ss q^2{\color{blue}u}\bar{\color{red}v}$}\rput[tr](9.5,9.2){$\ss\color{blue}u\color{red}v$}
\rput[tr](-0.3,9.7){$\ss\color{blue}u$}\rput[tr](9.7,-0.3){$\ss\color{red}v$}
\rput[b](0,15.5){$\ss\color{blue}a_1$}\rput[b](10,15.5){$\ss\color{red}a_2$}
\rput[l](23.5,0){$\ss\color{blue}b_1$}\rput[l](23.5,10){$\ss\color{red}b_2$}\endpspicture}\\
\label{RE}\text{and \ \ }\raisebox{-16.5mm}{\pspicture(-4,-2)(27,25)
\psline[linewidth=0.8pt,linecolor=blue,linearc=2](0,10)(2,10)(8,0)(13,0)\psline[linewidth=0.8pt,linecolor=blue](13,0)(13,15)
\psline[linewidth=0.8pt,linecolor=red,linearc=2](0,0)(2,0)(8,10)(23,10)\psline[linewidth=0.8pt,linecolor=red](23,10)(23,15)
\multirput(13,0)(10,10){2}{$\ss\bullet$}\rput(5,5){$\ss\bullet$}\rput(13,10){$\ss\bullet$}
\rput(1.7,5){$\ss q^2{\color{blue}u}\bar{\color{red}v}$}\rput[tr](12.5,9.2){$\ss\color{blue}u\color{red}v$}
\rput[tl](13.3,-0.3){$\ss\color{blue}u$}\rput[tl](23.3,9.7){$\ss\color{red}v$}
\rput[r](-0.5,10){$\ss\color{blue}a_1$}\rput[r](-0.5,0){$\ss\color{red}a_2$}
\rput[b](13,15.5){$\ss\color{blue}b_1$}\rput[b](23,15.5){$\ss\color{red}b_2$}\endpspicture}
&=\raisebox{-16.5mm}{\pspicture(-4,-2)(19,25)\psline[linewidth=0.8pt,linecolor=blue](0,10)(15,10)
\psline[linewidth=0.8pt,linecolor=blue,linearc=2](15,10)(15,15)(5,21)(5,23)
\psline[linewidth=0.8pt,linecolor=red](0,0)(5,0)\psline[linewidth=0.8pt,linecolor=red,linearc=2](5,0)(5,15)(15,21)(15,23)
\multirput(5,0)(10,10){2}{$\ss\bullet$}\rput(5,10){$\ss\bullet$}\rput(10,18){$\ss\bullet$}
\rput(9.9,15.6){$\ss q^2{\color{blue}u}\bar{\color{red}v}$}\rput[tr](4.5,9.2){$\ss\color{blue}u\color{red}v$}
\rput[tl](5.3,-0.3){$\ss\color{red}v$}\rput[tl](15.3,9.7){$\ss\color{blue}u$}
\rput[r](-0.5,10){$\ss\color{blue}a_1$}\rput[r](-0.5,0){$\ss\color{red}a_2$}
\rput[b](5,23.5){$\ss\color{blue}b_1$}\rput[b](15,23.5){$\ss\color{red}b_2$}\rput(17.2,9){.}\endpspicture}\end{align}
\item Left and right forms of the boundary unitarity equation (LBUE and RBUE),
\psset{unit=1.5mm}
\begin{equation}\label{BUE}\raisebox{-8.1mm}{
\pspicture(-3,-1)(52,11)\psline[linewidth=0.8pt,linecolor=blue,linearc=2](0,5)(5,5)(5,0)
\psline[linewidth=0.8pt,linecolor=blue](0,10)(0,5)\psline[linewidth=0.8pt,linecolor=blue](5,0)(10,0)
\multirput(0,5)(5,-5){2}{$\ss\bullet$}
\rput[tr](-0.7,5.2){$\ss\q{\color{blue}u}$}\rput[tr](4.3,0.1){$\ss\q\bar{\color{blue}u}$}
\rput[b](0,10.5){$\ss\color{blue}a$}\rput[l](10.5,0){$\ss\color{blue}b$}\rput[r](30,5){$=\;
\frac{\sigma(q{\color{blue}u})\,\sigma(q\bar{\color{blue}u})}{\sigma(q)^2}\,\delta_{{\color{blue}a}\tilde{{\color{blue}b}}}$}\rput(41,5){and}\endpspicture
\pspicture(-3,-1)(33,11)\psline[linewidth=0.8pt,linecolor=blue,linearc=2](5,0)(5,5)(10,5)
\psline[linewidth=0.8pt,linecolor=blue](0,0)(5,0)\psline[linewidth=0.8pt,linecolor=blue](10,5)(10,10)
\multirput(5,0)(5,5){2}{$\ss\bullet$}
\rput[tl](5.3,-0.4){$\ss q{\color{blue}u}$}\rput[tl](10.3,4.9){$\ss q\bar{\color{blue}u}$}
\rput[r](-0.5,0){$\ss\color{blue}a$}\rput[b](10,10.5){$\ss\color{blue}b$}\rput[r](33,5){$=\;
\frac{\sigma(q{\color{blue}u})\,\sigma(q\bar{\color{blue}u})}{\sigma(q)^2}\,\delta_{{\color{blue}a}\tilde{{\color{blue}b}}}$.}\endpspicture}
\end{equation}
\end{list}

In these equations, each graph contains external edges, for which only one of the endpoints is indicated,
and internal edges, for which both endpoints are indicated.
Each equation holds for all orientations, $a_1$, $b_1$, \ldots, of the external edges, with (as in Section~\ref{SimpProp})
an orientation being taken as in or out, with respect to the indicated endpoint,
and $\tilde{a}$ denoting the reversal of orientation~$a$.  For a particular
orientation of the external edges, each graph denotes a sum,
over all orientations of the internal edges which satisfy the six-vertex rule at each degree~4 vertex,
of products of weights for each degree~2 and degree~4 vertex shown. If the edges incident to a
degree~4 vertex appear horizontally and vertically, with an associated parameter~$u$ to the left of and below the vertex, then
the weight of the vertex is $\W{a}{b}{c}{d}{u}$, for orientations $a$,~$b$,~$c$ and~$d$ of the edges incident left, below, 
right and above the vertex, respectively.
If the edges incident to a degree~4 vertex appear diagonally, then
these edges and the associated parameter should be rotated so that the parameter appears to the left of and below the vertex, with
the weight then being determined as previously.
For degree~2 vertices, the incident edges always appear in the same configurations as used in the notation for the boundary weights,
with the weights being determined accordingly.  The color coding indicates that the parameters $u$, $v$ and $w$ can be naturally associated with
certain edges.

As an example, the $a_1=a_2=b_1=$\,in and $b_2=$\,out case of the right form of the reflection equation~\eqref{RE} is
\begin{multline*}W(\Wi,q^2u\bar{v})W(\WRii,u)W(\Wiv,uv)W(\WRiii,v)=
W(\WRiii,v)W(\Wi,uv)W(\WRii,u)W(\Wiv,q^2u\bar{v})+\\
W(\WRii,v)W(\Wiv,uv)W(\WRiii,u)W(\Wvi,q^2u\bar{v})+
W(\WRii,v)W(\Wv,uv)W(\WRiv,u)W(\Wiv,q^2u\bar{v}).\end{multline*}

The local relations~\eqref{YBE}--\eqref{BUE} can be proved by directly verifying that
each equation, for each orientation of external edges, is either trivial or reduces to a simple identity
satisfied by the rational functions of Table~\ref{weights}.
Due to the symmetry properties of the weights identified in Section~\ref{SimpProp}, many of the different cases in this verification
are equivalent.  For example, invariance under arrow reversal~\eqref{Wref} implies that
cases of an equation which are related by reversal of all external edge orientations are equivalent,
invariance under diagonal reflection~\eqref{Wref} implies that the vertical and horizontal forms
of the Yang--Baxter equation~\eqref{YBE} are equivalent, and the properties of vertical reflection~\eqref{Wrep} imply
that the left and right forms of the reflection equation~\eqref{RE} are equivalent.

There is an extensive literature regarding the local relations~\eqref{YBE}--\eqref{BUE}.
For some general information regarding the Yang--Baxter equation, as applied to the six-vertex model, see, for example, Baxter~\cite[pp.~187--190]{Bax82}.
The reflection equation was introduced (and applied to six-vertex model bulk weights) by Cherednik~\cite[Eq.~(10)]{Che84},
with important further results being obtained by Sklyanin~\cite{Skl88}.
The most general boundary weights which satisfy the reflection equation for standard
six-vertex model bulk weights were obtained, independently, by de Vega and
Gonz\'{a}lez-Ruiz~\cite[Eq.~(15)]{DevGon93}, and Ghoshal and Zamolodchikov~\cite[Eq.~(5.12)]{GhoZam94}.
The boundary weights used in this paper are a special case of these general weights, which were chosen
for their property of all having value~1 at $u=1$, as in~\eqref{comb},
thereby enabling the straight enumeration of odd-order DASASMs, as in~\eqref{ZASMenum} and~\eqref{ZASMenumPM}.  
It can be shown that,
up to unimportant normalization, these are the only case of the general boundary weights
which have this property.  Finally, note that two other special cases of the general six-vertex model boundary weights
were used by Kuperberg~\cite[Fig.~15]{Kup02}, to study classes of ASMs including VSASMs, VHSASMs, OSASMs and OOSASMs.

\subsection{Properties of the odd-order DASASM partition function}\label{GenProp}
Some important properties of the odd-order DASASM partition function $Z(u_1,\ldots,u_{n+1})$ will now be identified.

\begin{proposition}\label{recipth}
The odd-order DASASM partition function satisfies
\begin{equation}
\label{recip}Z(\u_1,\ldots,\u_{n+1})=Z(u_1,\ldots,u_{n+1}).
\end{equation}
\end{proposition}
\begin{proof}
First observe that an involution on the set of configurations of the six-vertex model on~$\mathcal{T}_n$ is provided by reflection
of each configuration in the central vertical line of~$\mathcal{T}_n$.  Applying this involution to each configuration in the sum~\eqref{Z}
for $Z(u_1,\ldots,u_{n+1})$, and using the properties~\eqref{Wrep} of the bulk and boundary weights under vertical reflection,
leads to the required result.
\end{proof}

\begin{proposition}\label{degth}
The odd-order DASASM partition function $Z(u_1,\ldots,u_{n+1})$ is an even Laurent polynomial in $u_i$ of lower degree at least $-2n$ and
upper degree at most $2n$,
for each $i=1,\ldots,n$, and a Laurent polynomial in $u_{n+1}$ of lower degree at least~$-n$ and upper degree at most $n$.
\end{proposition}
Note that the definitions of degrees being used in this paper are that 
a Laurent polynomial $\sum_{i=m}^na_ix^i$ in $x$, with $m\le n$ and $a_m,a_n\ne0$,
has lower and upper degrees $m$ and $n$, respectively.
\begin{proof}
Consider a configuration $C$ of the six-vertex model on~$\mathcal{T}_n$, and $i\in\{1,\ldots,n\}$.
The $C$-dependent term in the sum~\eqref{Z} for $Z(u_1,\ldots,u_{n+1})$
consists of a product of $n(n+2)$ weights, among which
there are~$2n-1$ bulk weights, one left boundary weight and one right boundary weight that depend on~$u_i$.
Under the bijection~\eqref{bij} from $C$ to an odd DASASM triangle, the local configurations which determine
these $2n+1$ $u_i$-dependent weights correspond to the entries of the triangle in the sequence~\eqref{path}.
Also, it follows from the bijection~\eqref{bij}, the explicit weights in Table~\ref{weights}, and the
form of the $C$-dependent term in~\eqref{Z}, that each nonzero entry in~\eqref{path} is associated with a weight of~$1$, and
each zero entry in~\eqref{path} is associated with a weight which is an odd Laurent polynomial in~$u_i$ of lower degree~$-1$ and upper degree~$1$.
The properties of the sequence~\eqref{path} imply that its number of zero entries is even and at most~$2n$.
Therefore, the $C$-dependent term in~\eqref{Z}
is an even Laurent polynomial in $u_i$ of lower degree at least $-2n$ and
upper degree at most $2n$, from which the required result for $u_i$ follows.

The result for $u_{n+1}$ can be proved similarly.
\end{proof}

\begin{proposition}\label{symm}
The odd-order DASASM partition function $Z(u_1,\ldots,u_{n+1})$ is symmetric in $u_1,\ldots,u_n$.
\end{proposition}
\begin{proof}
The proof is analogous to that used by Kuperberg~\cite[Lem.~11 \& Fig.~13]{Kup02} to show that
the partition function for $4n\times4n$ OOSASMs is symmetric in all of its parameters.

First note that it is sufficient to show that $Z(u_1,\ldots,u_{n+1})$ is symmetric in $u_i$ and $u_{i+1}$, for $i=1,\ldots,n-1$.
The proof of this will be outlined briefly for arbitrary~$n$ and~$i$, and then illustrated in more detail for the case $n=3$ and $i=2$.

Let $\mathcal{T}'_n$ be a modification of the graph $\mathcal{T}_n$, in which an additional degree~4 vertex $x$ has been introduced,
and the two edges incident with $(0,i)$ and $(0,i+1)$ are replaced by
four edges connecting~$x$ to $(0,i)$, $(0,i+1)$, $(1,i)$ and $(1,i+1)$.
It follows, using~\eqref{Z} and the conditions on the configurations~$C$ in~\eqref{Z}, 
that $W(\Wi,q^2\u_iu_{i+1})\,Z(u_1,\ldots,u_{n+1})$ can be expressed as a sum 
of products of bulk and boundary weights 
over all orientations of the edges of $\mathcal{T}'_n$, such that each edge incident with a top vertex $(0,j)$ 
is directed into that vertex and the six-vertex rule is satisfied at each degree~4 vertex, where the 
vertex~$x$ (whose incident edges necessarily have fixed orientations) is assigned a weight $W(\Wi,q^2\u_iu_{i+1})$
and the assignment of weights to other vertices is the same as for~$\mathcal{T}_n$.

It can now be seen that it is possible to apply to $W(\Wi,q^2\u_iu_{i+1})\,Z(u_1,\ldots,u_{n+1})$, in succession, the vertical
form of the Yang--Baxter equation~\eqref{YBE} $i-1$ times, the left form of the reflection equation~\eqref{RE} once,
the horizontal form of the Yang--Baxter equation~\eqref{YBE}
$2(n-i)-1$ times, the right form of the reflection equation~\eqref{RE} once, and the vertical form of the
Yang--Baxter equation~\eqref{YBE} $i-1$ times, where each of these equations involves a bulk weight with parameter $q^2\u_iu_{i+1}$.
The result of applying this sequence of relations is $W(\Wi,q^2\u_iu_{i+1})\,Z(u_1,\ldots u_{i-1},u_{i+1},u_i,u_{i+2},\ldots,u_{n+1})$,
as required.

For the case $n=3$ and $i=2$, the proof can be illustrated as follows, where the notation will be explained at the end:
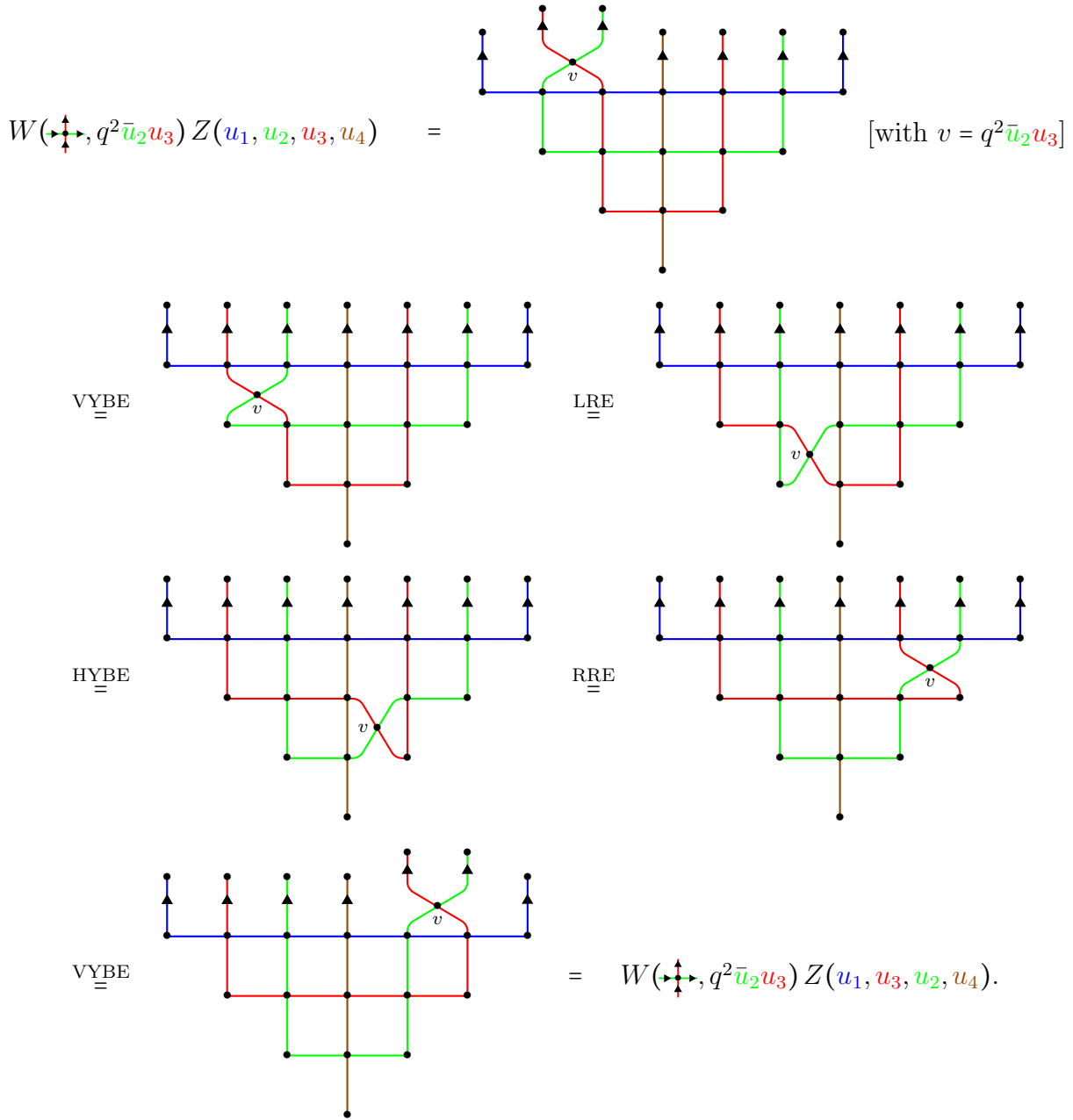
\psset{unit=0.9mm}
\begin{gather*}\pspicture(0,-1)(70,45)\rput[l](0,23)
{$W(\raisebox{-2mm}{\psset{unit=5.8mm}\pspicture(0.1,0)(1,1)
\rput(0.5,0.5){\psline[linewidth=0.7pt,linecolor=red](0,-0.5)(0,0.5)\psline[linewidth=0.7pt,linecolor=green](-0.5,0)(0.5,0)}
\rput(0.5,0.5){$\scriptscriptstyle\bullet$}
\psdots[dotstyle=triangle*,dotscale=0.8](0.5,0.85)\psdots[dotstyle=triangle*,dotscale=0.85](0.5,0.2)
\psdots[dotstyle=triangle*,dotscale=0.8,dotangle=-90](0.85,0.5)\psdots[dotstyle=triangle*,dotscale=0.85,dotangle=-90](0.2,0.5)
\endpspicture},q^2\ubii\uiii)\,Z(\ui,\uii,\uiii,\uiv)$}\endpspicture
\pspicture(-9,-1)(100,45)\rput(-8,23){$=$}
\psline[linewidth=0.8pt,linecolor=blue](0,40)(0,30)(60,30)(60,40)\psline[linewidth=0.8pt,linecolor=brwn](30,0)(30,40)
\psline[linewidth=0.8pt,linecolor=green,linearc=2](20,44)(20,38)(10,32)(10,20)
\psline[linewidth=0.8pt,linecolor=green](10,20)(50,20)(50,40)
\psline[linewidth=0.8pt,linecolor=red,linearc=2](10,44)(10,38)(20,32)(20,10)
\psline[linewidth=0.8pt,linecolor=red](20,10)(40,10)(40,40)\multirput(10,41)(10,0){2}{\psdots[dotstyle=triangle*,dotscale=1.2](0,0)}
\multirput(30,36)(10,0){4}{\psdots[dotstyle=triangle*,dotscale=1.2](0,0)}\rput(0,36){\psdots[dotstyle=triangle*,dotscale=1.2](0,0)}
\multirput(10,44)(10,0){2}{$\ss\bullet$}\multirput(0,40)(15,-5){2}{$\ss\bullet$}
\multirput(30,40)(10,0){4}{$\ss\bullet$}\multirput(0,30)(10,0){7}{$\ss\bullet$}\multirput(10,20)(10,0){5}{$\ss\bullet$}
\multirput(20,10)(10,0){3}{$\ss\bullet$}\rput(30,0){$\ss\bullet$}\rput[b](15,32){$\ss v$}
\rput[l](64,23){[with $v=q^2\ubii\uiii$]}
\endpspicture\\
\pspicture(-16,-1)(60,41)\rput(-11,23){$\stackrel{\mathrm{VYBE}}{=}$}
\psline[linewidth=0.8pt,linecolor=blue](0,40)(0,30)(60,30)(60,40)\psline[linewidth=0.8pt,linecolor=brwn](30,0)(30,40)
\psline[linewidth=0.8pt,linecolor=green](20,40)(20,30)
\psline[linewidth=0.8pt,linecolor=green,linearc=2](20,30)(20,28)(10,22)(10,20)
\psline[linewidth=0.8pt,linecolor=green](10,20)(50,20)(50,40)
\psline[linewidth=0.8pt,linecolor=red](10,40)(10,30)
\psline[linewidth=0.8pt,linecolor=red,linearc=2](10,30)(10,28)(20,22)(20,20)
\psline[linewidth=0.8pt,linecolor=red](20,20)(20,10)(40,10)(40,40)
\multirput(0,36)(10,0){7}{\psdots[dotstyle=triangle*,dotscale=1.2](0,0)}\rput(15,25){$\ss\bullet$}
\multirput(0,40)(10,0){7}{$\ss\bullet$}\multirput(0,30)(10,0){7}{$\ss\bullet$}\multirput(10,20)(10,0){5}{$\ss\bullet$}
\multirput(20,10)(10,0){3}{$\ss\bullet$}\rput(30,0){$\ss\bullet$}\rput[b](15,22){$\ss v$}\endpspicture
\pspicture(-22,-1)(60,41)\rput(-11,23){$\stackrel{\mathrm{LRE}}{=}$}
\psline[linewidth=0.8pt,linecolor=blue](0,40)(0,30)(60,30)(60,40)\psline[linewidth=0.8pt,linecolor=brwn](30,0)(30,40)
\psline[linewidth=0.8pt,linecolor=green](20,40)(20,10)
\psline[linewidth=0.8pt,linecolor=green,linearc=2](20,10)(22,10)(28,20)(30,20)
\psline[linewidth=0.8pt,linecolor=green](30,20)(50,20)(50,40)
\psline[linewidth=0.8pt,linecolor=red](10,40)(10,20)(20,20)
\psline[linewidth=0.8pt,linecolor=red,linearc=2](20,20)(22,20)(28,10)(30,10)
\psline[linewidth=0.8pt,linecolor=red](30,10)(40,10)(40,40)
\multirput(0,36)(10,0){7}{\psdots[dotstyle=triangle*,dotscale=1.2](0,0)}\rput(25,15){$\ss\bullet$}
\multirput(0,40)(10,0){7}{$\ss\bullet$}\multirput(0,30)(10,0){7}{$\ss\bullet$}\multirput(10,20)(10,0){5}{$\ss\bullet$}
\multirput(20,10)(10,0){3}{$\ss\bullet$}\rput(30,0){$\ss\bullet$}\rput[l](21.8,15){$\ss v$}\endpspicture\\
\pspicture(-16,-1)(60,41)\rput(-11,23){$\stackrel{\mathrm{HYBE}}{=}$}
\psline[linewidth=0.8pt,linecolor=blue](0,40)(0,30)(60,30)(60,40)\psline[linewidth=0.8pt,linecolor=brwn](30,0)(30,40)
\psline[linewidth=0.8pt,linecolor=green](20,40)(20,10)(30,10)
\psline[linewidth=0.8pt,linecolor=green,linearc=2](30,10)(32,10)(38,20)(40,20)
\psline[linewidth=0.8pt,linecolor=green](40,20)(50,20)(50,40)
\psline[linewidth=0.8pt,linecolor=red](10,40)(10,20)(30,20)
\psline[linewidth=0.8pt,linecolor=red,linearc=2](30,20)(32,20)(38,10)(40,10)
\psline[linewidth=0.8pt,linecolor=red](40,10)(40,40)
\multirput(0,36)(10,0){7}{\psdots[dotstyle=triangle*,dotscale=1.2](0,0)}\rput(35,15){$\ss\bullet$}
\multirput(0,40)(10,0){7}{$\ss\bullet$}\multirput(0,30)(10,0){7}{$\ss\bullet$}\multirput(10,20)(10,0){5}{$\ss\bullet$}
\multirput(20,10)(10,0){3}{$\ss\bullet$}\rput(30,0){$\ss\bullet$}\rput[l](31.8,15){$\ss v$}\endpspicture
\pspicture(-22,-1)(60,41)\rput(-11,23){$\stackrel{\mathrm{RRE}}{=}$}
\psline[linewidth=0.8pt,linecolor=blue](0,40)(0,30)(60,30)(60,40)\psline[linewidth=0.8pt,linecolor=brwn](30,0)(30,40)
\psline[linewidth=0.8pt,linecolor=green](20,40)(20,10)(40,10)(40,20)
\psline[linewidth=0.8pt,linecolor=green,linearc=2](40,20)(40,22)(50,28)(50,30)
\psline[linewidth=0.8pt,linecolor=green](50,30)(50,40)
\psline[linewidth=0.8pt,linecolor=red](10,40)(10,20)(50,20)
\psline[linewidth=0.8pt,linecolor=red,linearc=2](50,20)(50,22)(40,28)(40,30)
\psline[linewidth=0.8pt,linecolor=red](40,30)(40,40)
\multirput(0,36)(10,0){7}{\psdots[dotstyle=triangle*,dotscale=1.2](0,0)}\rput(45,25){$\ss\bullet$}
\multirput(0,40)(10,0){7}{$\ss\bullet$}\multirput(0,30)(10,0){7}{$\ss\bullet$}\multirput(10,20)(10,0){5}{$\ss\bullet$}
\multirput(20,10)(10,0){3}{$\ss\bullet$}\rput(30,0){$\ss\bullet$}\rput[b](45,22){$\ss v$}\endpspicture\\
\pspicture(-16,-1)(60,45)\rput(-11,23){$\stackrel{\mathrm{VYBE}}{=}$}
\psline[linewidth=0.8pt,linecolor=blue](0,40)(0,30)(60,30)(60,40)\psline[linewidth=0.8pt,linecolor=brwn](30,0)(30,40)
\psline[linewidth=0.8pt,linecolor=green](20,40)(20,10)(40,10)
\psline[linewidth=0.8pt,linecolor=green,linearc=2](40,10)(40,32)(50,38)(50,44)
\psline[linewidth=0.8pt,linecolor=red](10,40)(10,20)(50,20)
\psline[linewidth=0.8pt,linecolor=red,linearc=2](50,20)(50,32)(40,38)(40,44)
\multirput(40,41)(10,0){2}{\psdots[dotstyle=triangle*,dotscale=1.2](0,0)}
\multirput(0,36)(10,0){4}{\psdots[dotstyle=triangle*,dotscale=1.2](0,0)}\rput(60,36){\psdots[dotstyle=triangle*,dotscale=1.2](0,0)}
\multirput(40,44)(10,0){2}{$\ss\bullet$}\multirput(45,35)(15,5){2}{$\ss\bullet$}
\multirput(0,40)(10,0){4}{$\ss\bullet$}\multirput(0,30)(10,0){7}{$\ss\bullet$}\multirput(10,20)(10,0){5}{$\ss\bullet$}
\multirput(20,10)(10,0){3}{$\ss\bullet$}\rput(30,0){$\ss\bullet$}\rput[b](45,32){$\ss v$}\endpspicture
\pspicture(-22,-1)(60,45)\rput(-14,23){$=$}
\rput(25,23){$W(\raisebox{-2mm}{\psset{unit=5.8mm}\pspicture(0.1,0)(1,1)
\rput(0.5,0.5){\psline[linewidth=0.7pt,linecolor=red](0,-0.5)(0,0.5)\psline[linewidth=0.7pt,linecolor=green](-0.5,0)(0.5,0)}
\rput(0.5,0.5){$\scriptscriptstyle\bullet$}
\psdots[dotstyle=triangle*,dotscale=0.8](0.5,0.85)\psdots[dotstyle=triangle*,dotscale=0.85](0.5,0.2)
\psdots[dotstyle=triangle*,dotscale=0.8,dotangle=-90](0.85,0.5)\psdots[dotstyle=triangle*,dotscale=0.85,dotangle=-90](0.2,0.5)
\endpspicture},q^2\ubii\uiii)\,Z(\ui,\uiii,\uii,\uiv)$.}\endpspicture
\end{gather*}

In these diagrams, each graph denotes a sum, over all orientations of the unlabeled edges which satisfy the six-vertex rule at each 
degree~4 vertex, of products of weights
for each degree~2 and degree~4 vertex.  The vertex weights are obtained using the same conventions as in Section~\ref{LocRel},
with the assignment of parameters for vertices whose incident edges appear horizontally and vertically being
determined by the colors of the incident vertices as in the example in~\eqref{col}.  The abbreviations above the $=$ signs are those
given in Section~\ref{LocRel}, and indicate the local relations which give the associated equalities.
\end{proof}

It can easily be checked that Propositions~\ref{recipth}--\ref{symm} are also satisfied if the odd-order DASASM partition function $Z(u_1,\ldots,u_{n+1})$
is replaced by one of the odd-order DASASM$_\pm$ partition functions $Z_\pm(u_1,\ldots,u_{n+1})$,
and that the latter are even or odd in $u_{n+1}$, with
$Z_\pm(u_1,\ldots,u_n,-u_{n+1})=\pm(-1)^nZ_\pm(u_1,\ldots,u_n,u_{n+1})$.
However, these additional results will not be needed.

\subsection{Specializations of the odd-order DASASM partition function}\label{Spec}
It will now be shown, in Propositions~\ref{spec1prop}--\ref{spec3prop}, that for certain specializations of 
the parameters, the DASASM partition function of order~$2n+1$ reduces,
up to a factor, to a DASASM partition function of order~$2n-1$ or~$2n-3$.

Only the specialization in Proposition~\ref{spec3prop} will be used in the proof of Theorem~\ref{DASASMFulldetTh} in Section~\ref{DASASMFulldetThPr},
while the specializations in Propositions~\ref{spec1prop},~\ref{spec2prop} and~\ref{spec3prop} will
be used in the alternative proof of Theorem~\ref{DASASMFulldetTh} in Section~\ref{DASASMFulldetThPrAlt}.

A further specialization, for which the DASASM partition function is zero, will be given in Proposition~\ref{spec4prop}.
This, together with Propositions~\ref{spec1prop} and~\ref{spec2prop}, will be used in the alternative proof of Corollary~\ref{DASASMdetCoroll} 
in Section~\ref{DASASMFulldetThPrAlt}.

\begin{proposition}\label{spec1prop}
If $u_1=q$, then the odd-order DASASM partition function satisfies
\begin{equation}\label{spec1}
Z(q,u_2,\ldots,u_{n+1})=
\frac{(q+\q)\prod_{i=2}^n\sigma(q^3u_i)^2\,\sigma(q^3u_{n+1})}{\sigma(q^4)^{2n-1}}\,Z(u_2,\ldots,u_{n+1}).\end{equation}
\end{proposition}
Note that Proposition~\ref{spec1prop} can be combined with Propositions~\ref{recipth}--\ref{symm} to obtain specializations of $Z(u_1,\ldots,u_{n+1})$ at
$u_i=q^{\pm1}$ and $u_i=-q^{\pm1}$, for $i=1,\ldots,n$.
These specializations will be discussed in Section~\ref{DASASMFulldetThPrAlt}.
\begin{proof}\psset{unit=7mm}
Let $u_1=q$, consider $Z(u_1,\ldots,u_{n+1})$, and apply the formula in~\eqref{Wred} for the reduction of
right boundary weights at~$q$ to $W(C_{1,2n+1},u_1)$ in~\eqref{Z}.  It then follows, using the six-vertex rule and the upward orientation of the top edges
in each configuration, that the contribution from configuration $C$ in~\eqref{Z} is zero unless
the local configurations in the first row of~$\mathcal{T}_n$ are fixed, with
$C_{11}=\Li$ and $C_{12}=\ldots=C_{1,2n}=\Vi$.  This now leads to the RHS of~\eqref{spec1}.
\end{proof}

\begin{proposition}\label{spec2prop}
If $u_1u_2=q^2$, then the odd-order DASASM partition function satisfies
\begin{multline}\label{spec2}
Z(u_1,u_2,\ldots,u_{n+1})=\\
\frac{\sigma(u_1)\sigma(qu_1)\sigma(u_2)\sigma(qu_2)
\prod_{i=3}^n\bigl(\sigma(q^2u_1u_i)\sigma(q^2u_2u_i)\bigr)^2\,\sigma(q^2u_1u_{n+1})\sigma(q^2u_2u_{n+1})}
{\sigma(q)^4\,\sigma(q^4)^{2(2n-3)}}\\
\times Z(u_3,\ldots,u_{n+1}).\end{multline}
\end{proposition}
Note that Proposition~\ref{spec2prop} can be combined with Propositions~\ref{recipth}--\ref{symm} to obtain specializations of $Z(u_1,\ldots,u_{n+1})$ at
$u_iu_j=q^{\pm2}$ and $u_iu_j=-q^{\pm2}$, for distinct $i,j\in\{1,\ldots,n\}$.
These specializations will be discussed in Section~\ref{DASASMFulldetThPrAlt}.
\begin{proof}\psset{unit=7mm}
The proof will be outlined for arbitrary $n$, and then illustrated in more detail for the case $n=3$.
Let $u_1u_2=q^2$, consider $Z(u_1,\ldots,u_{n+1})$,
and apply the formula in~\eqref{Wred} for the reduction of bulk weights at~$q^2$ to
$W(C_{12},u_1u_2)$ and $W(C_{1,2n},u_1u_2)$ in~\eqref{Z}.  This leads to a fixing of local configurations in the first row of~$\mathcal{T}_n$, specifically
$C_{11}=\Li$ and $C_{13}=\ldots=C_{1,2n-1}=\Vi$ for each configuration $C$ in~\eqref{Z} with a nonzero contribution,
thereby giving a factor $W(\WLi,u_1)\prod_{i=3}^n W(\Wi,u_1u_i)^2\,W(\Wi,u_1u_{n+1})$.

Now apply the right boundary unitarity equation~\eqref{BUE} to the boundary weights $W(C_{1,2n+1},u_1)$ and $W(C_{2,2n},u_2)$ in~\eqref{Z}.
This gives a factor $\sigma(u_1)\sigma(u_2)/\sigma(q)^2$, and
leads to a fixing of local configurations in the second row of~$\mathcal{T}_n$, specifically
$C_{22}=\Li$ and $C_{23}=\ldots=C_{2,2n-1}=\Vi$ for each configuration~$C$ in~\eqref{Z} with a nonzero contribution,
thereby giving a further factor $W(\WLi,u_2)\prod_{i=3}^n W(\Wi,u_2u_i)^2\,W(\Wi,u_2u_{n+1})$, and yielding the RHS of~\eqref{spec2}.

For $n=3$, the proof can be illustrated as follows:
\psset{unit=0.8mm}
\begin{align*}&Z(\ui,\uii,\uiii,\uiv)|_{\ui\uii=q^2}\\[1.5mm]
&\makebox[11mm]{$=$}\raisebox{-16.5mm}{\pspicture(-4,-4)(60,42)
\psline[linewidth=0.8pt,linecolor=blue](0,40)(0,30)(60,30)(60,40)\psline[linewidth=0.8pt,linecolor=brwn](30,0)(30,40)
\psline[linewidth=0.8pt,linecolor=green](10,40)(10,20)(50,20)(50,40)\psline[linewidth=0.8pt,linecolor=red](20,40)(20,10)(40,10)(40,40)
\multirput(0,36)(10,0){7}{\psdots[dotstyle=triangle*,dotscale=1.2](0,0)}
\multirput(0,40)(10,0){7}{$\ss\bullet$}\multirput(0,30)(10,0){7}{$\ss\bullet$}\multirput(10,20)(10,0){5}{$\ss\bullet$}
\multirput(20,10)(10,0){3}{$\ss\bullet$}\rput(30,0){$\ss\bullet$}\endpspicture}\\
&\makebox[11mm]{$\stackrel{\eqref{Wred}}{=}$}\frac{\sigma(q\ui)}{\sigma(q)}\raisebox{-16.5mm}{\pspicture(6,-2)(60,42)
\psline[linewidth=0.8pt,linecolor=blue,linearc=3](12,28)(14,30)(20,30)\psline[linewidth=0.8pt,linecolor=blue](20,30)(49,30)
\psline[linewidth=0.8pt,linecolor=blue,linearc=3](52,28)(54,30)(60,30)\psline[linewidth=0.8pt,linecolor=blue](60,30)(60,40)
\psline[linewidth=0.8pt,linecolor=brwn](30,0)(30,40)\psline[linewidth=0.8pt,linecolor=green](10,20)(50,20)
\psline[linewidth=0.8pt,linecolor=green,linearc=3](10,20)(10,26)(12,28)\psline[linewidth=0.8pt,linecolor=green,linearc=3](50,20)(50,26)(52,28)
\psline[linewidth=0.8pt,linecolor=red](20,40)(20,10)(40,10)(40,40)
\multirput(20,36)(10,0){3}{\psdots[dotstyle=triangle*,dotscale=1.2](0,0)}\rput(60,36){\psdots[dotstyle=triangle*,dotscale=1.2](0,0)}
\rput(45,30){\psdots[dotstyle=triangle*,dotscale=1.2,dotangle=-90](0,0)}
\multirput(20,40)(10,0){3}{$\ss\bullet$}\rput(49,30){$\ss\bullet$}\multirput(60,30)(0,10){2}{$\ss\bullet$}
\multirput(20,30)(10,0){3}{$\ss\bullet$}\multirput(10,20)(10,0){5}{$\ss\bullet$}
\multirput(20,10)(10,0){3}{$\ss\bullet$}\rput(30,0){$\ss\bullet$}\endpspicture}\\
&\makebox[11mm]{$=$}\frac{\sigma(q\ui)\,\sigma(q^2\ui\uiii)^2\,\sigma(q^2\ui\uiv)}{\sigma(q)\,\sigma(q^4)^3}
\raisebox{-16.5mm}{\pspicture(6,-2)(60,42)
\psline[linewidth=0.8pt,linecolor=blue,linearc=3](52,28)(54,30)(60,30)\psline[linewidth=0.8pt,linecolor=blue](60,30)(60,40)
\psline[linewidth=0.8pt,linecolor=brwn](30,0)(30,30)\psline[linewidth=0.8pt,linecolor=green](10,30)(10,20)(50,20)
\psline[linewidth=0.8pt,linecolor=green,linearc=3](50,20)(50,26)(52,28)\psline[linewidth=0.8pt,linecolor=red](20,30)(20,10)(40,10)(40,30)
\multirput(10,26)(10,0){4}{\psdots[dotstyle=triangle*,dotscale=1.2](0,0)}\rput(60,36){\psdots[dotstyle=triangle*,dotscale=1.2](0,0)}
\multirput(60,30)(0,10){2}{$\ss\bullet$}\multirput(10,30)(10,0){4}{$\ss\bullet$}\multirput(10,20)(10,0){5}{$\ss\bullet$}
\multirput(20,10)(10,0){3}{$\ss\bullet$}\rput(30,0){$\ss\bullet$}\endpspicture}\\
&\makebox[11mm]{$\stackrel{\mathrm{RBUE}}{=}$}
\frac{\sigma(\ui)\,\sigma(q\ui)\,\sigma(\uii)\,\sigma(q^2\ui\uiii)^2\,\sigma(q^2\ui\uiv)}{\sigma(q)^3\,\sigma(q^4)^3}
\raisebox{-18mm}{\pspicture(6,-4)(50,34)
\psline[linewidth=0.8pt,linecolor=brwn](30,0)(30,30)\psline[linewidth=0.8pt,linecolor=green](10,30)(10,20)(50,20)
\psline[linewidth=0.8pt,linecolor=red](20,30)(20,10)(40,10)(40,30)
\multirput(10,26)(10,0){4}{\psdots[dotstyle=triangle*,dotscale=1.2](0,0)}
\rput(46,20){\psdots[dotstyle=triangle*,dotscale=1.2,dotangle=-90](0,0)}
\multirput(10,30)(10,0){4}{$\ss\bullet$}\multirput(10,20)(10,0){5}{$\ss\bullet$}
\multirput(20,10)(10,0){3}{$\ss\bullet$}\rput(30,0){$\ss\bullet$}\endpspicture}\\
&\makebox[11mm]{$=$}\frac{\sigma(\ui)\,\sigma(q\ui)\,\sigma(\uii)\,\sigma(q\uii)\,(\sigma(q^2\ui\uiii)\,
\sigma(q^2\uii\uiii))^2\,\sigma(q^2\ui\uiv)\,\sigma(q^2\uii\uiv)}
{\sigma(q)^4\,\sigma(q^4)^6}\,Z(\uiii,\uiv),\end{align*}
where the notation is the same as in the example in the proof of Proposition~\ref{symm}.
Note that, in the second and third graphs above, the curved lines connecting certain pairs of vertices denote single edges
(which accordingly have a single orientation), but their color changes midway so that the parameters associated with 
their endpoints are given correctly.
\end{proof}

\begin{proposition}\label{spec3prop}
If $u_1u_{n+1}=q^2$, then the odd-order DASASM partition function satisfies
\begin{equation}\label{spec3}
Z(u_1,\ldots,u_{n+1})=
\frac{\sigma(qu_1)\,(\sigma(q\u_1)+\sigma(q))\prod_{i=2}^n\sigma(q^2u_1u_i)\sigma(q^2u_iu_{n+1})}{\sigma(q)^2\,\sigma(q^4)^{2n-2}}
\,Z(u_2,\ldots,u_n,u_1).\end{equation}
\end{proposition}
Note that Proposition~\ref{spec3prop} can be combined with Propositions~\ref{recipth}--\ref{symm} to obtain specializations of $Z(u_1,\ldots,u_{n+1})$ at
$u_iu_{n+1}=q^{\pm2}$ and $u_iu_{n+1}=-q^{\pm2}$, for $i=1,\ldots,n$.
These specializations will be discussed in Section~\ref{DASASMFulldetThPr}.
\begin{proof}
The proof will be outlined briefly for arbitrary $n$, and then illustrated in more detail for the case $n=3$.

\psset{unit=7mm}
Let $u_1u_{n+1}=q^2$, consider $Z(u_1,\ldots,u_{n+1})$,
and apply the formula in~\eqref{Wred} for the reduction of bulk weights at~$q^2$ to the weight
$W(C_{1,n+1},u_1u_{n+1})$ in~\eqref{Z}.  This leads to
a fixing of local configurations in the left half of the first row of~$\mathcal{T}_n$, specifically
$C_{11}=\Li$ and $C_{12}=\ldots=C_{1n}=\Vi$ for each configuration~$C$ in~\eqref{Z} with a nonzero contribution, thereby giving a factor
$W(\WLi,u_1)\prod_{i=2}^n W(\Wi,u_1u_i)$.

Now, for $i=2,\ldots,n$, apply, in succession, the horizontal
form of the Yang--Baxter equation~\eqref{YBE} $n-i$ times, the right form of the reflection equation~\eqref{RE} once,
and the vertical form of the Yang--Baxter equation~\eqref{YBE}
$i-2$ times, where each of these equations involves a bulk weight with parameter~$u_iu_{n+1}=q^2\u_1u_i$.
This leads to a further fixing of $n-1$ local configurations, which gives a factor $\prod_{i=2}^n W(\Wi,u_iu_{n+1})$.  The overall result is then
\begin{multline*}\frac{\sigma(qu_1)\prod_{i=2}^n\sigma(q^2u_1u_i)\sigma(q^2u_iu_{n+1})}{\sigma(q)\,\sigma(q^4)^{2n-2}}\bigl((W(\WRiv,u_1) +W(\WRii,u_1))
Z_-(u_2,\ldots,u_n,u_1)\\
+(W(\WRi,u_1)+W(\WRiii,u_1))Z_+(u_2,\ldots,u_n,u_1)\bigr),\end{multline*}
from which the RHS of~\eqref{spec3} follows.

For $n=3$, the proof can be illustrated as follows:
\psset{unit=0.8mm}
\begin{align*}&Z(\ui,\uii,\uiii,\uiv)|_{\ui\uiv=q^2}\\[1.5mm]
&\makebox[19mm]{$=$}\raisebox{-16.5mm}{\pspicture(-4,-2)(60,42)
\psline[linewidth=0.8pt,linecolor=blue](0,40)(0,30)(60,30)(60,40)\psline[linewidth=0.8pt,linecolor=brwn](30,0)(30,40)
\psline[linewidth=0.8pt,linecolor=green](10,40)(10,20)(50,20)(50,40)
\psline[linewidth=0.8pt,linecolor=red](20,40)(20,10)(40,10)(40,40)
\multirput(0,36)(10,0){7}{\psdots[dotstyle=triangle*,dotscale=1.2](0,0)}
\multirput(0,40)(10,0){7}{$\ss\bullet$}\multirput(0,30)(10,0){7}{$\ss\bullet$}\multirput(10,20)(10,0){5}{$\ss\bullet$}
\multirput(20,10)(10,0){3}{$\ss\bullet$}\rput(30,0){$\ss\bullet$}\endpspicture}\\
&\makebox[19mm]{$\stackrel{\eqref{Wred}}{=}$}\raisebox{-16.5mm}{\pspicture(-1,-2)(60,42)
\psline[linewidth=0.8pt,linecolor=blue](0,40)(0,30)(29,30)
\psline[linewidth=0.8pt,linecolor=blue](40,30)(60,30)(60,40)\psline[linewidth=0.8pt,linecolor=blue,linearc=3](32,28)(34,30)(40,30)
\psline[linewidth=0.8pt,linecolor=brwn,linearc=3](30,0)(30,26)(32,28)
\psline[linewidth=0.8pt,linecolor=green](10,40)(10,20)(50,20)(50,40)\psline[linewidth=0.8pt,linecolor=red](20,40)(20,10)(40,10)(40,40)
\multirput(0,36)(10,0){3}{\psdots[dotstyle=triangle*,dotscale=1.2](0,0)}\multirput(40,36)(10,0){3}{\psdots[dotstyle=triangle*,dotscale=1.2](0,0)}
\rput(25,30){\psdots[dotstyle=triangle*,dotscale=1.2,dotangle=-90](0,0)}
\multirput(0,40)(10,0){3}{$\ss\bullet$}\multirput(0,30)(10,0){3}{$\ss\bullet$}\rput(29,30){$\ss\bullet$}
\multirput(40,40)(10,0){3}{$\ss\bullet$}\multirput(40,30)(10,0){3}{$\ss\bullet$}\multirput(10,20)(10,0){5}{$\ss\bullet$}
\multirput(20,10)(10,0){3}{$\ss\bullet$}\rput(30,0){$\ss\bullet$}\endpspicture}\\
&\makebox[19mm]{$=$}\frac{\sigma(q\ui)\,\sigma(q^2\ui\uii)\,\sigma(q^2\ui\uiii)}{\sigma(q)\,\sigma(q^4)^2}
\raisebox{-16.5mm}{\pspicture(6,-2)(60,42)\psline[linewidth=0.8pt,linecolor=blue,linearc=3](32,28)(34,30)(40,30)
\psline[linewidth=0.8pt,linecolor=blue](40,30)(60,30)(60,40)
\psline[linewidth=0.8pt,linecolor=brwn,linearc=3](30,0)(30,26)(32,28)
\psline[linewidth=0.8pt,linecolor=green](10,30)(10,20)(50,20)(50,40)\psline[linewidth=0.8pt,linecolor=red](20,30)(20,10)(40,10)(40,40)
\multirput(10,26)(10,0){2}{\psdots[dotstyle=triangle*,dotscale=1.2](0,0)}\multirput(40,36)(10,0){3}{\psdots[dotstyle=triangle*,dotscale=1.2](0,0)}
\multirput(40,40)(10,0){3}{$\ss\bullet$}\multirput(10,30)(10,0){2}{$\ss\bullet$}\multirput(40,30)(10,0){3}{$\ss\bullet$}\multirput(10,20)(10,0){5}{$\ss\bullet$}
\multirput(20,10)(10,0){3}{$\ss\bullet$}\rput(30,0){$\ss\bullet$}\endpspicture}\\
&\makebox[19mm]{$=$}\frac{\sigma(q\ui)\,\sigma(q^2\ui\uii)\,\sigma(q^2\ui\uiii)}{\sigma(q)\,\sigma(q^4)^2}
\raisebox{-12.5mm}{\pspicture(10,8)(60,42)
\psline[linewidth=0.8pt,linecolor=blue,linearc=2](24,10)(26,10)(38,30)(60,30)\psline[linewidth=0.8pt,linecolor=blue](60,30)(60,40)
\psline[linewidth=0.8pt,linecolor=green](14,40)(14,30)\psline[linewidth=0.8pt,linecolor=green,linearc=2](14,30)(32,30)(38,20)(50,20)
\psline[linewidth=0.8pt,linecolor=green](50,20)(50,40)
\psline[linewidth=0.8pt,linecolor=red](24,40)(24,20)\psline[linewidth=0.8pt,linecolor=red,linearc=2](24,20)(26,20)(32,10)(40,10)
\psline[linewidth=0.8pt,linecolor=red](40,10)(40,40)
\rput[r](33.4,25){$\ss v$}\rput[r](27.4,15){$\ss w$}
\multirput(14,36)(10,0){2}{\psdots[dotstyle=triangle*,dotscale=1.2](0,0)}\multirput(40,36)(10,0){3}{\psdots[dotstyle=triangle*,dotscale=1.2](0,0)}
\multirput(14,40)(10,0){2}{$\ss\bullet$}\multirput(40,40)(10,0){3}{$\ss\bullet$}
\multirput(14,30)(10,0){2}{$\ss\bullet$}\multirput(40,30)(10,0){3}{$\ss\bullet$}
\rput(24,20){$\ss\bullet$}\multirput(40,20)(10,0){2}{$\ss\bullet$}
\multirput(29,15)(6,10){2}{$\ss\bullet$}\multirput(24,10)(16,0){2}{$\ss\bullet$}
\rput[l](65,29){[with $v=\uii\uiv=q^2\ubi\uii$,}
\rput[l](67,21){$w=\uiii\uiv=q^2\ubi\uiii$]}\endpspicture}\\
&\makebox[19mm]{$\stackrel{\mathrm{HYBE,\,RRE}}{=}$}\frac{\sigma(q\ui)\,\sigma(q^2\ui\uii)\,\sigma(q^2\ui\uiii)}{\sigma(q)\,\sigma(q^4)^2}
\raisebox{-12.5mm}{\pspicture(16,8)(60,46)
\psline[linewidth=0.8pt,linecolor=blue,linearc=2](30,10)(32,10)(38,20)(50,20)\psline[linewidth=0.8pt,linecolor=blue,linearc=2](50,20)(50,32)(60,38)(60,44)
\psline[linewidth=0.8pt,linecolor=green](20,40)(20,30)(60,30)\psline[linewidth=0.8pt,linecolor=green,linearc=2](60,30)(60,32)(50,38)(50,44)
\psline[linewidth=0.8pt,linecolor=red](30,40)(30,20)\psline[linewidth=0.8pt,linecolor=red,linearc=2](30,20)(32,20)(38,10)(40,10)
\psline[linewidth=0.8pt,linecolor=red](40,10)(40,40)
\rput[b](55,31.5){$\ss v$}\rput[r](33.4,15){$\ss w$}
\multirput(20,36)(10,0){3}{\psdots[dotstyle=triangle*,dotscale=1.2](0,0)}\multirput(50,41)(10,0){2}{\psdots[dotstyle=triangle*,dotscale=1.2](0,0)}
\multirput(50,44)(10,0){2}{$\ss\bullet$}\multirput(20,40)(10,0){3}{$\ss\bullet$}\rput(55,35){$\ss\bullet$}
\multirput(20,30)(10,0){5}{$\ss\bullet$}\multirput(30,20)(10,0){3}{$\ss\bullet$}
\multirput(30,10)(10,0){2}{$\ss\bullet$}\rput(35,15){$\ss\bullet$}\endpspicture}\\
&\makebox[19mm]{$\stackrel{\mathrm{RRE,\,VYBE}}{=}$}\frac{\sigma(q\ui)\,\sigma(q^2\ui\uii)\,\sigma(q^2\ui\uiii)}{\sigma(q)\,\sigma(q^4)^2}
\raisebox{-12.5mm}{\pspicture(16,5)(60,52)
\psline[linewidth=0.8pt,linecolor=green](20,40)(20,30)(60,30)
\psline[linewidth=0.8pt,linecolor=green,linearc=2](60,30)(60,38)(50,44)(50,50)
\psline[linewidth=0.8pt,linecolor=red](30,40)(30,20)(50,20)
\psline[linewidth=0.8pt,linecolor=red,linearc=2](50,20)(50,32)(40,38)(40,44)
\psline[linewidth=0.8pt,linecolor=blue](30,10)(40,10)(40,30)\psline[linewidth=0.8pt,linecolor=blue,linearc=2](40,30)(40,32)(60,44)(60,50)
\rput[b](45,31.5){$\ss w$}\rput[b](55,37.5){$\ss v$}
\multirput(20,36)(10,0){2}{\psdots[dotstyle=triangle*,dotscale=1.2](0,0)}
\rput(40,41){\psdots[dotstyle=triangle*,dotscale=1.2](0,0)}\multirput(50,47)(10,0){2}{\psdots[dotstyle=triangle*,dotscale=1.2](0,0)}
\multirput(50,50)(10,0){2}{$\ss\bullet$}\rput(40,44){$\ss\bullet$}\multirput(20,40)(10,0){2}{$\ss\bullet$}\multirput(45,35)(10,6){2}{$\ss\bullet$}
\multirput(20,30)(10,0){5}{$\ss\bullet$}\multirput(30,20)(10,0){3}{$\ss\bullet$}\multirput(30,10)(10,0){2}{$\ss\bullet$}\endpspicture}\\ 
&\makebox[19mm]{$=$}\frac{\sigma(q\ui)\,\sigma(q^2\ui\uii)\,\sigma(q^2\uii\uiv)\,\sigma(q^2\ui\uiii)\,\sigma(q^2\uiii\uiv)}{\sigma(q)\,\sigma(q^4)^4}
\raisebox{-14mm}{\pspicture(-4,-5)(40,34)\psline[linewidth=0.8pt,linecolor=green](0,30)(0,20)(40,20)(40,30)
\psline[linewidth=0.8pt,linecolor=red](10,30)(10,10)(30,10)(30,30)
\psline[linewidth=0.8pt,linecolor=blue](10,0)(20,0)(20,30)\multirput(0,26)(10,0){5}{\psdots[dotstyle=triangle*,dotscale=1.2](0,0)}
\multirput(0,30)(10,0){5}{$\ss\bullet$}\multirput(0,20)(10,0){5}{$\ss\bullet$}\multirput(10,10)(10,0){3}{$\ss\bullet$}
\multirput(10,0)(10,0){2}{$\ss\bullet$}\endpspicture}\\
&\makebox[19mm]{$=$}\frac{\sigma(q\ui)\,\sigma(q^2\ui\uii)\,\sigma(q^2\uii\uiv)\,\sigma(q^2\ui\uiii)\,\sigma(q^2\uiii\uiv)}{\sigma(q)\,\sigma(q^4)^4}\,
\Bigl(\frac{\sigma(q\ubi)}{\sigma(q)}+1\Bigr)\\[0mm]
&\hspace{97mm}\times\bigl(Z_-(\uii,\uiii,\ui)+Z_+(\uii,\uiii,\ui)\bigr)\\[3.5mm]
&\makebox[19mm]{$=$}\frac{\sigma(q\ui)\,(\sigma(q\ubi)+\sigma(q))\,\sigma(q^2\ui\uii)\,\sigma(q^2\uii\uiv)\,
\sigma(q^2\ui\uiii)\,\sigma(q^2\uiii\uiv)}{\sigma(q)^2\,\sigma(q^4)^4}\,Z(\uii,\uiii,\ui),\end{align*}
where the notation is the same as in the examples in the proofs of 
Propositions~\ref{symm} and~\ref{spec2prop}.  Note that in proceeding from the fourth to the fifth
lines above, the only change is that the graph is redrawn and partly recolored (where the recoloring
is justified by the existence of a relation between~$u_1$ and~$u_4$). This is done so that the applications of the 
horizontal form of the Yang--Baxter equation~\eqref{YBE} in the next step can be visualized more easily.\end{proof}

\begin{proposition}\label{spec4prop}
The odd-order DASASM partition function satisfies
\begin{equation}\label{spec4}
Z(u_1,\ldots,u_{n-1},q^2,1)=0.\end{equation}
\end{proposition}
Note that Proposition~\ref{spec4prop} can be combined with Propositions~\ref{recipth}--\ref{symm} to 
give $Z(u_1,\ldots,u_n,1)=0$ at $u_i=q^{\pm2}$ and $u_i=-q^{\pm2}$, for $i=1,\ldots,n$.
\begin{proof}
This result can be obtained by using Proposition~\ref{spec3prop} (to give $Z(q^2,u_2,\ldots,u_n,1)=0$),
and Proposition~\ref{symm} (to interchange $u_1$ and $u_n$ in $Z(u_1,\ldots,u_n,1)$).

Alternatively, it can be obtained directly, so this approach will also be given here.
Let $v$ denote the second-last vertex in the central column of the graph $\mathcal{T}_n$, i.e., $v=(n,n+1)$.
Now observe that the set of configurations of the six-vertex model on~$\mathcal{T}_n$ (for $n\ge1$) can be partitioned
into subsets of size three (which, incidentally, implies that $|\DASASM(2n+1)|$ is divisible by~3, for $n\ge1$), 
such that the configurations within each subset have the same orientations on all edges, except 
for those incident with $v$ from the left, below and the right.
In particular, if the edge incident with $v$ from above has orientation~$a$,
then the orientations $l$, $b$ and $r$ of the edges incident with $v$ from the left, below and the right, respectively,
are $(l,b,r)=$ $(a,\tilde{a},\tilde{a})$, $(\tilde{a},\tilde{a},a)$ or $(\tilde{a},a,\tilde{a})$
(where orientations are taken as in or out, with respect to $v$, and $\tilde{a}$ is the reversal of~$a$).
It can now be seen, using~\eqref{Z} and applying~\eqref{Wred} to the bulk weight of $v$, that the contribution to $Z(u_1,\ldots,u_{n-1},q^2,1)$ 
of the case $(a,\tilde{a},\tilde{a})$ is zero, 
while the contributions of the cases $(\tilde{a},\tilde{a},a)$ and $(\tilde{a},a,\tilde{a})$ cancel, since
the only vertex at which their weights differ is $(n,n+2)$, with the right boundary 
weights for that vertex being~$1$ and $\sigma(q\,\q^2)/\sigma(q)=-1$. The result~\eqref{spec4} now follows
immediately.
\end{proof}

\subsection{Proof of Theorem~\ref{DASASMFulldetTh}}\label{DASASMFulldetThPr}
Theorem~\ref{DASASMFulldetTh} will now be proved, using results from Sections~\ref{GenProp} and~\ref{Spec}.

Consider a family of functions $X(u_1,\ldots,u_{n+1})$ which satisfy the following properties.
\begin{enumerate}
\item $X(u_1)=1$.
\item $X(u_1,\ldots,u_{n+1})$ is a Laurent polynomial in $u_{n+1}$ of lower degree at least~$-n$ and upper degree at most $n$.
\item If $u_1u_{n+1}=q^2$, then\\
$\displaystyle X(u_1,\ldots,u_{n+1})=
\frac{\sigma(qu_1)\,(\sigma(q\u_1)+\sigma(q))\prod_{i=2}^n\sigma(q^2u_1u_i)\sigma(q^2u_iu_{n+1})}{\sigma(q)^2\,\sigma(q^4)^{2n-2}}
\,X(u_2,\ldots,u_n,u_1).$
\item $X(\u_1,\ldots,\u_{n+1})=X(u_1,\ldots,u_{n+1})$.
\item $X(u_1,\ldots,u_{n+1})$ is even in $u_i$, for each $i=1,\ldots,n$.
\item $X(u_1,\ldots,u_{n+1})$ is symmetric in $u_1,\ldots,u_n$.
\end{enumerate}

It will first be shown that $X(u_1,\ldots,u_{n+1})$ is uniquely determined by these properties.

By~(ii), $X(u_1,\ldots,u_{n+1})$ is uniquely determined if it is known at $2n+1$ values of $u_{n+1}$.
It can be seen that combining~(iii) with (iv)--(vi) leads to expressions for $X(u_1,\ldots,u_{n+1})$
at~$4n$ values of $u_{n+1}$, i.e., at $u_{n+1}=q^{\pm2}\u_i$ and $u_{n+1}=-q^{\pm2}\u_i$, for each $i=1,\ldots,n$.
In particular,~(iii) consists of an expression for $X(u_1,\ldots,u_n,q^2\u_1)$ in terms of 
$X(u_2,\ldots,u_n,u_1)$.  Replacing $u_1,\ldots,u_n$ by $\u_1,\ldots,\u_n$, and applying~(iv), 
then gives an expression for $X(u_1,\ldots,u_n,\q^2\u_1)$ in terms of $X(u_2,\ldots,u_n,u_1)$.
Replacing $u_1$ by $-u_1$ in the previous two cases, and applying~(v), then gives expressions
for $X(u_1,\ldots,u_n,-q^{\pm2}\u_1)$ in terms of $X(u_2,\ldots,u_n,-u_1)$.  Finally, interchanging
$u_1$ and $u_i$ for any $i\in\{2,\ldots,n\}$ in the previous four cases, and applying (vi), 
gives expressions for $X(u_1,\ldots,u_n,q^{\pm2}\u_i)$ in terms of $X(u_1,\ldots,u_{i-1},u_{i+1},\ldots,u_n,u_i)$,
and for $X(u_1,\ldots,u_n,-q^{\pm2}\u_i)$ in terms of $X(u_1,\ldots,u_{i-1},u_{i+1},\ldots,u_n,-u_i)$.

Since $X(u_1)$ is fully known by (i),
it follows by induction on $n$ that $X(u_1,\ldots,u_{n+1})$ is known at $2n+1$ (or, in fact, $4n$) values of~$u_{n+1}$.
Therefore, $X(u_1,\ldots,u_{n+1})$ is uniquely determined by properties (i)--(vi).

The validity of \eqref{DASASMFulldet} in Theorem~\ref{DASASMFulldetTh} will now be verified by
showing that both sides satisfy properties (i)--(vi).

If $X(u_1,\ldots,u_{n+1})$ is taken to be the odd-order DASASM partition function $Z(u_1,\ldots,u_{n+1})$,
i.e., the LHS of~\eqref{DASASMFulldet}, then the required properties are satisfied, since (i) follows from~\eqref{Z} with $n=0$,
while (ii)--(vi) follow from Propositions~\ref{recipth},~\ref{degth},~\ref{symm} and~\ref{spec3prop}.

For the remainder of this section, let $X(u_1,\ldots,u_{n+1})$ be the RHS of~\eqref{DASASMFulldet},
and define
\begin{align}
\notag F(u_1,\ldots,u_{n+1})&=\frac{\sigma(q^2)^n}{\sigma(q)^{2n}\,\sigma(q^4)^{n^2}}
\prod_{i=1}^n\frac{\sigma(u_i)\sigma(qu_i)\sigma(q\u_i)\sigma(q^2u_iu_{n+1})\sigma(q^2\u_i\u_{n+1})}{\sigma(u_i\u_{n+1})}\\
\notag&\hspace{60mm}\times\prod_{1\le i<j\le n}\biggl(\frac{\sigma(q^2u_iu_j)\sigma(q^2\u_i\u_j)}{\sigma(u_i\u_j)}\biggr)^2,\\
\notag D(u_1,\ldots,u_{n+1})&=\det_{1\le i,j\le n+1}\left(\begin{cases}\frac{q^2+\q^2+u_i^2+\u_j^2}{\sigma(q^2u_iu_j)\,
\sigma(q^2\u_i\u_j)},&i\le n\\[1.5mm]
\frac{1}{u_j^2-1},&i=n+1\end{cases}\right),\\
\notag Y(u_1,\ldots,u_{n+1})&=(u_{n+1}-1)\,F(u_1,\ldots,u_{n+1})\,D(u_1,\ldots,u_{n+1}),\\
\label{FDY}\widehat{Y}(u_1,\ldots,u_{n+1})&=(u_{n+1}^2-1)\,F(u_1,\ldots,u_{n+1})\,D(u_1,\ldots,u_{n+1}),\end{align}
so that
\begin{align}\notag X(u_1,\ldots,u_{n+1})&=Y(u_1,\ldots,u_{n+1})\,+\,Y(\u_1,\ldots,\u_{n+1})\\
\label{YY}&=\frac{\widehat{Y}(u_1,\ldots,u_{n+1})}{u_{n+1}+1}\,+\,\frac{\widehat{Y}(\u_1,\ldots,\u_{n+1})}{\u_{n+1}+1}.\end{align}

Then $F(u_1)=1$, $D(u_1)=\frac{1}{u_1^2-1}$ and $Y(u_1)=\frac{1}{u_1+1}$, from which it follows that~(i) is satisfied.

Proceeding to~(ii), it can be seen that $(u_{n+1}^2-1)
\prod_{i=1}^n\sigma(q^2u_iu_{n+1})\sigma(q^2\u_i\u_{n+1})\,D(u_1,\ldots,u_{n+1})$ is an (even) Laurent polynomial in $u_{n+1}$
of lower degree at least~$-2n$ and upper degree at most~$2n$.
Also, $D(u_1,\ldots,u_n,\pm u_i)=0$ for each $i=1,\ldots,n$, since
columns $i$ and $n+1$ of the matrix are then equal. It follows that
$(u_{n+1}^2-1)\prod_{i=1}^n\sigma(q^2u_iu_{n+1})\sigma(q^2\u_i\u_{n+1})/
\sigma(u_i\u_{n+1})\,D(u_1,\ldots,u_{n+1})$,
and hence also $\widehat{Y}(u_1,\ldots,u_{n+1})$, is a Laurent polynomial in $u_{n+1}$ of lower degree at least~$-n$ and upper degree at most $n$.
Now note that if a function $f(x)$ is a Laurent polynomial of lower degree at least~$-n$ and upper degree at most $n$,
then so is the function $g(x)=f(x)/(x+1)+f(\bar{x})/(\bar{x}+1)$, since $g(x)=(f(x)+xf(\bar{x}))/(x+1)$, where $f(x)+xf(\bar{x})$
is a Laurent polynomial of lower degree at least~$-n$ and upper degree at most $n+1$ which vanishes at $x=-1$.
Therefore, it follows from the previous observations and the last expression of~\eqref{YY}
that (ii) is satisfied.

Proceeding to~(iii), consider $Y(u_1^{\pm1},\ldots,u_{n+1}^{\pm1})$, and multiply the first row of the matrix in
$D(u_1^{\pm1},\ldots,u_{n+1}^{\pm1})$
by the factor $\sigma(q^2\u_1\u_{n+1})$ from $F(u_1,\ldots,u_{n+1})$ ($=F(\u_1,\ldots,\u_{n+1})$).  Now set $u_1u_{n+1}=q^2$,
which leads to the first row becoming
$(0,\ldots,0,(q^2+\q^2+u_1^{\pm2}+q^{\mp4}u_1^{\pm2})/\sigma(q^4))$.  It then follows, after rearranging
and canceling certain products and signs, that for $u_1u_{n+1}=q^2$,
\begin{equation*}Y(u_1^{\pm1},\ldots,u_{n+1}^{\pm1})=
\frac{\sigma(qu_1)\,(\sigma(q\u_1)+\sigma(q))\prod_{i=2}^n\sigma(q^2u_1u_i)\sigma(q^2u_iu_{n+1})}{\sigma(q)^2\,\sigma(q^4)^{2n-2}}
\,Y(u_2^{\pm1},\ldots,u_n^{\pm1},u_1^{\pm1}).\end{equation*}
Therefore, (iii) is satisfied, due to the first equation of~\eqref{YY}.

Proceeding to the remaining properties,~(iv) is satisfied due to~\eqref{YY},~(v) is satisfied since
$F(u_1,\ldots,u_{n+1})$ and all entries of the matrix in $D(u_1,\ldots,u_{n+1})$ are even in $u_i$, for each $i=1,\ldots,n$,
and~(vi) is satisfied since if $u_i$ and $u_j$ are interchanged, then
$F(u_1,\ldots,u_{n+1})$ is unchanged, while rows $i$ and $j$, and columns $i$ and $j$, are interchanged in the matrix in
$D(u_1,\ldots,u_{n+1})$, for all distinct $i,j\in\{1,\ldots,n\}$.

\subsection{Alternative proofs of Theorem~\ref{DASASMFulldetTh} and Corollary~\ref{DASASMdetCoroll}}\label{DASASMFulldetThPrAlt}
Alternative proofs of Theorem~\ref{DASASMFulldetTh} and Corollary~\ref{DASASMdetCoroll} will now be sketched.  Essentially,
these involve regarding the odd-order DASASM partition function $Z(u_1,\ldots,u_{n+1})$ as a Laurent polynomial in $u_1$,
rather than a Laurent polynomial in $u_{n+1}$, as done in Section~\ref{DASASMFulldetThPr}.

Consider a family of functions $X(u_1,\ldots,u_{n+1})$ which satisfy properties (i) and (iii)--(vi) of Section~\ref{DASASMFulldetThPr},
together with the following properties.
\begin{enumerate}
\setcounter{enumi}{6}
\item $X(u_1,u_2)$ is given by the RHS of~\eqref{Z1}.
\item $X(u_1,\ldots,u_{n+1})$ is a Laurent polynomial in $u_1$ of lower degree at least~$-2n$ and upper degree at most $2n$.
\item $\displaystyle X(q,u_2,\ldots,u_{n+1})=
\frac{(q+\q)\prod_{i=2}^n\sigma(q^3u_i)^2\,\sigma(q^3u_{n+1})}{\sigma(q^4)^{2n-1}}\,X(u_2,\ldots,u_{n+1}).$
\item If $u_1u_2=q^2$, then\\
$\displaystyle X(u_1,u_2,\ldots,u_{n+1})=\\
\hspace*{10mm}\frac{\sigma(u_1)\sigma(qu_1)\sigma(u_2)\sigma(qu_2)\prod_{i=3}^n\bigl(\sigma(q^2u_1u_i)\sigma(q^2u_2u_i)\bigr)^2\,
\sigma(q^2u_1u_{n+1})\sigma(q^2u_2u_{n+1})}{\sigma(q)^4\,\sigma(q^4)^{2(2n-3)}}\\
\hspace*{125mm}\times X(u_3,\ldots,u_{n+1}).$
\end{enumerate}

It follows, using an argument similar to that in the first part of Section~\ref{DASASMFulldetThPr},
that $X(u_1,\ldots,u_{n+1})$ is uniquely determined by these properties.  In particular,
by~(v) (with $i=1$) and (viii), $X(u_1,\ldots,u_{n+1})$ is uniquely determined if it is known at $2n+1$ values of $u_1^2$.
Combining~(iii),~(ix) and~(x) with~(iv) and~(vi) gives expressions for $X(u_1,\ldots,u_{n+1})$
at~$2n+2$ values of $u_1^2$, i.e., at $u_1^2=q^{\pm2}$ and $u_1^2=q^{\pm4}\u_i^2$, for $i=2,\ldots,n+1$.
Using these expressions, and applying induction on~$n$ with~(i) and~(vii),
it follows that $X(u_1,\ldots,u_{n+1})$ is known at $2n+1$ (or, in fact, $2n+2$) values of~$u_1^2$, as required.

If $X(u_1,\ldots,u_{n+1})$ is taken to be the odd-order DASASM partition function $Z(u_1,\ldots,u_{n+1})$,
then as already found in Section~\ref{DASASMFulldetThPr},~(i) and (iii)--(vi) are satisfied.
Furthermore, (vii)--(x) are satisfied, since (vii) is given by~\eqref{Z1},
while (viii)--(x) follow from Propositions~\ref{degth},~\ref{spec1prop} and~\ref{spec2prop}.

If $X(u_1,\ldots,u_{n+1})$ is now taken to be the RHS of~\eqref{DASASMFulldet},
then as already found in Section~\ref{DASASMFulldetThPr},~(i) and (iii)--(vi) are satisfied.
It can also be shown, using arguments similar to those of Section~\ref{DASASMFulldetThPr},
that (vii)--(x) are satisfied, from which the required equality
in Theorem~\ref{DASASMFulldetTh} then follows.

Note that when showing that~(viii) is satisfied by the RHS of~\eqref{DASASMFulldet},
it can be verified straightforwardly that the function
\begin{multline*}P(u_1,\ldots,u_{n+1})=\\
\frac{\sigma(u_1)\sigma(qu_1)\sigma(q\u_1)\sigma(q^2u_1u_{n+1})\sigma(q^2\u_1\u_{n+1})
\prod_{i=2}^n(\sigma(q^2u_1u_i)\sigma(q^2\u_1\u_i))^2\,D(u_1,\ldots,u_{n+1})}{\sigma(u_1\u_{n+1})}\end{multline*}
is a Laurent polynomial in~$u_1$.
However, it then needs to be shown that $P(u_1,\ldots,u_{n+1})/$ $\prod_{i=2}^n\sigma(u_1\u_i)^2$ is also a Laurent polynomial in~$u_1$.
This can be done by introducing
\begin{multline*}\label{LP2}P'(v_1,\ldots,v_n;u_1,\ldots,u_{n+1})=\\
\frac{\sigma(u_1)\sigma(q^2v_1u_1)\sigma(q^2\bar{v}_1\u_1)
\prod_{i=2}^{n+1}\sigma(q^2v_1u_i)\sigma(q^2\bar{v}_1\u_i)\prod_{i=2}^n\sigma(q^2v_iu_1)\sigma(q^2\bar{v}_i\u_1)}
{\sigma(u_1\u_{n+1})}\\
\times\det_{1\le i,j\le n+1}\left(\begin{cases}\frac{q^2+\q^2+v_i^2+\u_j^2}
{\sigma(q^2v_iu_j)\,\sigma(q^2\bar{v}_i\u_j)},&i\le n\\[1.5mm]
\frac{1}{u_j^2-1},&i=n+1\end{cases}\right).\end{multline*}
It can then be checked that $P'(v_1,\ldots,v_n;u_1,\ldots,u_{n+1})$ is a Laurent polynomial in~$u_1$ and~$v_1$, which
vanishes at $u_1=\pm u_i$ and $v_1=\pm v_i$ for each $i=2,\ldots,n$,
so that $P'(v_1,\ldots,v_n;u_1,\ldots,u_{n+1})/$ $\prod_{i=2}^n\sigma(u_1\u_i)\sigma(v_1\bar{v}_i)$ is also
a Laurent polynomial in~$u_1$ and~$v_1$.
Furthermore, it can be checked that  $P(u_1,\ldots,u_{n+1})=P'(u_1,\ldots,u_n;u_1,\ldots,u_{n+1})/((qu_1+\q\u_1)(q\u_1+\q u_1))$
and that $P'(u_1,\ldots,u_n;$ $u_1,\ldots,u_{n+1})$ vanishes at $u_1^2=-q^{\pm2}$, from which the required result follows.

Proceeding to the alternative proof of Corollary~\ref{DASASMdetCoroll}, 
let $X(u_1,\ldots,u_n,1)$ satisfy properties (i) and (iv)-(vi) of Section~\ref{DASASMFulldetThPr},
and properties (vii)--(x) of this section, with $u_{n+1}$ set to~$1$ in each case, together
with a modification of property (iii) of Section~\ref{DASASMFulldetThPr} given by
\begin{itemize}
\item[(iii$'$)] $X(u_1,\ldots,u_{n-1},q^2,1)=0$.
\end{itemize}
Then, using the same argument as used for the alternative proof of Theorem~\ref{DASASMFulldetTh},
$X(u_1,\ldots,u_n,1)$ is uniquely determined.  In particular, combining~(iii$'$),~(ix) and~(x) with~(iv) and~(vi) 
gives expressions for $X(u_1,\ldots,u_n,1)$
at~$2n+2$ values of $u_1^2$, i.e., at $u_1^2=q^{\pm2}$, $u_1^2=q^{\pm4}$ and $u_1^2=q^{\pm4}\u_i^2$, for $i=2,\ldots,n$.

If $X(u_1,\ldots,u_n,1)$ is taken to be either the LHS or RHS of~\eqref{DASASMdet}, then, as found previously 
for the case of arbitrary $u_{n+1}$,~(i) and (iv)--(x) are satisfied.
Furthermore, (iii$'$) is satisfied by the LHS due to Proposition~\ref{spec4prop},
and by the RHS due to the presence of the term $\sigma(q^2\u_n)$ on the RHS.
The required equality in Corollary~\ref{DASASMdetCoroll} now follows.

Note that this can be regarded as a shorter proof of Corollary~\ref{DASASMdetCoroll} than that of Sections~\ref{MainResults} and~\ref{DASASMFulldetThPr} 
since, with regards to the LHS of~\eqref{DASASMdet}, Proposition~\ref{spec3prop} has been replaced by 
Propositions~\ref{spec1prop},~\ref{spec2prop} and~\ref{spec4prop}, each of which has a simpler
proof than Proposition~\ref{spec3prop}, and with regards to the RHS of~\eqref{DASASMdet},
computations involving the more complicated RHS of~\eqref{DASASMFulldet} have now been avoided.

\subsection{Proof of Theorem~\ref{ZDASASMFullSchurTh}}\label{ZDASASMFullSchurPr}
Theorem~\ref{ZDASASMFullSchurTh} will now be proved, using Theorem~\ref{DASASMFulldetTh}.

In particular, the following determinant identity of Okada~\cite[Thm.~4.2]{Oka98} will be applied to~\eqref{DASASMFulldet} at $q=e^{i\pi/6}$.
For all $a_1,x_1,b_1,y_1,\ldots,a_k,x_k,b_k,y_k$,
\begin{equation}\label{OkadaId}
\det_{1\le i,j\le k}\biggl(\frac{a_i-b_j}{x_i-y_j}\biggr)=\frac{(-1)^{k(k+1)/2}}{\prod_{i,j=1}^k(x_i-y_j)}\:
\det\begin{pmatrix}
1&a_1&x_1&a_1x_1&\!\ldots\!&x_1^{k-1}&a_1x_1^{k-1}\\
1&b_1&y_1&b_1y_1&\!\ldots\!&y_1^{k-1}&b_1y_1^{k-1}\\
\vdots&\vdots&\vdots&\vdots&\!\ddots\!&\vdots&\vdots\\
1&a_k&x_k&a_kx_k&\!\ldots\!&x_k^{k-1}&a_kx_k^{k-1}\\
1&b_k&y_k&b_ky_k&\!\ldots\!&y_k^{k-1}&b_ky_k^{k-1}\end{pmatrix}.\end{equation}

Let $Y(u_1,\ldots,u_{n+1})$ be defined as in~\eqref{FDY}, so that
\begin{equation}\label{DASASMFulldet1}Z(u_1,\ldots,u_{n+1})=
Y(u_1,\ldots,u_{n+1})+Y(\u_1,\ldots,\u_{n+1}).\end{equation}
Setting $q=e^{i\pi/6}$ gives
\begin{multline}\label{combdet1}Y(u_1,\ldots,u_{n+1})\big|_{q=e^{i\pi/6}}=
\frac{(u_{n+1}-1)\,u_{n+1}^n}{3^{n(n-1)/2}}
\prod_{i=1}^n\frac{(u_i^2-1)(u_i^2-1+\u_i^2)(u_i^2u_{n+1}^2+1+\u_i^2\u_{n+1}^2)}{u_i^2-u_{n+1}^2}\\[1.5mm]
\times\,\prod_{1\le i<j\le n}\frac{(u_i^2u_j^2+1+\u_i^2\u_j^2)^2}{(u_i^2-u_j^2)(\u_i^2-\u_j^2)}\;
\det_{1\le i,j\le n+1}\left(\begin{cases}\frac{u_i^2+1+\u_j^2}{u_i^2u_j^2+1+\u_i^2\u_j^2},&i\le n\\[1.5mm]
\frac{1}{u_j^2-1},&i=n+1\end{cases}\right).\end{multline}
Now observe that
\begin{equation*}\frac{u_i^2+1+\u_j^2}{u_i^2u_j^2+1+\u_i^2\u_j^2}=
u_i^2\u_j^2\,\frac{u_i^2+u_i^4-(\u_j^2+\u_j^4)}{u_i^6-\u_j^6},\qquad
\frac{1}{u_j^2-1}=-\u_j^2\:\frac{-1-(\u_j^2+\u_j^4)}{1-\u_j^6},\end{equation*}
and apply~\eqref{OkadaId} to~\eqref{combdet1} with
\begin{equation*}a_i=\begin{cases}u_i^2+u_i^4,&i\le n,\\-1,&i=n+1,\end{cases}\quad
x_i=\begin{cases}u_i^6,&i\le n,\\1,&i=n+1,\end{cases}\quad
b_j=\u_j^2+\u_j^4,\quad
y_j=\u_j^6,\quad k=n+1,\end{equation*}
to give
\begin{multline}\label{combdet2}Y(u_1,\ldots,u_{n+1})\big|_{q=e^{i\pi/6}}=
\frac{(-1)^{n(n-1)/2}}{u_{n+1}^2\,\prod_{i=1}^n\prod_{j=1}^{n+1}(u_i^6-\u_j^6)\,\prod_{i=1}^{n+1}(1-\u_i^6)}\\
\times\,\frac{(u_{n+1}-1)\,u_{n+1}^n}{3^{n(n-1)/2}}
\prod_{i=1}^n\frac{(u_i^2-1)(u_i^2-1+\u_i^2)(u_i^2u_{n+1}^2+1+\u_i^2\u_{n+1}^2)}{u_i^2-u_{n+1}^2}
\prod_{1\le i<j\le n}\frac{(u_i^2u_j^2+1+\u_i^2\u_j^2)^2}{(u_i^2-u_j^2)(\u_i^2-\u_j^2)}\\
\times\,\det\begin{pmatrix}
1&u_1^2+u_1^4&u_1^6&u_1^8+u_1^{10}&\!\ldots\!&u_1^{6n}&u_1^{6n+2}+u_1^{6n+4}\\
1&\u_1^2+\u_1^4&\u_1^6&\u_1^8+\u_1^{10}&\!\ldots\!&\u_1^{6n}&\u_1^{6n+2}+\u_1^{6n+4}\\
\vdots&\vdots&\vdots&\vdots&\!\ddots\!&\vdots&\vdots\\
1&u_n^2+u_n^4&u_n^6&u_n^8+u_n^{10}&\!\ldots\!&u_n^{6n}&u_n^{6n+2}+u_n^{6n+4}\\
1&\u_n^2+\u_n^4&\u_n^6&\u_n^8+\u_n^{10}&\!\ldots\!&\u_n^{6n}&\u_n^{6n+2}+\u_n^{6n+4}\\
1&-1&1&-1&\!\ldots\!&1&-1\\
1&\u_{n+1}^2+\u_{n+1}^4&\u_{n+1}^6&\u_{n+1}^8+\u_{n+1}^{10}&\!\ldots\!&\u_{n+1}^{6n}&\u_{n+1}^{6n+2}+\u_{n+1}^{6n+4}\\
\end{pmatrix}.\end{multline}
The determinant in~\eqref{combdet2} is equal to
\begin{multline}\label{combdet3}
\det\begin{pmatrix}
1&1+u_1^2+u_1^4&u_1^2+u_1^4+u_1^6&u_1^6+u_1^8+u_1^{10}&\!\ldots\!&u_1^{6n}+u_1^{6n+2}+u_1^{6n+4}\\
1&1+\u_1^2+\u_1^4&\u_1^2+\u_1^4+\u_1^6&\u_1^6+\u_1^8+\u_1^{10}&\!\ldots\!&\u_1^{6n}+\u_1^{6n+2}+\u_1^{6n+4}\\
\vdots&\vdots&\vdots&\vdots&\!\ddots\!&\vdots\\
1&1+u_n^2+u_n^4&u_n^2+u_n^4+u_n^6&u_n^6+u_n^8+u_n^{10}&\!\ldots\!&u_n^{6n}+u_n^{6n+2}+u_n^{6n+4}\\
1&1+\u_n^2+\u_n^4&\u_n^2+\u_n^4+\u_n^6&\u_n^6+\u_n^8+\u_n^{10}&\!\ldots\!&\u_n^{6n}+\u_n^{6n+2}+\u_n^{6n+4}\\
1&0&0&0&\!\ldots\!&0\\
1&1+\u_{n+1}^2+\u_{n+1}^4&\u_{n+1}^2+\u_{n+1}^4+\u_{n+1}^6&\u_{n+1}^6+\u_{n+1}^8+\u_{n+1}^{10}&\!\ldots\!&\u_{n+1}^{6n}+\u_{n+1}^{6n+2}+\u_{n+1}^{6n+4}\\
\end{pmatrix}\\
\shoveleft{=\prod_{i=1}^n(1+u_i^2+u_i^4)\;\prod_{i=1}^{n+1}(1+\u_i^2+\u_i^4)\;
\det\begin{pmatrix}
1&u_1^2&u_1^6&u_1^8&\!\ldots\!&u_1^{6n-4}&u_1^{6n}\\
1&\u_1^2&\u_1^6&\u_1^8&\!\ldots\!&\u_1^{6n-4}&\u_1^{6n}\\
\vdots&\vdots&\vdots&\vdots&\!\ddots\!&\vdots&\vdots\\
1&u_n^2&u_n^6&u_n^8&\!\ldots\!&u_n^{6n-4}&u_n^{6n}\\
1&\u_n^2&\u_n^6&\u_n^8&\!\ldots\!&\u_n^{6n-4}&\u_n^{6n}\\
1&\u_{n+1}^2&\u_{n+1}^6&\u_{n+1}^8&\!\ldots\!&\u_{n+1}^{6n-4}&\u_{n+1}^{6n}\\
\end{pmatrix}}\\
\shoveleft{=(-1)^n\,\prod_{i=1}^n(1+u_i^2+u_i^4)(u_i^2-\u_{n+1}^2)(\u_i^2-\u_{n+1}^2)(u_i^2-\u_i^2)\;\prod_{i=1}^{n+1}(1+\u_i^2+\u_i^4)}\\
\times\,\prod_{1\le i<j\le n}\!(u_i^2-u_j^2)(u_i^2-\u_j^2)(\u_i^2-u_j^2)(\u_i^2-\u_j^2)\;
s_{(n,n-1,n-1,\ldots,2,2,1,1)}(u_1^2,\u_1^2,\ldots,u_n^2,\u_n^2,\u_{n+1}^2),\end{multline}
where the first expression is obtained by adding column $i-1$ to
column $i$ in the matrix in~\eqref{combdet2}, for each $i=2,\ldots,2n+2$,
and the final expression is obtained by reversing the order of the columns of the matrix and applying~\eqref{Schurdet}.

Replacing the determinant in~\eqref{combdet2} by the final expression in~\eqref{combdet3}, and canceling certain products and signs,
now gives
\begin{equation}
Y(u_1,\ldots,u_{n+1})\big|_{q=e^{i\pi/6}}=\textstyle3^{-n(n-1)/2}\,\frac{u_{n+1}^n}{u_{n+1}+1}\,
s_{(n,n-1,n-1,\ldots,2,2,1,1)}(u_1^2,\u_1^2,\ldots,u_n^2,\u_n^2,\u_{n+1}^2),
\end{equation}
from which~\eqref{ZDASASMFullSchur} follows using~\eqref{DASASMFulldet1}.

Finally, note that an alternative proof of Theorem~\ref{ZDASASMFullSchurTh} would be to
take $X(u_1,\ldots,u_{n+1})$ as the RHS of~\eqref{ZDASASMFullSchur}, and show that properties (i)--(vi)
of Section~\ref{DASASMFulldetThPr}, or properties (i) and (iii)--(vi) of Section~\ref{DASASMFulldetThPr} and
(vii)-(x) of Section~\ref{DASASMFulldetThPrAlt}, with $q=e^{i\pi/6}$ in (iii), (ix) and (x), are then satisfied.

\section{Discussion}
In this final section, some further matters related to the main results of this paper are discussed.
\subsection{Factorization of the Schur Function in Corollary~\ref{ZDASASMSchurTh}}\label{Schurfact}
As shown in Corollary~\ref{ZDASASMSchurTh}, the odd-order DASASM partition function at $u_{n+1}=1$ and $q=e^{i\pi/6}$
is, up to a simple factor, given by the Schur function $s_{(n,n-1,n-1,\ldots,2,2,1,1)}(u_1^2,\u_1^2,\ldots,u_n^2,\u_n^2,1)$.

In~\cite{AyyBehFis17}, it will be shown that this function is a member of a collection of Schur functions which can be
factorized in terms of characters of irreducible representations of orthogonal and symplectic groups.
This collection also includes certain Schur functions indexed by rectangular shapes,
for which the factorizations were obtained by Ciucu and Krattenthaler~\cite[Thms.~3.1~\&~3.2]{CiuKra09}.

Using the same notation for
odd orthogonal and symplectic characters as that of Ciucu and Krattenthaler~\cite[Eqs.~(3.7)~\&~(3.9)]{CiuKra09},
the factorization of the Schur function in~\eqref{ZDASASMSchur} is explicitly
\begin{align}\notag&s_{(2n,2n-1,2n-1,\ldots,2,2,1,1)}(u_1,\u_1,\ldots,u_{2n},\u_{2n},1)=\\
\notag&\quad\textstyle\prod_{i=1}^{2n}(u_i+1+\u_i)\:
sp_{(n-1,n-1,\ldots,2,2,1,1)}(u_1,\ldots,u_{2n})\:so_{(n,n-1,n-1,\ldots,2,2,1,1)}(u_1,\ldots,u_{2n}),\\
\notag&s_{(2n+1,2n,2n,\ldots,2,2,1,1)}(u_1,\u_1,\ldots,u_{2n+1},\u_{2n+1},1)=\\
\label{DASASMSchurfact}&\quad\textstyle\prod_{i=1}^{2n+1}(u_i+1+\u_i)\:
sp_{(n,n-1,n-1,\ldots,2,2,1,1)}(u_1,\ldots,u_{2n+1})\:so_{(n,n,\ldots,2,2,1,1)}(u_1,\ldots,u_{2n+1}).
\end{align}
It can now be seen, which was not explicitly apparent in~\eqref{ZDASASMSchur}, that $\prod_{i=1}^n(u_i^2+1+\u_i^2)$
is a factor of $Z(u_1,\ldots,u_n,1)\big|_{q=e^{i\pi/6}}$.

\subsection{The odd-order DASASM partition function at special values of $q$}
As shown by, for example, Kuperberg~\cite{Kup02} and Okada~\cite{Oka06},
the partition functions of cases of the six-vertex model associated with symmetry classes of ASMs
often simplify if a particular parameter, which corresponds to the parameter $q$ in this paper,
is assigned to certain roots of unity.

For odd-order DASASMs, one such assignment is $q=e^{i\pi/6}$, which has been
considered in detail in Sections~\ref{MainResults} and~\ref{ZDASASMFullSchurPr},
and which is associated with straight enumeration.

Another such assignment is $q=e^{i\pi/4}$, for which
\begin{equation}\label{iPi4}
\bigl((-i\,\sigma(q^4))^{n^2}\,Z(u_1,\ldots,u_{n+1})\bigr)\big|_{q\rightarrow e^{i\pi/4}}=
\prod_{i=1}^n(u_i+\u_i)(u_iu_{n+1}+\u_i\u_{n+1})\prod_{1\le i<j\le n}(u_iu_j+\u_i\u_j)^2.\end{equation}
This result can be proved directly, as follows. First, observe that
multiplying $Z(u_1,\ldots,u_{n+1})$ by $(-i\,\sigma(q^4))^{n^2}$ is equivalent to renormalizing each of the~$n^2$ bulk weights
in each term of~\eqref{Z} by a factor of $-i\,\sigma(q^4)$.  Setting $q\rightarrow e^{i\pi/4}$, these renormalized weights are
$(-i\,\sigma(q^4)W(c,u))|_{q\rightarrow e^{i\pi/4}}=u+\u$,
for \rule[-3mm]{0mm}{7.3mm}$c=\Wi,\,\Wii,\,\Wiii,\,\Wiv$, and $(-i\,\sigma(q^4)W(c,u))|_{q\rightarrow e^{i\pi/4}}=0$, for $c=\Wv,\,\Wvi$.
It now follows from the bijection~\eqref{bij}, and the properties of odd DASASM triangles,
that the $C$-dependent term of~\eqref{Z} associated with~$\bigl((-i\,\sigma(q^4))^{n^2}\,Z(u_1,\ldots,u_{n+1})\bigr)\big|_{q\rightarrow e^{i\pi/4}}$ is zero
unless the odd DASASM triangle which corresponds to~$C$ consists of
all 0's, except for a single $1$ at either the start or end of each row. Observing that
a $0$ at the start or end of row $i$, for $i=1,\ldots,n$, is associated with weight $\sigma(qu_i)/\sigma(q)$ or $\sigma(q\u_i)/\sigma(q)$, respectively,
and that $1$'s are associated with weight~$1$, it can then be seen that the relevant sum over $2^n$ configurations
is $\prod_{i=1}^n\bigl(\sigma(qu_i)/\sigma(q)+\sigma(q\u_i)/\sigma(q)\bigr)(u_iu_{n+1}+\u_i\u_{n+1})\,\prod_{1\le i<j\le n}(u_iu_j+\u_i\u_j)^2$,
which gives~\eqref{iPi4}.  Note that, since the central entry of the relevant DASASMs in this case is always~1,
it also follows that $\bigl(\sigma(q^4)^{n^2}\,Z_+(u_1,\ldots,u_{n+1})\bigr)\big|_{q\rightarrow e^{i\pi/4}}=
\bigl(\sigma(q^4)^{n^2}\,Z(u_1,\ldots,u_{n+1})\bigr)\big|_{q\rightarrow e^{i\pi/4}}$ and
$\bigl(\sigma(q^4)^{n^2}\,Z_-(u_1,\ldots,u_{n+1})\bigr)\big|_{q\rightarrow e^{i\pi/4}}=0$.

We have also obtained some results and conjectures for the odd-order DASASM partition function at $q=e^{i\pi/3}$.
In this case, $W(c,1)|_{q=e^{i\pi/3}}=-1$,
for \rule[-3mm]{0mm}{7.3mm}$c=\Wi,\,\Wii,\,\Wiii,\,\Wiv$, and $W(c,1)|_{q=e^{i\pi/3}}=1$, for $c=\Wv,\,\Wvi$.
Hence, due to~\eqref{bij} and the second equation of~\eqref{comb},
$Z(1,\ldots,1)\big|_{q=e^{i\pi/3}}$ or $Z_\pm(1,\ldots,1)\big|_{q=e^{i\pi/3}}$
correspond to enumerations in which $A\in\DASASM(2n+1)$ is weighted by
$(-1)^{M(A)}$, where $M(A)$ is the number of $0$'s among the entries $A_{ij}$ for $i=1,\ldots,n$, $j=i+1,\ldots,2n+1-i$,
i.e., the number of $0$'s in the odd DASASM triangle associated with $A$ which are
not at the start or end of a row.  We conjecture that these weighted enumerations are given by
\begin{align}\notag\sum_{A\in\DASASM(2n+1)}(-1)^{M(A)}&=(-1)^{n(n-1)/2}\,V_n,\\
\sum_{A\in\DASASM_\pm(2n+1)}(-1)^{M(A)}&=\textstyle\frac{1}{2}(-1)^{n(n-1)/2}\:\bigl((-1)^n\mp3\bigr)\,V_n,\end{align}
where\rule[-1.6ex]{0ex}{3ex} 
$V_n=\frac{(2n)!\,\lfloor(3n-1)/2\rfloor!}{3^{\lfloor(n-1)/2\rfloor}\,(3n)!\,\lfloor(n-1)/2\rfloor!}\prod_{i=0}^n\frac{(3i)!}{(n+i)!}$.
As shown by Okada~\cite[Thm.~1.2 (A5) \&~(A6)]{Oka06}, $V_n$ is the number of $(2n+1)\times(2n+1)$ VHSASMs.

\subsection{HTSASMs and DASASMs}
Using the same argument as used for odd-order DASASMs in Section~\ref{DASASMs}, it can be seen that the central entry 
of an odd-order HTSASM must again be~$\pm1$. 
Denoting the sets of all $(2n+1)\times(2n+1)$ HTSASMs with fixed central entry~$-1$ and~$1$ as $\mathrm{HTSASM}_-(2n+1)$ 
and $\mathrm{HTSASM}_+(2n+1)$ respectively, it was shown by Razumov and Stroganov~\cite[p.~1197]{RazStr06a} that
\begin{equation}\label{HTSASMStroganovEq}\frac{|\mathrm{HTSASM}_-(2n+1)|}{|\mathrm{HTSASM}_+(2n+1)|}=\frac{n}{n+1}.\end{equation}
Therefore, combining~\eqref{StroganovEq} and~\eqref{HTSASMStroganovEq} gives
\begin{equation}\label{HTSASMDASASM}\frac{|\mathrm{HTSASM}_-(2n+1)|}{|\mathrm{HTSASM}_+(2n+1)|}=
\frac{|\mathrm{DASASM}_-(2n+1)|}{|\DASASM_+(2n+1)|}.\end{equation}
Since $\DASASM(2n+1)_\pm$ is a subset of $\mathrm{HTSASM}_\pm(2n+1)$,~\eqref{HTSASMDASASM} states that
the ratio between the numbers of $(2n+1)\times(2n+1)$ HTSASMs with central entry~$-1$ and~$1$
remains unchanged if the matrices are restricted to those which are also diagonally and antidiagonally symmetric.
It would be interesting to obtain a direct proof of this result, without necessarily showing that
either of the ratios is $n/(n+1)$.

\subsection{Further work}
Some directions in which we are undertaking further work closely
related to that of this paper are as follows.

First, as discussed in Section~\ref{LocRel}, the boundary weights used for odd-order DASASMs are
a special case of the most general boundary weights which satisfy the reflection equation
for the six-vertex model.  By using other boundary weights in the odd-order DASASM partition function,
properties analogous to those of Sections~\ref{GenProp} and~\ref{Spec} are again satisfied, and
it is possible to obtain results, including enumeration formulae, for certain subclasses of odd-order DASASMs.
Work on this has been reported in~\cite{AyyBehFis16}.

Second, it is straightforward to define cases of the six-vertex model which are similar
to the case considered in this paper, and whose configurations are in bijection with
DSASMs or even-order DASASMs.  The underlying graphs for these cases have already been introduced by
Kuperberg~\cite[Figs.~12 \&~13]{Kup02}, in order to study OSASMs
and OOSASMs.  (More precisely, these are graphs for $2n\times2n$ DSASMs and $4n\times4n$ 
DASASMs, but the generalizations to DSASMs of any order and DASASMs of any
even order are trivial.)  Replacing the boundary weights which were used for OSASMs and OOSASMs
by those used in Table~\ref{weights} for odd-order DASASMs, leads to partition functions which give
the numbers of DSASMs or even-order DASASMs, when $q=e^{i\pi/6}$ and the other parameters are all~1.  
Alternatively, using yet further boundary weights leads to 
partition functions associated with subclasses of DSASMs and even-order DASASMs
other than OSASMs and OOSASMs. 
All of these partition functions satisfy properties 
analogous to those identified for the odd-order DASASM partition function in Sections~\ref{GenProp} and~\ref{Spec},
which enables various results to be obtained for these cases.  Work on this will be reported in~\cite{BehFis16a,BehFis16b}.

\section*{Acknowledgements}
We thank Arvind Ayyer for very helpful discussions.
The first two authors acknowledge hospitality and support from the Galileo Galilei Institute, Florence, Italy, during the programme
``Statistical Mechanics, Integrability and Combinatorics'' held in May--June 2015.
The second author acknowledges support from the Austrian Science Foundation FWF, START grant Y463.
The third author acknowledges support from Research Programs L1-069 and Z1-5434 of the Slovenian Research Agency.

\let\oldurl\url
\makeatletter
\renewcommand*\url{%
        \begingroup
        \let\do\@makeother
        \dospecials
        \catcode`{1
        \catcode`}2
        \catcode`\ 10
        \url@aux
}
\newcommand*\url@aux[1]{%
        \setbox0\hbox{\oldurl{#1}}%
        \ifdim\wd0>\linewidth
                \strut
                \\
                \vbox{%
                        \hsize=\linewidth
                        \kern-\lineskip
                        \raggedright
                        \strut\oldurl{#1}%
                }%
        \else
                \hskip0pt plus\linewidth
                \penalty0
                \box0
        \fi
        \endgroup
}
\makeatother
\gdef\MRshorten#1 #2MRend{#1}
\gdef\MRfirsttwo#1#2{\if#1M
MR\else MR#1#2\fi}
\def\MRfix#1{\MRshorten\MRfirsttwo#1 MRend}
\renewcommand\MR[1]{\relax\ifhmode\unskip\spacefactor3000 \space\fi
  \MRhref{\MRfix{#1}}{{\tiny \MRfix{#1}}}}
\renewcommand{\MRhref}[2]{
 \href{http://www.ams.org/mathscinet-getitem?mr=#1}{#2}}

\bibliography{Bibliography}
\bibliographystyle{amsplainhyper}
\end{document}